\date{} % % March 17 2003
\title{Conformal weldings of random surfaces:\\ SLE and the quantum gravity zipper}
\author{Scott Sheffield}
\documentclass[12pt]{article}
\usepackage{amsmath}
\usepackage{amssymb}
\usepackage{amsthm}
\usepackage{amsfonts}
\usepackage{graphicx}
\usepackage{tabularx}
\usepackage[normalem]{ulem}
\usepackage{microtype}
\input epsf.tex

\allowdisplaybreaks

\newif\ifhyper\IfFileExists{hyperref.sty}{\hypertrue}{\hyperfalse}

\ifhyper\usepackage{hyperref}\fi

\newif\ifdraft
\drafttrue
\def\note#1/{\ifdraft {\bf [#1]}\fi}
\long\def\comment#1{}
\numberwithin{equation}{section}
\numberwithin{figure}{section}

\newtheorem{theorem}{Theorem}
\numberwithin{theorem}{section}
\newtheorem{corollary}[theorem]{Corollary}
\newtheorem{lemma}[theorem]{Lemma}
\newtheorem{proposition}[theorem]{Proposition}

\theoremstyle{remark}
\theoremstyle{remark}\newtheorem{remark}[theorem]{Remark}
\def\eref#1{(\ref{#1})}

\newcommand{\R}{\mathbb{R}}
\newcommand{\C}{\mathbb{C}}
\newcommand{\D}{\mathbb{D}}
\newcommand{\Z}{\mathbb{Z}}

\newcommand{\h}{\mathfrak{h}}
\newcommand{\density}{\rho}
\newcommand{\zcap}{\mathcal Z^{\mathrm{CAP}}}
\newcommand{\zlength}{\mathcal Z^{\mathrm{LEN}}}

\def\Harm{{\mathrm{Harm}}}
\def\Supp{{\mathrm{Supp}}}
\def\H{\mathbb{H}}

\def\Cov{\mathop{\mathrm{Cov}}}
\def\Var{\mathop{\mathrm{Var}}}

\def\even{\mathrm{E}}
\def\odd{\mathrm{O}}

\def\Im{{\rm Im}\,}
\def\Re{{\rm Re}\,}

\def\SLEkk#1/{$\mathrm{SLE}(#1)$}
\def\SLEr#1/{$\mathrm{SLE(\kappa;#1)}$}
\def\SLEkr#1;#2/{$\mathrm{SLE(#1;#2)}$}
\def\SLEkrkr/{$\mathrm{SLE(\kappa;\rho)}$}
\def\SLEk/{\SLEkk{\kappa}/}
\def\SLEtwo/{\SLEkk2/}
\def\SLE/{$\mathrm{SLE}$}

\def\CLEkk#1/{$\mathrm{CLE}(#1)$}
\def\CLEk/{\CLEkk{\kappa}/}
\def\CLE/{$\mathrm{CLE}$}

\def\Ito/{It\^o}
\def \eps {\varepsilon}

\def\proofof#1{{ \medbreak \noindent {\bf Proof of #1.} }}

\def\flowset{\mathbb F}

\def\density{{\rho}}

\begin{document} \maketitle \begin{abstract}
We construct a conformal welding of two Liouville quantum gravity random surfaces and show that the interface between them is a random fractal curve called the Schramm-Loewner evolution (SLE), thereby resolving a variant of a conjecture of Peter Jones.  We also demonstrate some surprising symmetries of this construction,
which are consistent with the belief that (path-decorated) random planar maps have (SLE-decorated) Liouville quantum gravity as a scaling limit.
We present several precise conjectures and open questions.
\end{abstract}

\newpage
\tableofcontents
\newpage

\medbreak {\noindent\bf Acknowledgments.}  For valuable insights, helpful conversations, and comments on early drafts, the author thanks Julien Dub\'edat, Bertrand Duplantier, Richard Kenyon, Greg Lawler, Jason Miller, Curtis McMullen, Steffen Rohde, Oded
Schramm, Stanislav Smirnov, Fredrik Johansson Viklund, Wendelin Werner, and David Wilson.  The author also thanks the organizers and participants of the PIMS 2010 Summer School in Probability (Seattle) and the Clay Mathematics Institute 2010 Summer School (B\'uzios, Brazil), where drafts of this paper were used as lecture notes for courses. We also thank Nathana\"el Berestycki, Ewain Gwynne, Xin Sun, and a reading group at the Newton Institute in Cambridge for helpful comments. The author was partially supported by a grant from the Simons Foundation while at the Newton Institute, and also by DMS-1209044, a fellowship from the Simons Foundation, and EPSRC grants {EP/L018896/1} and {EP/I03372X/1}.

\section{Introduction} \label{s.intro}
\subsection{Overview} \label{ss.overview}
Liouville quantum gravity and the Schramm-Loewner evolution (SLE) rank among the great mathematical physics discoveries of the last few decades.  Liouville quantum gravity, introduced in the physics literature by Polyakov in 1981 in the context of string theory, is a canonical model of a random two dimensional Riemannian manifold \cite{MR623209, MR623210}.  The Schramm-Loewner evolution, introduced by Schramm in 1999, is a canonical model of a random path in the plane that doesn't cross itself \cite{MR1776084, MR2153402}.  Each of these models is the subject of a large and active literature spanning physics and mathematics.

Our goal here is to connect these two objects to each other in the simplest possible way.  Roughly speaking, we will show that if one glues together two independent Liouville quantum gravity random surfaces along boundary segments (in a boundary-length-preserving way) --- and then conformally maps the resulting surface to a planar domain --- then the interface between the two surfaces is an SLE.

Peter Jones conjectured several years ago that SLE could be obtained in a similar way --- specifically, by gluing (what in our language amounts to) one Liouville quantum gravity random surface and one deterministic Euclidean disc.  Astala, Jones, Kupiainen, and Saksman showed that the construction Jones proposed produces a well-defined curve \cite{2009arXiv0909.1003A, MR2600118}, but Binder and Smirnov recently announced a proof (involving multifractal exponents) that this curve is not a form of SLE, and hence the original Jones conjecture is false \cite{privatesmirnov} (see Section \ref{ss.capstationary}).  Our construction shows that a simple variant of the Jones conjecture is in fact true.

Beyond this, we discover some surprising symmetries.  For example, it turns out that there is one particularly natural random simply connected surface (called a {\em $\gamma$-quantum wedge}) that has an infinite-length boundary isometric to $\R$ (almost surely) which contains a distinguished ``origin.''   Although this surface is simply connected, it is almost surely highly non-smooth and it has a random fractal structure.  We will explain precisely how it is defined in Section \ref{ss.lengthstationary}.  The origin divides the boundary into two infinite-length boundary arcs.  Suppose we glue (in a boundary-length preserving way) the right arc of one such surface to the left arc of an independent random surface with the same law, then conformally map the combined surface to the complex upper half plane $\H$ (sending the origin to the origin and $\infty$ to $\infty$ --- see figure below), and then {\em erase} the boundary interface.  The geometric structure of the combined surface can be pushed forward to give geometric structure (including an area measure) on $\H$.  It is natural to wonder how well one can guess, from this geometric structure on $\H$, where the now-erased interface used to be.

\begin {figure}[htbp]
\begin {center}  
\includegraphics [width=5in]{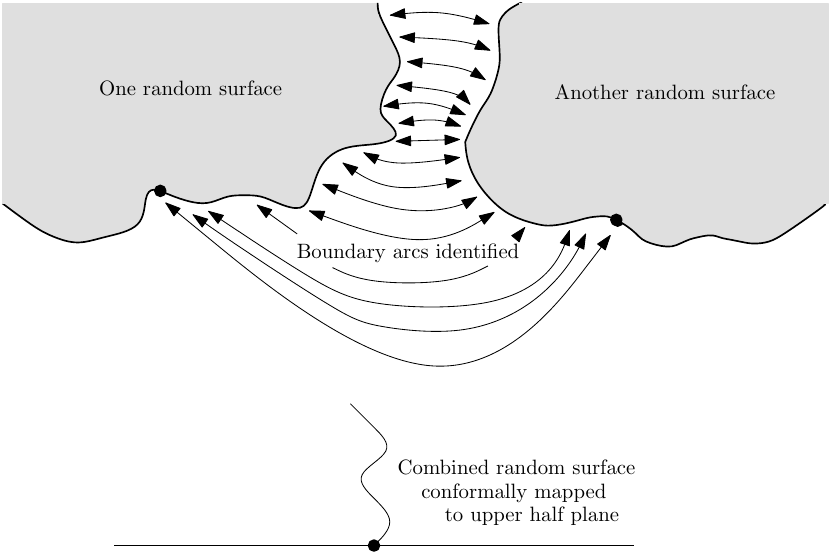}
%\caption {\label{qmapfig} A quantum surface coordinate change.}
\end {center}
\end {figure}

We will show that the geometric structure yields no information at all.  That is, the conditional law of the interface is that of an SLE in $\H$ {\em independently} of the underlying geometry (a fact formally stated as part of Theorem \ref{t.lengthstationary}).  Another way to put this is that conditioned on the combined surface, all of the information about the interface is contained in the {\em conformal structure} of the combined surface, which determines the embedding in $\H$ (up to rescaling $\H$ via multiplication by a positive constant, which does not affect the law of the path, since the law of SLE is scale-invariant).

This apparent coincidence is actually quite natural from one point of view.  We recall that one reason (among many) for studying SLE is that it arises as the fine mesh ``scaling limit'' of random simple paths on lattices.  Liouville quantum gravity is similarly believed (though not proved) to be the scaling limit of random discretized surfaces and random planar maps.  The independence mentioned above turns out to be consistent with (indeed, at least heuristically, a {\em consequence of}) certain scaling limit conjectures (and a related conformal invariance Ansatz) that we will formulate precisely (in Section \ref{ss.discretelimitdiscussion}) for the first time here.

Polyakov initially proposed Liouville quantum gravity as a model for the intrinsic Riemannian manifold parameterizing the space-time trajectory of a string \cite{MR623209}.  From this point of view, the welding/subdivision of such surfaces is analogous to the concatenation/subdivision of one-dimensional time intervals (which parameterize point-particle trajectories).
It seems natural to try to understand complicated string trajectories by decomposing them into simpler pieces (and/or gluing pieces together), which should involve subdividing and/or welding the corresponding Liouville quantum gravity surfaces.  The purpose of this paper is to study these weldings and subdivisions mathematically.  We will not further explore the physical implications here.

In a recent memoir \cite{2008arXiv0812.0183}, Polyakov writes that he first became convinced of the connection between the discrete models and Liouville quantum gravity in the 1980's after jointly deriving, with Knizhnik and Zamolodchikov, the so-called {\em KPZ formula} for certain Liouville quantum gravity scaling dimensions and comparing them with known combinatorial results for the discrete models \cite{MR947880}.  With Duplantier, the present author recently formulated and proved the KPZ formula in a mathematical way \cite{2008arXiv0808.1560D, MR2501276} (see also \cite{MR2506765, 2008arXiv0807.1036R}).  This paper is in some sense a sequel to \cite{2008arXiv0808.1560D}, and we refer the reader there for references and history.

We will find it instructive to develop Liouville quantum gravity along with a closely related construction called the {\em AC geometry} or {\em imaginary geometry}.  Both Liouville quantum gravity and the imaginary geometry are based on a simple object called the Gaussian free field.

\subsection{Random geometries from the Gaussian free field} \label{ss.randomgeometryGFF}
The two dimensional Gaussian free field (GFF) is a natural higher dimensional analog of Brownian motion that plays a prominent role in mathematics and physics.  See the survey \cite{Sh} and the introductions of \cite{MR2486487, SchrammSheffieldGFF2} for a detailed account.  On a planar domain $D$, one can define both a zero boundary GFF and a free boundary GFF (the latter being defined only modulo an additive constant, which we will sometimes fix arbitrarily).  In both cases, an instance of the GFF is a random sum $$h = \sum_i \alpha_i f_i,$$ where the $\alpha_i$ are i.i.d.\ mean-zero unit-variance normal random variables, and the $f_i$ are an orthonormal basis for a Hilbert space of real-valued functions on $D$ (or in the free boundary case, functions modulo additive constants) endowed with the Dirichlet inner product
$$(f_1,f_2)_\nabla := (2\pi)^{-1} \int_D \nabla f_1(z) \cdot \nabla f_2(z) dz.$$
The Hilbert space is the completion of either the space of smooth compactly supported functions $f: D \to \R$ (zero boundary) or the space of all smooth functions $f:D \to \R$ modulo additive constants with $(f,f)_\nabla < \infty$ (free boundary).  In each case, $h$ is understood not as a random function on $D$ but as a random distribution or generalized function on $D$.  (Mean values of $h$ on certain sets are also defined, but the value of $h$ at a particular point is not defined.)  One can fix the additive constant for the free boundary GFF in various ways, e.g., by requiring the mean value of $h$ on some set to be zero.  We will review these definitions in Section \ref{GFFoverview}.

There are two natural ways to produce a ``random geometry'' from the Gaussian free field.  The first construction is (critical) {\bf Liouville quantum gravity}.  Here, one replaces the usual Lebesgue measure $dz$ on a smooth domain $D$ with a random measure $\mu_h = e^{\gamma h(z)}dz$, where $\gamma \in [0,2)$ is a fixed constant and $h$ is an instance of (for now) the free boundary GFF on $D$ (with an additive constant somehow fixed --- we will actually consider various ways of fixing the additive constant later in the paper; one way is to require the mean value of $h$ on some fixed set to be $0$).
Since $h$ is not defined as a function on $D$, one has to use a regularization procedure to be precise: \begin{equation} \label{e.mudef} \mu = \mu_h := \lim_{\eps \to 0} \eps^{\gamma^2/2} e^{\gamma h_\eps(z)}dz,\end{equation}
where $dz$ is Lebesgue measure on $D$, $h_\eps(z)$ is the mean value of $h$ on the circle $\partial B_\eps(z)$ and the limit represents weak convergence (on compact subsets) in the space of measures on $D$.  (The limit exists almost surely, at least if $\eps$ is restricted to powers of two \cite{2008arXiv0808.1560D}.)  We interpret $\mu_h$ as the area measure of a random surface conformally parameterized by $D$.  When $x \in \partial D$, we let $h_\eps(x)$ be the mean value of $h$ on $D \cap \partial B_\eps(x)$.  On a linear segment of $\partial D$, we may define a boundary length measure by \begin{equation} \label{e.nudef}
\nu = \nu_h := \lim_{\eps \to 0} \eps^{\gamma^2/4}e^{\gamma h_\eps(x)/2}dx,\end{equation} where $dx$ is Lebesgue measure on $\partial D$.  (For details see \cite{2008arXiv0808.1560D}, which also relates the above random measures to the curvature-based action used to define Liouville quantum gravity in the physics literature.)

We could also parameterize the same surface with a different domain $\widetilde D$, and our regularization procedure implies a simple rule for changing coordinates.  Suppose that $\psi$ is a conformal map from a domain $\widetilde D$ to $D$ and write $\widetilde h$ for the
distribution on $\widetilde D$ given by $h \circ \psi + Q \log |\psi'|$
where $$Q := \frac{2}{\gamma} + \frac{\gamma}{2},$$ as in Figure \ref{qmapfig} \footnote{\label{foot::distdef} We use the same distribution composition notation as \cite{2008arXiv0808.1560D}: i.e., If $\phi$ is a conformal map from $D$ to a domain $\tilde D$ and $h$
is a distribution on $D$, then we define the pullback $h \circ
\phi^{-1}$ of $h$ to be a distribution on $\tilde D$ defined by $(h
\circ \phi^{-1}, \tilde \density) = (h, \density)$ whenever
$\density \in H_s(D)$ and $\tilde \density = |\phi'|^{-2} \density
\circ \phi^{-1}$.  (Here $\phi'$ is the complex derivative of
$\phi$, and $(h,\density)$ is the value of the distribution $h$
integrated against $\density$.) Note that if $h$ is a continuous
function (viewed as a distribution via the map $\density \to
\int_{D} \density(z) h(z) dz$), then the distribution $h \circ
\phi^{-1}$ thus defined is the ordinary composition of $h$ and
$\phi^{-1}$ (viewed as a distribution).
}.  Then $\mu_h$ is
almost surely the image under $\psi$ of the measure $\mu_{\widetilde h}$. That is, $\mu_{\widetilde h}(A) = \mu_h(\psi(A))$ for $A \subset \widetilde D$.  Similarly, $\nu_h$ is almost surely the image under $\psi$ of the measure $\nu_{\widetilde h}$ \cite{2008arXiv0808.1560D}.  In fact, \cite{2008arXiv0808.1560D} formally defines a {\bf quantum surface} to be an equivalence class of pairs $(D,h)$ under the equivalence transformations (see Figure \ref{qmapfig}) \begin{equation} \label{Qmap} (D,h) \to \psi^{-1}(D,h) := (\psi^{-1}(D), h \circ \psi + Q \log |\psi'|) = (\widetilde D, \widetilde h), \end{equation}
noting that both area and boundary length are well defined for such surfaces.  The invariance of $\nu_h$ under \eqref{Qmap} actually yields a definition of the quantum boundary length measure $\nu_h$ when the boundary of $D$ is not piecewise linear---i.e., in this case, one simply maps to the upper half plane (or any other domain with a piecewise linear boundary) and computes the length there.\footnote{It remains an open question whether the {\em interior} of a quantum surface is canonically a metric space.  A pair $(D,h)$ is a metric space parameterized by $D$ when, for distinct $x,y \in D$ and $\delta >0$, one defines the distance $d_\delta(x,y)$ to be the smallest number of Euclidean balls in $D$ of $\mu_h$ mass $\delta$ required to cover {\em some} continuous path from $x$ to $y$ in $D$.  We conjecture but cannot prove that for some constant $\beta$ the limiting metric $$\lim_{\delta \to 0} \delta^\beta d_\delta$$ exists a.s.\ and is invariant under the transformations described by \eqref{Qmap}.}

\begin {figure}[htbp]
\begin {center}
\includegraphics [width=3in]{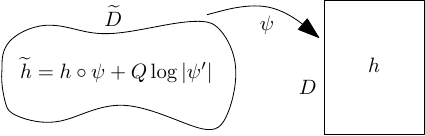}
\caption {\label{qmapfig} A quantum surface coordinate change.}
\end {center}
\end {figure}

The second construction involves ``flow lines'' of the unit vector field $e^{ih/\chi}$ where $\chi \not = 0$ is a fixed constant (see Figure \ref{flowfigure}), or alternatively flow lines of $e^{i(h/\chi + c)}$ for a constant $c \in [0,2\pi)$.
The author has proposed calling this collection of flow lines the {\bf AC geometry}\footnote{AC stands for ``altimeter-compass.'' If the graph of
$h$ is viewed as a mountainous terrain, then a hiker holding an analog
altimeter---with a needle indicating altitude modulo $2\pi \chi$---in one
hand and a compass in the other can trace an AC ray by walking at
constant speed (continuously changing direction as
necessary) in such a way that the two needles always point in the
same direction.} of $h$, but a recent series of joint works with Jason Miller uses the term {\bf imaginary geometry} \cite{ms2012imag1,ms2012imag2,ms2012imag3,ms2013imag4}. Makarov once proposed the term ``magnetic gravity'' in a lecture, suggesting that in some sense the AC geometry is to Liouville quantum gravity as electromagnetism is to electrostatics.  We will discuss additional interpretations in Section \ref{s.interpretation} and the appendix.

Although $h$ is a distribution and not a function, one can make sense of flow lines using
the couplings between the Schramm-Loewner evolution (SLE) and the GFF in \cite{S05, SchrammSheffieldGFF2}, which were further developed in \cite{MR2525778} and more recently in \cite{2009arXiv0909.5377M, 2010JSP...140....1H, 2010arXiv1006.1853I}.  The paths in these couplings are generalizations of the GFF contour lines of \cite{SchrammSheffieldGFF2}.

We define an {\bf AC surface} to be an equivalence class of pairs under the following variant of \eqref{Qmap}:
\begin{equation} \label{chimap} (D,h) \to (\psi^{-1}(D), h \circ \psi - \chi \arg \psi') = (\widetilde D, \widetilde h),\end{equation}
as in Figure \ref{acmapfig}.
%(In the AC geometry context, we could denote the RHS of \eqref{chimap} by $\psi^{-1}(D,h)$, but we instead reserve
%that notation for the RHS of \eqref{Qmap}.)
The reader may observe that (at least when $h$ is smooth) the flow lines of the LHS of \eqref{chimap} are the $\psi$ images of the flow lines of the RHS.  To check this, first consider the simplest case: if $\psi^{-1}$ is a rotation (i.e., multiplication by a modulus-one complex number), then \eqref{chimap} ensures that the unit flow vectors $e^{ih/\chi}$ (as in Figure \ref{flowfigure}) are rotated by the same amount that $D$ is rotated.  The general claim follows from this, since every conformal map looks locally like the composition of a dilation and a rotation (see Section \ref{windingintro}).

\begin {figure}[htbp]
\begin {center}
\includegraphics [width=3in]{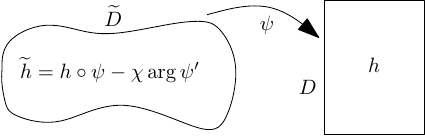}
\caption {\label{acmapfig} An AC surface coordinate change.}
\end {center}
\end {figure}

Recalling the conformal invariance of the GFF, if the $h$ on the left side of \eqref{Qmap} and \eqref{chimap} is a centered (expectation zero) Gaussian free field on $D$ then the distribution on the right hand side is a centered (expectation zero) GFF on $\widetilde D$ {\em plus} a deterministic function.  In other words, {\em changing the domain of definition is equivalent to recentering the GFF}.  The deterministic function is harmonic if $D$ is a planar domain, but it can also be defined (as a non-harmonic function) when $D$ is a surface with curvature (see \cite{2008arXiv0808.1560D}).  In what follows, we will often find it convenient to define quantum and AC surfaces on the complex half plane $\H$ using a (free or zero boundary) GFF on $\H$, sometimes recentered by the addition of a deterministic function that we will call $\h_0$.  We will state our main results in the introduction for fairly specific choices of $\h_0$.  We will extend these results to more general underlying geometries in Section \ref{couplingsection} and Section \ref{s.zipper}.

\subsection{Theorem statements: SLE/GFF couplings} \label{ss.couplingsubsecton}

We will give explicit relationships between the Gaussian free field and both ``forward'' and ``reverse'' forms of SLE in Theorems \ref{forwardcoupling} and \ref{reversecoupling} below.  We will subsequently interpret these theorems as statements about AC geometry and Liouville quantum gravity, respectively.
We will prove Theorem \ref{forwardcoupling} in Section \ref{subsec::couplingproofs} using a series of calculations.  These calculations are not really new to this paper, although the precise form of the argument we give has not been published elsewhere \footnote{The argument presented in Section \ref{subsec::couplingproofs}, together with the relevant calculations, appeared in lecture slides some time ago \cite{S05}, and is by now reasonably well known.  Dub\'edat presented another short derivation of this statement within a long foundational paper \cite{MR2525778}. More recent variants appear in \cite{2010JSP...140....1H, 2010arXiv1006.1853I}, and in the series of papers  \cite{ms2012imag1,ms2012imag2,ms2012imag3,ms2013imag4}, which studies the couplings in further detail.  Prior to these works, Kenyon and Schramm derived (but never published) a calculation relating SLE to the GFF in the case $\kappa = 8$.  One could also have inferred the existence of such a relationship from the fact --- due to Lawler, Schramm, and Werner --- that SLE$_8$ is a continuum scaling limit of uniform spanning tree boundaries \cite{MR2044671}, and the fact --- due to Kenyon --- that the winding number ``height functions'' of uniform spanning trees have the GFF as a scaling limit \cite{MR1782431, MR1872739, MR2415464}.}  Our main reason for proving Theorem \ref{forwardcoupling} in Section \ref{subsec::couplingproofs} is that we wish to simultaneously prove Theorem \ref{reversecoupling}.  Theorem \ref{reversecoupling} is completely new to this paper (and essential to the other results obtained in this paper), but it is very closely related to Theorem \ref{forwardcoupling}.  Proving the two results side by side allows us to highlight the similarities and differences.

We will not give a detailed introduction to SLE here, but there are many excellent surveys on SLE; see, e.g., the introductory references \cite{2003math......3354W, MR2457071} for basic SLE background.  To set notation, we recall that when $\eta$ is an instance of chordal SLE$_\kappa$ in $\H$ from $0$ to $\infty$, the conformal maps $g_t: \H \setminus \eta([0,t]) \to \H$, normalized so that $\lim_{z \to \infty} |g_t(z) - z| = 0$, satisfy \begin{equation} \label{e.firstloewner} d g_t(z) = \frac{2}{g_t(z) - W_t}dt,\end{equation} with $W_t = \sqrt{\kappa}B_t = g_t(\eta(t))$, where $B_t$ is a standard Brownian motion.  In fact, this can be taken as the definition of SLE$_\kappa$.  Rohde and Schramm proved in \cite{MR2153402} that for each $\kappa$ and instance of $B_t$, there is almost surely a unique continuous curve $\eta$ in $\H$ from $0$ to $\infty$, parameterized by $[0,\infty)$, for which \eqref{e.firstloewner} holds for all $t$.  When $\eta$ is parameterized so that \eqref{e.firstloewner} holds, the quantity $t$ is called the (half-plane) {\em capacity} of $\gamma([0,t])$.  The curve $\eta$ is almost surely a simple curve when $\kappa \in [0,4]$, a self-intersecting but non-space-filling curve when $\kappa \in (4,8)$, and a space-filling curve (ultimately hitting every point in $\H$) when $\kappa \geq 8$ \cite{MR2153402}.

The maps $$f_t(z) := g_t(z) - W_t$$ satisfy $$df_t(z) = \frac{2}{f_t(z)}dt - \sqrt{\kappa}dB_t,$$ and $f_t(\eta(t)) = 0$.  Throughout this paper, we will use $f_t$ rather than $g_t$ to describe the Loewner flow.  If $\eta_T = \eta([0,T])$ is a segment of an SLE trace, denote by $K_T$ the complement of the unbounded component of $\H \setminus \eta_T$. In the statements of Theorem \ref{forwardcoupling} and Theorem \ref{reversecoupling} below and throughout the paper, we will discuss several kinds of random distributions on $\H$. To show that these objects are well defined as distributions on $\H$, we will make implicit use of some basic facts about distributions:
\begin{enumerate}
\item If $h$ is a distribution on a domain $D$ then its restriction to a subdomain is a distribution on that subdomain. (This follows by simply restricting the class of test functions to those supported on the subdomain.)
\item If $h$ a distribution on a domain $D$ and $\phi$ is a conformal map from $D$ to a domain $\tilde D$ then $h \circ \phi^{-1}$ is a distribution on $\tilde D$. (Recall Footnote~\ref{foot::distdef}.)
\item An instance of the zero boundary GFF on a subdomain of $D$ is also well defined as a distribution on all of $D$. (See Section~2.1 of \cite{SchrammSheffieldGFF2}.)
\item If $h$ is an $L^1$ function on $D$, then $h$ can be understood as a distribution on $D$ defined by $(h, \density) = \int_{D} \density(z) h(z) dz$.
\end{enumerate}

In the proof of Theorem~\ref{forwardcoupling} in Section \ref{subsec::couplingproofs}, we will show that even though the function $\arg f_t'$ that appears in the theorem statement is a.s.\ unbounded,  it can also a.s.\ be understood as a distribution on $\H$ (see the discussion after the theorem statement below).

\begin{theorem} \label{forwardcoupling}
Fix $\kappa \in (0,4]$ and let $\eta_T$ be the segment of SLE$_\kappa$ generated by the Loewner flow
\begin{equation}\label{e.floewner} d f_t(z)  =  \frac{2}{f_t(z)}dt - \sqrt{\kappa}dB_t,\,\,\,\,\,\,\,\,f_0(z) = z\end{equation}
up to a fixed time $T>0$.
Write \begin{eqnarray*}  \h_0(z) &:=& \frac{-2}{\sqrt{\kappa}}  \arg z, \,\,\,\,\,\,\,\,\,\,\,\,\,\,\,\quad\quad\quad\quad \chi := \frac{2}{\sqrt{\kappa}} - \frac{\sqrt{\kappa}}{2},\\
\h_t(z) & := &
\h_0(f_t(z)) - \chi \arg f_t'(z).
\end{eqnarray*}
Here $\arg(f_t(z))$ (which is a priori defined only up to an additive multiple of $2\pi$) is chosen to belong $(0, \pi)$ when $f_t(z) \in \H$; we similarly define $\arg f_t'(z)$ by requiring that (when $t$ is fixed) it is continuous on $\H \setminus \eta_T$ and tends to $0$ at $\infty$.  Let $\widetilde h$ be an instance of the
zero boundary GFF on $\H$, independent of $B_t$.  Then the following two random distributions on $\H$ agree in law:\footnote{Note that $f_T$ maps $\H \setminus K_T$ to $\H$, so $(\H, h)$ and $(\H \setminus K_T, h \circ f_T - \chi \arg f_T')$ describe equivalent AC surfaces by \eqref{chimap}.}
\begin{eqnarray*}
 h & := &  \h_0 + \widetilde h. \\
 h \circ f_T - \chi \arg f_T'  & = &  \h_T + \widetilde h \circ f_T.
 \end{eqnarray*}
   The two distributions above also agree in law when $\kappa \in (4,8)$ if we replace $\widetilde h \circ f_T$ with a GFF on $\H \setminus \eta([0,t])$ (which in this case means the sum of an independent zero boundary GFF on each component of $\H \setminus \eta([0,t])$) and take $\h_t(z) := \lim_{s \to \tau(z)_-} \h_s(z)$ if $z$ is absorbed at time $\tau(z) \leq t$.
\end{theorem}

\begin{noindent}{\bf Alternative statement of Theorem \ref{forwardcoupling}:} Using our coordinate change and AC surface definitions, we may state the theorem when $\kappa < 4$ somewhat more elegantly as follows: the law of the AC surface $(\H,h)$ is invariant under the operation of independently sampling $f_T$ using a Brownian motion and \eqref{e.floewner}, transforming the AC surface via the coordinate change $f_T^{-1}$ (going from right to left in Figure \ref{forwardftfig}\ \footnote{All figures in this paper are sketches, not representative simulations.} --- see also Figure \ref{forwardftarrowfig}) in the manner of \eqref{chimap}, and erasing the path $\eta_T$ (to obtain an AC surface parameterized by $\H$ instead of $\H \setminus \eta_T$).  We discuss the geometric intuition behind the alternative statement in Section \ref{windingintro}.
\end{noindent}
\vspace{.1in}

\begin {figure}[h]
\begin {center}
\includegraphics [width=4in]{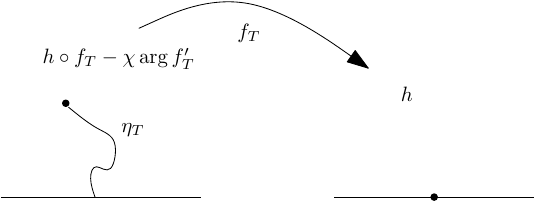}
\caption {\label{forwardftfig} Forward coupling.}
\end {center}
\end {figure}

Note that, as a function, $\h_T$ is not defined on $\eta_T$ itself. 
% However, since $\eta_T$ a.s.\ has Lebesgue measure zero, one might expect that we can define $\h_T$ arbitrarily on $\eta_T$ without affecting the integral of $\h_T$ times a smooth function.  
However, we will see in Section \ref{subsec::couplingproofs} that $\h_T$ is a.s.\ well defined as a distribution, independently of how we define it as a function on $\eta_T$ itself.  This will follow from the fact that, when $\kappa = 4$, this $\h_T$ is almost surely a bounded function off of $\eta_T$, and when $\kappa \not = 4$, the restriction of $\h_T$ to any compact subset of $\H$ is almost surely in $L^p$ for each $p < \infty$ (see Section \ref{couplingsection}).  The fact that $\widetilde h \circ f_T$ is well defined as a distribution on $\H$ (not just as a distribution on $\H \setminus \eta_T$) follows from conformal invariance of the GFF, and the fact (mentioned above, proved in \cite{SchrammSheffieldGFF2}) that a zero boundary GFF instance on a subdomain can be understood as a distribution on the larger domain.

Another standard approach for generating a segment $\eta_T$ of an SLE curve is via the reverse Loewner flow, whose definition is recalled in the statement of the following theorem. (Note that if $T$ is a fixed constant, then the law of the $\eta_T$ generated by reverse Loewner evolution is the same as that generated by forward Loewner evolution; see Figures \ref{forwardftfig} and \ref{reverseftfig}.)

\begin{theorem} \label{reversecoupling}
Fix $\kappa > 0$ and let $\eta_T$ be the segment of SLE$_\kappa$ generated by a reverse Loewner flow
\begin{equation} \label{e.rloewner} d f_t(z)  =  \frac{-2}{f_t(z)}dt - \sqrt{\kappa}dB_t,\,\,\,\,\,\,\,\,f_0(z) = z\end{equation}
up to a fixed time $T>0$.
Write \begin{eqnarray*}  \h_0(z)  &:=&  \frac{2}{\sqrt{\kappa}}  \log|z|, \,\,\,\,\,\,\,\,\,\,\,\,\,\,\,\quad\quad\quad\quad Q :=  \frac{2}{\sqrt{\kappa}} + \frac{\sqrt{\kappa}}{2}, \\
\h_t(z) &:=&  \h_0(f_t(z)) + Q \log |f_t'(z)|,
\end{eqnarray*}
and let $\widetilde h$ be an instance of the
free boundary GFF on $\H$, independent of $B_t$.  Then the following two random distributions (modulo additive constants) on $\H$ agree in law:\footnote{Note that $f_T$ maps $\H$ to $\H \setminus K_T$, so $(\H, h \circ f_T + Q \log |f_T'|)$ and $(\H \setminus K_T, h)$ describe equivalent quantum surfaces by \eqref{Qmap}.  Indeed,
$(\H, h \circ f_T + Q \log |f_T'|) = f_T^{-1} (\H \setminus K_T, h)$.}
\begin{eqnarray*} h & := & \h_0 + \widetilde h. \\
h \circ f_T + Q \log |f_T'| & = & \h_T + \widetilde h \circ f_T.
\end{eqnarray*}
\end{theorem}

\begin {figure}[htbp]
\begin {center}
\includegraphics [width=4in]{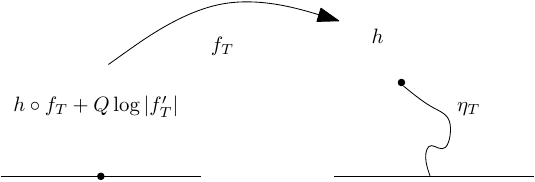}
\caption {\label{reverseftfig} Reverse coupling.}
\end {center}
\end {figure}

\begin{noindent}{\bf Alternative statement of Theorem \ref{reversecoupling}:} A more elegant way to state the theorem is that the law of $(\H,h)$ is invariant under the operation of independently sampling $f_T$, cutting out $K_T$ (equivalent to $\eta_T$ when $\kappa \leq 4$), and transforming via the coordinate change $f_T^{-1}$ (going from right to left in Figure \ref{reverseftfig}) in the manner of \eqref{Qmap}.
\end{noindent}
\vspace{.1in}

Both theorems give us an {\em alternate} way of sampling a distribution with the law of $h$ --- i.e., by first sampling the $B_t$ process (which determines $\eta_T$), then sampling a (fixed or free boundary) GFF $\widetilde h$ and taking
$$h = \h_T + \widetilde h \circ f_T.$$  This two part sampling procedure produces a coupling of $\eta_T$ with $h$.  In the forward SLE setting of Theorem \ref{forwardcoupling}, it was shown in \cite{MR2525778} that in any such coupling, $\eta_T$ is almost surely equal to a particular path-valued {\em function} of $h$.  (This was also done in \cite{SchrammSheffieldGFF2} in the case $\kappa = 4$.)  In other words, in such a coupling, $h$ determines $\eta_T$ almost surely.  This is important for our geometric interpretations.  Even though $h$ is not defined pointwise as a function, we would like to geometrically interpret $\eta$ as a level set of $h$ (when $\kappa = 4$) or a flow line of $e^{ih/\chi}$ (when $\kappa < 4$), as we stated above and will explain in more detail in Section \ref{windingintro}.  It is thus conceptually natural that such curves are uniquely determined by $h$ (as they would be if $h$ were a smooth function, see Section \ref{windingintro}).

As mentioned earlier, this paper introduces and proves Theorem \ref{reversecoupling} while highlighting its similarity to Theorem \ref{forwardcoupling}.   Indeed, it won't take us much more work to prove Theorems \ref{forwardcoupling} and \ref{reversecoupling} together than it would take to prove one of the two theorems alone.  It turns out that in both Figure \ref{forwardftfig} (which illustrates Theorem \ref{forwardcoupling}) and Figure \ref{reverseftfig} (which illustrates Theorem \ref{reversecoupling}), the field illustrated on the left hand side of the figure (which agrees with $h$ in law) actually determines $\eta_T$ and the map $f_T$, at least when $\kappa < 4$.  In the former context (Figure \ref{forwardftfig}) this a major result due to Dub\'edat \cite{MR2525778} (see also the exposition on this point in \cite{ms2012imag1}).  It says that a certain ``flow line'' is a.s.\ uniquely determined by $h$.  The statement in the latter context is a major result obtained in this paper, stated in Theorems \ref{conformalwelding} and \ref{conformalweldingunique}.  With some hard work, we will be able to show that the map $f_T$ describes a conformal welding in which boundary arcs of equal quantum boundary length are ``welded together''.  Once we have this, the fact that the boundary measure uniquely characterizes $f_T$ will be obtained by applying a general ``removability'' result of Jones and Smirnov, as we will explain in Section \ref{ss.confweldingintro}.

\subsection{Theorem statements: conformal weldings} \label{ss.confweldingintro}

We will now try to better understand Theorem \ref{reversecoupling} in the special case $\kappa < 4$. Note that a priori the $h$ in Theorem \ref{reversecoupling} is defined only up to additive constant.  We can either choose the constant arbitrarily (e.g., by requiring that the mean value of $h$ on some set be zero) or avoid specifying the additive constant and consider the measures $\mu_h$ and $\nu_h$ to be defined only up to a global multiplicative constant.  The choice does not affect the theorem statement below.

\begin{theorem} \label{conformalwelding}
Suppose that $\kappa < 4$ and that $h$ and $\eta_T$ are coupled in the way described at the end of the previous section, i.e., $h$ is generated by first sampling the $B_t$ process up to time $T$ in order to generate $f_T$ via a reverse Loewner flow, and then choosing $\widetilde h$ independently and writing $h = \h_T + \widetilde h \circ f_T$, and $\eta_T\bigl((0,T]\bigr) = \H \setminus f_T \H$.\footnote{It is not known whether an analog of Theorem \ref{conformalwelding} can obtained in the case $\kappa=4$.  The standard procedure for constructing the boundary measure $\nu_h$ breaks down when $\kappa=4, \gamma=2$, but a scheme was introduced \cite{duplantier2012critical, duplantier2012renormalization} to create a non-trivial boundary measure $\nu_h$.  The open problems listed in Section \ref{s.questions} also also address a related question in the $\kappa > 4$ setting.}
  Given a point $z$ along the path $\eta_T$, let $z_- < 0 < z_+$ denote the two points in $\R$ that $f_T$ (continuously extended to $\R$) maps to $z$.  Then almost surely $$\nu_h([z_-, 0]) = \nu_h([0,z_+])$$
for all $z$ on $\eta_T$.
\end{theorem}

Theorem \ref{conformalwelding} is a relatively difficult theorem, and it will be the last thing we prove.We next define $R=R_h: (-\infty,0] \to [0, \infty)$ so that $\nu_h([x, 0]) = \nu_h([0,R(x)])$ for all $x$ (recall that $\nu$ is a.s.\ atom free \cite{2008arXiv0808.1560D}).  This $R$ gives a homeomorphism from $[0_-, 0]$ to $[0,0_+]$ that we call a {\bf conformal welding} of these two intervals.  We stress that the values $0_-$ and $0_+$ depend on $T$, but the overall homeomorphism $R$ between $(-\infty,0]$ and $[0, \infty)$ is determined by the boundary measure $\nu_h$, whose law does not depend on $T$ (although the coupling between $h$, $\tilde h$, and $\eta_T$ described in the theorem statement clearly depends on $T$).  Since $\eta_T$ is simple, it clearly determines the restriction of $R$ to $[0_-, 0]$.  (See Figure \ref{zipperfig}.)  It turns out that $R$ also determines $\eta_T$:

\begin{theorem}\label{conformalweldingunique}
For $\kappa < 4$, in the setting of Theorem \ref{conformalwelding}, the homeomorphism $R$ from $[0_-, 0]$ to $[0,0_+]$ uniquely determines the curve $\eta_T$.  In other words, it is almost surely the case that if $\widetilde \eta_{\tilde T}$ is any other simple curve in $\H$ such that the homeomorphism induced by its reverse Loewner flow is the same as $R$ on $[0_-, 0]$, then $\widetilde \eta_{\tilde T} = \eta_T$.  In particular, $h$ determines $\eta_T$ almost surely.
\end{theorem}
\begin{proof} The author learned from Smirnov that Theorem \ref{conformalweldingunique} follows almost immediately from Theorem \ref{conformalwelding} together with known results in the literature.  If there were a distinct candidate $\widetilde \eta_T$ with a corresponding $\widetilde f_T$, then $\phi = \widetilde f_T \circ f_T^{-1}$ --- extended from $\H$ to $\R$ by continuity, and to all of $\C$ by Schwarz reflection --- would be a non-trivial homeomorphism of $\C$ (with $\lim_{z \to \infty} \phi(z) -z = 0$) which was conformal on $\C \setminus (\eta_T \cup \bar \eta_T)$, where $\bar \eta_T$ denotes the complex conjugate of $\eta_T$.  Thus, to prove Theorem \ref{conformalweldingunique}, it suffices to show that no such map exists.  In complex analysis terminology, this is equivalent by definition to showing that the curve $\eta_T \cup \bar \eta_T$ is {\em removable}.
Rohde and Schramm showed that the complement of $\eta([0,T])$ is a.s.\ a H\"older domain for $\kappa < 4$ (see Theorem 5.2 of \cite{MR2153402}) and that $\eta$ is a.s.\ a simple curve in this setting.  In particular, $\eta_T \cup \bar \eta_T$ is almost surely the boundary of its complement, and this complement is a H\"older domain.  (More about H\"older continuity appears in work of Beliaev and Smirnov \cite{MR2525631} and Kang \cite{MR2257400} and Lind \cite{MR2386236}.)  Jones and Smirnov showed generally that boundaries of H\"older domains are removable (Corollary 2 of \cite{MR1785402}).  The same observations are used in \cite{2009arXiv0909.1003A}. \end{proof}

We remark that the above arguments also show that $\eta \cup \bar \eta$ is removable when $\eta$ is the entire SLE path.  In the coming sections, we will often interpret the left and right components of $\H \setminus \eta$ as distinct quantum surfaces, where the right boundary arc of one surface is welded (along $\eta$) to the left boundary arc of another surface in a quantum-boundary-length-preserving way.  When the law of $\eta$ is given by SLE$_\kappa$ with $\kappa < 4$, removability implies that $\eta$ is almost surely determined (up to a constant rescaling of $\H$) by the way that these boundary arcs are identified.  In other words, aside from constant rescalings, there is no homeomorphism of $\H$, fixing $0$ and $\infty$, whose restriction to $\H \setminus \eta$ is conformal.

\subsection{Corollary: capacity stationary quantum zipper} \label{ss.capstationary}

This subsection contains some discussion and interpretation of some simple consequences of Theorems \ref{conformalwelding} and \ref{conformalweldingunique}, in particular Corollary \ref{stationaryzipping} below.  We first observe that for $\kappa < 4$, Theorem \ref{conformalweldingunique} implies that $R$ determines $\eta_T$ almost surely for any given $T > 0$.  In particular, this means that $R$ determines an entire reverse Loewner evolution $f_t = f^h_t$ for all $t \geq 0$, and that this $f^h_t$ is (in law) a {\em reverse} SLE$_\kappa$ flow.    Similarly, given a chordal curve $\eta$ from $0$ to $\infty$ in $\H$, we denote by $f^\eta_t$ the {\em forward} Loewner flow corresponding to $\eta$.  The following is now an immediate corollary of the domain Markov property for SLE and Theorems \ref{reversecoupling}, \ref{conformalwelding} and \ref{conformalweldingunique}.  As usual, transformations $f(D,h)$ are defined using \eqref{Qmap}.

\begin{corollary} \label{stationaryzipping}
Fix $\kappa \in (0,4)$.  Let $h = \h_0 + \widetilde h$ be as in Theorem \ref{reversecoupling} and let $\eta$ be an SLE$_\kappa$ on $\H$ chosen independently of $h$.  Let $D_1$ be the left component of $\H \setminus \eta$ and $h^{D_1}$ the restriction of $h$ to $D_1$.    Let $D_2$ be the right component of $\H \setminus \eta$ and $h^{D_2}$ the restriction of $h$ to $D_2$.  For $t \geq 0$, write
\begin{eqnarray*}
\zcap_t \bigl((D_1, h^{D_1}), (D_2, h^{D_2})\bigr) &=& \bigl(f^h_t(D_1, h^{D_1}), f^h_t(D_2, h^{D_2})\bigr),\\
\zcap_{-t} \bigl((D_1, h^{D_1}), (D_2, h^{D_2})\bigr) &=& \bigl(f^\eta_t(D_1, h^{D_1}), f^\eta_t(D_2, h^{D_2})\bigr). \end{eqnarray*}
Note that both $h$ and $\eta$ are determined by the pair $\bigl((D_1, h^{D_1}), (D_2, h^{D_2})\bigr)$, and that $f_t^h$ and $f^\eta_t$ are also a.s.\ determined by this pair, so that 
the maps $\zcap_t$ and $\zcap_{-t}$ are well defined for almost all pairs $\bigl((D_1, h^{D_1}), (D_2, h^{D_2})\bigr)$ chosen in the manner described above.
Then the law of $\bigl((D_1, h^{D_1}), (D_2, h^{D_2})\bigr)$ is invariant under $\zcap_t$ for all $t$.  Also, for all $s$ and $t$, $$\zcap_{s+t} = \zcap_s \zcap_t$$ almost surely.
\end{corollary}

\begin {figure}[htbp]
\begin {center}
\includegraphics [width=2in]{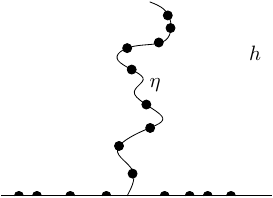}
\caption {\label{zipperfig} Sketch of $\eta$ with marks spaced at intervals of the same $\nu_h$ length along $\partial D_1$ and $\partial D_2$.  Here $(-\infty, 0]$ and $[0,\infty)$ are the two open strands of the ``zipper'' while $\eta$ is the closed (zipped up) strand.  Semicircular dots on $\R$ are ``zipped together'' by $f^h_t$. Circular dots on $\eta$ are ``pulled apart'' by $f^\eta_t$.  (Recall that under the reverse Loewner flow $f^h_t$, the center of a semicircle on the negative real axis will reach the origin at the same time as the center of the corresponding semicircle on the positive real axis.)  The law of $\bigl((D_1, h^{D_1}), (D_2, h^{D_2})\bigr)$ is invariant under ``zipping up'' by $t$ capacity units or ``zipping down'' by $t$ capacity units.}
\end {center}
\end {figure}

Because the forward and reverse Loewner evolutions are parameterized according to half plane capacity, we refer to the group of transformations $\zcap_t$ as the {\bf capacity quantum zipper}, see Figure \ref{zipperfig}.  (The term ``zipper'' in the Loewner evolution context has been used before; see the ``zipper algorithm'' for numerically computing conformal mappings in \cite{MR2361903} and the references therein.)  When $t>0$, applying $\zcap_t$ is called ``zipping up'' the pair of quantum surfaces by $t$ capacity units and applying $\zcap_{-t}$ is called ``zipping down'' or ``unzipping'' by $t$ capacity units.

\vspace{.1in}

To begin to put this construction in context, we recall that the general {\em conformal welding} problem is usually formulated in terms of identifying unit discs $\D_1$ and $\D_2$ along their boundaries via a given homeomorphism $\phi$ from $\partial \D_1$ to $\partial \D_2$ to create a sphere with a conformal structure.  Precisely, one wants a simple loop $\eta$ in the complex sphere, dividing the sphere into two pieces such that if conformal maps $\psi_i$ from the $\D_i$ to the two pieces are extended continuously to their boundaries, then $\psi_1 \circ \psi_2^{-1}$ is $\phi$.  In general, not every homeomorphism $\phi$ between disc boundaries is a conformal welding in this way, and when it is, it does not always come from an $\eta$ that is (modulo conformal automorphisms of the sphere) unique; in fact, arbitrarily small changes to $\phi$ can lead to large changes in $\eta$ and some fairly exotic behavior (see e.g.\ \cite{MR2373370}).

The theorems of this paper can also be formulated in terms of a sphere obtained by gluing two discs along their boundaries: in particular, one can zip up the quantum surfaces of Corollary \ref{stationaryzipping} ``all the way'' (see Figure \ref{reverseftnormalizedfig} and Section \ref{ss.capstataddconst}), which could be viewed as welding two Liouville quantum surfaces (each of which is topologically homeomorphic to a disc) to obtain an SLE loop in the sphere, together with an instance of the free boundary GFF on the sphere.

Note that in the construction described above, the quantum surfaces are defined only modulo an additive constant for the GFF, and we construct the two surfaces together in a particular way.  In Section \ref{ss.lengthstationary} (Theorem \ref{t.lengthstationary}), we will describe a related construction in which one takes two {\em independent} quantum surfaces (each with its additive constant well-defined) and welds them together to obtain SLE.

As mentioned in Section \ref{ss.overview}, Peter Jones conjectured several years ago that an SLE loop could be obtained by (what in our language amounts to) welding a quantum surface to a deterministic Euclidean disc.  (The author first learned of this conjecture during a private conversation with Jones in early 2007 \cite{privatejones}.)  Astala, Jones, Kupiainen, and Saksman recently showed that such a welding exists and determines a unique loop (up to conformal automorphism of the sphere) \cite{2009arXiv0909.1003A, MR2600118}. Binder and Smirnov recently announced (to the author, in private communication \cite{privatesmirnov}) that they have obtained a proof that the original conjecture of Jones is false. By computing a multifractal spectrum, they showed that the loop constructed in \cite{2009arXiv0909.1003A, MR2600118} does not look locally like SLE.  However, our construction, together with Theorem \ref{t.lengthstationary} below, shows that a natural variant of the Jones conjecture --- involving two independent quantum surfaces instead of one quantum surface and one Euclidean disc --- is in fact true.

We also remark that the ``natural'' $d$-dimensional measure on (or parameterization of) an SLE curve of Hausdorff dimension $d$ was only constructed fairly recently \cite{2009arXiv0906.3804L, 2010arXiv1006.4936L, lawler2012minkowski}, and it was shown to be uniquely characterized by certain symmetries, in particular the requirement that it transforms like a $d$-dimensional measure under the maps $f_t$ (i.e., if the map locally stretches space by a factor of $r$, then it locally increases the measure by a factor of $r^d$).  Our construction here can be viewed as describing, for $\kappa < 4$, a natural ``quantum'' parameterization of SLE$_\kappa$, which is similarly characterized by transformation laws, in particular the requirement that adding $C$ to $h$ --- which scales area by a factor of $e^{\gamma C}$ --- scales length by a factor of $e^{\gamma C/2}$.  These ideas are discussed further in \cite{DS3}.

The relationship between Euclidean and quantum natural fractal measures and their evolution under capacity invariant quantum zipping is developed in \cite{DS3} in a way that makes use of the KPZ formula \cite{MR947880, 2008arXiv0808.1560D}.

\subsection{Quantum wedges and quantum length stationarity} \label{ss.lengthstationary}

This subsection contains ideas and definitions that are important for the proofs of Theorem \ref{conformalwelding} and \ref{conformalweldingunique}, as well as the statement of another of this paper's main results, Theorem \ref{t.lengthstationary}, which we will actually prove before Theorem \ref{conformalwelding}.  The reader who prefers to first see proofs of Theorems \ref{forwardcoupling} and \ref{reversecoupling} and some discussion of the consequences may read Sections \ref{GFFoverview} and \ref{couplingsection}, as well as much of Section \ref{s.zipper}, independently of this subsection.

Theorem \ref{t.lengthstationary} includes a variant of Corollary \ref{stationaryzipping} in which one parameterizes time by ``amount of quantum length zipped up'' instead of by capacity.  The ``stationary'' picture will be described as a particular random quantum surface $\mathcal S$ with two marked boundary points and a chordal SLE $\eta$ connecting the two marked points.  The theorem will state that this $\eta$ divides $\mathcal S$ into two quantum surfaces $\mathcal S_1$ and $\mathcal S_2$ that are {\em independent} of each other.  (One can also reverse the procedure and first choose the $\mathcal S_i$ --- these are the so-called $\gamma$-quantum wedges mentioned earlier --- and then weld them together to produce $\mathcal S$ and the interface $\eta$.)  As we have already mentioned, this independence appears at first glance to be a rather bizarre coincidence.  However, as we will see in Section \ref{ss.discretelimitdiscussion}, this kind of result is to be expected if SLE-decorated Liouville quantum gravity is (as conjectured) the scaling limit of path-decorated random planar maps.

Before we state Theorem \ref{t.lengthstationary} formally, we will need to spend a few paragraphs constructing a particular kind of scale invariant random quantum surface that we will call an ``$\alpha$ quantum wedge.''  The reader who has never encountered quantum wedges before before may wish to first read Section 1.4 of \cite{wedgespaper}, which contains a more recent and better illustrated discussion of the quantum wedge construction.

We begin this construction by making a few general remarks. Recall that given any quantum surface represented by $(\widetilde D, \widetilde h)$ --- with two distinguished boundary points --- we can change coordinates via \eqref{Qmap} and represent it as the pair $(\H, h)$ for some $h$, where $\H$ is the upper half plane, and the two marked points are taken to be $0$ and $\infty$.  We will represent the ``quantum wedges'' we construct in this way, and we will focus on constructions in which there is almost surely a finite amount of $\mu_h$ mass and $\nu_h$ mass in each bounded neighborhood of $0$ and an infinite amount in each neighborhood of $\infty$.  In this case, the corresponding quantum surface is {\bf half-plane-like} in the sense that it has one distinguished boundary point ``at infinity'' and one distinguished ``origin'' --- and each neighborhood of ``infinity'' includes infinite area and an infinite length portion of the surface boundary, while the complement of such a neighborhood contains only finite area and a finite-length portion of the surface boundary.  We will let $\mathcal S_h$ denote the doubly marked quantum surface described by $h$ in this way.

The $h$ describing $\mathcal S_h$ is canonical except that we still have one free parameter corresponding to constant rescalings of $\H$ by \eqref{Qmap}.  For each $a > 0$, such a rescaling is given by
 \begin{equation} \label{Qmaprescaling} (\H,h) \to (\H , h (a \cdot) + Q \log |a|). \end{equation}
We can fix this parameter by requiring that $\mu_h \bigl(B_1(0) \cap \H \bigr) = 1$.  We will let $\mu_h$ be zero on the negative half plane so that we write this slightly more compactly as $\mu_h \bigl(B_1(0) \bigr) = 1$. (Alternatively, one could normalize so that $\nu_h\bigl([-1,1]\bigr) = 1$.)  We call the $h$ for which this holds the {\bf canonical description} of the doubly marked quantum surface.

Now to construct a ``quantum wedge'' it will suffice to give the law of the corresponding $h$.  To this end, we first recall that one can decompose the Hilbert space for the free boundary GFF into an orthogonal sum of the space of functions which are radially symmetric about zero and the space of functions with zero mean about all circles centered at zero  \cite{2008arXiv0808.1560D}.  Consequently, we can write $h(\cdot) = h_{|\cdot|}(0) + h^\dagger(\cdot)$, where $h^\dagger_\eps(0) = 0$ for all $\eps$, and $h_{|z|}(0)$ is (of course) a continuous and radially symmetric function of $z$.  This is a decomposition of the GFF $h$ into its projection onto two $(\cdot, \cdot)_\nabla$ orthogonal subspaces, so $h_{|\cdot|}(0)$ and $h^\dagger(\cdot)$ are independent of each other \cite{Sh}; the latter is a scale invariant random distribution and defined without an additive constant (since its mean is set to be zero on all circles centered at the origin).  Now we define three types of quantum surfaces (the first two being defined only up an additive constant for $h$, which corresponds to a constant-factor rescaling of the surface itself).  The third may seem unmotivated; however, the reader may note that it is similar in the spirit to the second, except that the third $h$ is actually a well defined random distribution (as opposed to a random distribution modulo additive constant), so that $(\H, h)$ is a well-defined quantum surface.

\begin{enumerate}
\item {\bf Definition --- unscaled quantum wedge on $\H$:} the quantum surface $(\H, h)$ where $h$ is an instance of the free boundary GFF (which is defined up to additive constant, so that the quantum surface is defined only up to rescaling).  In this case, $h_{|\cdot|}$ agrees in law with $B_{-\log|\cdot|}$ when $B_t, \,\,\, t \in \R$ is $\sqrt{2}$ times a standard Brownian motion defined up to a global additive constant).  We think of $B_t$ as a Brownian motion with diffusive rate $2$, which will be understood throughout the discussion below.  We can write $$h = h^\dagger(\cdot) +  B_{-\log|\cdot|},$$ where $h^\dagger(\cdot)$ and $B_{-\log|\cdot|}$ are independent.
\item {\bf Definition --- $\alpha$-log-singular free quantum surface on $\H$:} the quantum surface $(\H,h)$ where \begin{equation} \label{eqn::freesurface} h = h^\dagger(\cdot) + \alpha\bigl(- \log |\cdot|\bigr)+  B_{-\log|\cdot|} ,\end{equation} with $h^\dagger$ and $B$ as above (and $h$ also defined only up to additive constant).
\item {\bf Definition --- $\alpha$-quantum wedge:} for $\alpha < 0$, the quantum surface $(\H, h)$ where \begin{equation}
\label{eqn::hwedge} h = h^\dagger(\cdot) +  Q \bigl(- \log|\cdot| \bigr) + A_{-\log|\cdot|}, \end{equation} and the process $A_t, \,\,\, t \in \R$ is defined in a particular way: namely, for $t \geq 0$,  $A_t$ is a Brownian motion with drift $\alpha - Q$, i.e., $A_t = B_t + (\alpha - Q)t$.  Also, for $t \geq 0$, the negative-time process $A_{-t}$ is chosen independently as a Brownian motion with drift $-(\alpha - Q)$ {\em conditioned} not to revisit zero.  This involves conditioning on a probability zero event, so let us state this another way to be clear.  Note that $\tilde B_t = B_t -(\alpha-Q)t$ has positive drift and hence a.s.\ $s_0 = \sup \{s : \tilde B_s = 0 \} < \infty$.  Then the law of $A_{-t}$ (for $t \geq 0$) is the law of $\tilde B_{t+s_0}$, for $t \geq 0$.
\end{enumerate}

To begin to motivate the definition above, note that applying the coordinate transformation \eqref{Qmaprescaling} to the $\alpha$-quantum wedge defined by \eqref{eqn::hwedge}, where the coordinate change map is a rescaling by a factor of $a$, amounts to replacing \eqref{eqn::hwedge} with $$h^\dagger(a \cdot) +  Q \bigl(- \log|a \cdot| \bigr) + A_{-\log|a \cdot|} + Q \log|a|=  h^\dagger(a \cdot) +  Q \bigl(- \log|\cdot| \bigr) + A_{\log a -\log|\cdot|} .$$
Since the law of $h^\dagger$ is scaling invariant, we find that the coordinate change described amounts to a horizontal translation of $A$ by $- \log a$.  That is, the quantum surface obtained by sampling $A$ and then sampling $h^\dagger$ independently agrees in law with the quantum surface obtained by sampling $A$, translating the graph of $A$ horizontally by some (possibly random) amount, and then sampling $h^\dagger$ independently.

We think of $A_t$ as a Brownian process that drifts steadily as a Brownian motion with drift $(\alpha - Q)$ from $-\infty$, reaches zero at some point, and then subsequently evolves as a regular Brownian motion with the same drift.  Since translating the graph of $A_t$ horizontally doesn't affect the law of the quantum surface obtained, we choose (for concreteness) the translation for which $\inf \{t: A_t = 0 \} = 0$.   (We remark that the process $A_t$ can also be interpreted as the log of a Bessel process, reparameterized by quadratic variation, noting that the graph of such a reparameterization is {\em a priori} only defined up to a horizontal translation; this point of view is explained and used extensively in \cite{wedgespaper}.)

    % via \eqref{Qmaprescaling}, thereby horizontally translating the second graph (in Figure \ref{logplots} or \ref{logplots2}) so that it passes through the origin.  Since replacing $C$ with $C+C'$ (where $C'$ is fixed) does not affect the $C \to \infty$ limit, the law of the quantum wedge we have constructed is indeed invariant under multiplying area by a constant.
%with the given drift when $s > 0$, but $A_{-s}$ is a Brownian motion with minus that drift conditioned fnot hit zero again (so $A_s$ first reaches $0$ at time $s=0$).  From the description above, one can show that the law of the process in the first graph, recentered at $(s',C)$, tends to the law of $A$ as $C$ gets large.
    Now we make another simple claim: the $\alpha$-quantum wedge is a doubly marked quantum surface whose law is invariant under the multiplication of its area by a constant.
To explain what this means, let us observe that when $C \in \R$, we can ``multiply the surface area by the constant $e^C$'' by replacing $h$ with $h+C/\gamma$, or equivalently, by replacing $A$ with $A + C/\gamma$.  Let $t_0 = \inf \{t: \tilde A_t = 0 \}$ and write $\tilde A_t = A_{t_0 +t} + C/\gamma$.  By the definition of $t_0$, we find that $\tilde A_t$ (like $A_t$) is a process that drifts up from $-\infty$, reaches zero for the first time when $t=0$, and then subsequently evolves as a Brownian motion with drift.  Indeed, it is not hard to see that $\tilde A_t$ has the same law as $A_t$. To deduce the claim, we then observe that the distribution of $h^{\dagger}$ is fixed; and since the radial parts $h_{|\cdot|}(0)$ of the GFF are continuous and independent of $\mu_{h^{\dagger}}$ and converge to a limit in law, we may conclude that $e^{\gamma h_{|\cdot|}(0)} d \mu_{h^{\dagger}}$ converges in law.
%A canonical description of $\mathcal S_{h+C/\gamma}$ is obtained by finding an $\eps$ for which $\mu_{h+C/\gamma}\bigl(B_\eps(0)\bigr) = 1$, or equivalently $\mu_h\bigl(B_\eps(0)\bigr) = e^{-C}$, and rescaling by \eqref{Qmaprescaling}.
  %In other words, we find an origin-centered semi-disc of $\mu_h$ area $e^{-C}$, we rescale $\H$ via \eqref{Qmaprescaling} so that this semi-disc is mapped to the Euclidean semi-disc $B_1(0) \cap \H$, and we add the constant $C/\gamma$ to $h$ so that the quantum area of this semi-disc becomes one.  Our task is now to define a class of doubly marked quantum surfaces (quantum wedges) whose canonical descriptions have laws that are invariant under area rescalings of this type.

For future reference, we mention that one has a natural notion of ``convergence'' for quantum surfaces of this type: if $h^1, h^2, \ldots$ are the canonical descriptions of a sequence of doubly marked quantum surfaces and $h$ is the canonical description of $\mathcal S_h$, then we say that the sequence $\mathcal S_{h^i}$ {\bf converges} to $\mathcal S_h$ if the corresponding measures $\mu_{h^i}$ converge weakly to $\mu_h$ on all bounded subsets of $\H$.

One motivation for the definition of a quantum wedge is the following, which can be deduced from the description of quantum typical points given in Section 6 of \cite{2008arXiv0808.1560D}.  It says (in a certain special setting; for a stronger result, see Proposition \ref{p.gammawedge2}) that if one zooms in near a ``quantum-boundary-measure-typical'' point, one finds that the quantum surface looks like a $\gamma$-quantum wedge near that point.

\begin{proposition} \label{p.gammawedge}
Fix $\gamma \in [0,2)$ and let $D$ be a bounded subdomain of $\H$ for which $\partial D \cap \R$ is a segment of positive length.  Let $\widetilde h$ be an instance of the GFF with zero boundary conditions on $\partial D \setminus \R$ and free boundary conditions on $\partial D \cap \R$.  Let $[a,b]$ be any sub-interval of $\partial D \cap \R$ and let $\h_0$ be a continuous function on $D$ that extends continuously to the interval $(a,b)$.
\begin {figure}[htbp]
\begin {center}
\includegraphics [height=.75in]{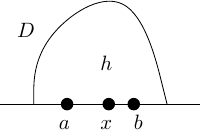}
\caption {\label{boundaryx} Point $x$ sampled from $\nu_h$ (restricted to $[a,b]$).}
\end {center}
\end {figure}
Let $dh$ be the law of $\h_0 + \widetilde h$, and let $\nu_h[a,b]dh$ denote the measure whose Radon-Nikodym derivative w.r.t.\ $dh$ is $\nu_h[a,b]$.  (Assume that this is a finite measure --- i.e., the $dh$ expectation of $\nu_h[a,b]$ is finite.)  Now suppose we \begin{enumerate}
\item sample $h$ from $\nu_h[a,b] dh$ (normalized to be a probability measure),
\item then sample $x$ uniformly from $\nu_h$ restricted to $[a,b]$ (normalized to be a probability measure),
\item and then let $h^*$ be $h$ translated by $-x$ units horizontally (i.e., recentered so that $x$ becomes the origin).
\end{enumerate}
Then as $C \to \infty$ the random quantum surfaces $\mathcal S_{h^* + C/\gamma}$ converge in law (w.r.t.\ the topology of convergence of doubly-marked quantum surfaces) to a $\gamma$-quantum wedge.
\end{proposition}

\begin{proof}
We first recall that in this setting the description of quantum typical points in Section 6 of \cite{2008arXiv0808.1560D} implies a very explicit description of the joint law of the pair $x$ and $h$ sampled in Proposition \ref{p.gammawedge}.  The marginal law of $x$ is absolutely continuous with respect to Lebesgue measure, and conditioned on $x$ the law of $h$ is that of its original law {\em plus} a deterministic function that has the form $-\gamma \log|x - \cdot|$ plus a deterministic smooth function.  In a small neighborhood of $x$, this deterministic smooth function is approximately constant, which means that $h^*$ looks like (up to additive constant) the $h$ used to define an $\alpha$-$\log$-singular free quantum surface in \eqref{eqn::freesurface}, with $\alpha = \gamma$.  If we write $A_t' = B_t + (\alpha - Q)t$, then we find that $h^*$ looks like the $h$ used to define a $\gamma$-quantum wedge in \eqref{eqn::hwedge}, except with $A$ replaced by $A'$.

Now replacing $h^*$ by $h^* + C/\gamma$ corresponds to adding $C/\gamma$ to the process $B$ from \eqref{eqn::freesurface}, and hence also corresponds to adding $C/\gamma$ to the process $A'$, which translates the graph of $A'$ vertically. Recall from above that translating the graph of $A'$ horizontally corresponds to a coordinate change; so we can translate $A'$ so that it hits zero for the first time at the origin. It is not hard to see that as $C \to \infty$, the law of $A'$ thus translated converges to the law of $A$. Since the law of $h^\dagger$ is scale invariant and can be chosen independently, this implies the proposition statement.
\end{proof}

      We will later show (see Proposition \ref{p.gammawedge2}) that the conclusion of the proposition still holds if (when generating $x$ and $h$) we condition on particular values for $\nu_h[a,x]$ and $\nu_h[x,b]$.

The following is an immediate consequence of Proposition \ref{p.gammawedge}.  It tells us that the $\gamma$-quantum wedge is stationary with respect to shifting the origin by a given amount of quantum length.  (When $\gamma = 0$, the proposition simply states that $\H$ itself is invariant under horizontal translations.  Proposition \ref{p.gammawedgeinvariant} is the general quantum analog of this invariance.)

\begin{proposition} \label{p.gammawedgeinvariant}
Fix a constant $L > 0$.  Suppose that $(\H, h)$ is a $\gamma$-quantum wedge.  Then choose $y > 0$ so that $\nu_h[0,y] = L$, and let $h^*$ be $h$ translated by $-y$ units horizontally (i.e., recentered so that $y$ becomes the origin).  Then $(\H, h^*)$ is a $\gamma$-quantum wedge.
\end{proposition}
\begin{proof}
Suppose that $x$ is the point chosen uniformly from the quantum boundary measure in Proposition \ref{p.gammawedge}, and $x'$ is the point translated $\delta L$ quantum length units to the right from $x$, so that $\nu_h[x,x'] = \delta L$. Note that such an $x'$ exists with a probability that tends to $1$ as $\delta \to 0$, and that the law of $x'$ converges (in total variation sense) to the law of $x$ as $\delta \to 0$. In the rescaled surfaces in Proposition \ref{p.gammawedge}, boundary lengths are scaled by $e^{C/2}$, so if we set $\delta = e^{-C/2}$, then the distance between $x$ and $x'$ is $L$ after the rescaling.  Since this $\delta$ tends to zero as $C \to \infty$ we conclude that the limiting surface law is (as desired) invariant under the operation that translates the origin by $L$ units of quantum boundary length.\end{proof}

%Note that the statement that $(\H, h)$ is a $\gamma$-quantum wedge only tells us about the law of $h$ modulo scalings of the form in \eqref{Qmaprescaling}.  If we wish to phrase the proposition in terms of canonical descriptions, we may first assume that $h$ is a canonical description.  To obtain $h^*$ we translate by $-x$ {\em and then rescale} via \eqref{Qmaprescaling}, so that $h^*$ is also a canonical description.

The following will be proved in Section \ref{s.zipper}.

%turn out to be a corollary of Corollary \ref{stationaryzipping} and the above propositions (though it will still take some work to prove this in.

\begin{theorem} \label{t.lengthstationary}
{\bf Wedge decomposition:} Fix $\gamma \in (0,2)$, and let $\mathcal S$ be a $(\gamma - 2/\gamma)$-quantum wedge with canonical description $h$.  Let $\eta$ be a chordal SLE$_\kappa$ in $\H$ from $0$ to $\infty$, with $\kappa = \gamma^2$, chosen independently of $h$.  Let $D_1$ and $D_2$ be the left and right components of $\H \setminus \eta$, and let $h^{D_1}$ and $h^{D_2}$ be the restrictions of $h$ to these domains.  Then the quantum surfaces represented by $(D_1, h^{D_1})$ and $(D_2, h^{D_2})$ are {\em independent} $\gamma$-quantum wedges (marked at $0$ and $\infty$), and their quantum boundary lengths along $\eta$ agree.

{\bf Zipper stationarity:} Moreover, suppose we define $$\zlength_{-t}\bigl((D_1, h^{D_1}), (D_2, h^{D_2})\bigr)$$ as follows.  First find $z$ on $\eta$ for which the quantum boundary lengths along $D_1$ and $D_2$ (which are well defined by unzipping) along $\eta$ between $0$ and $z$ are both equal to $t$.  Let $t'$ be the time that $\eta$ hits $z$ (when $\eta$ is parameterized by capacity) and define $$\zlength_{-t}\bigl((D_1, h^{D_1}), (D_2, h^{D_2})\bigr) = \,\,\, \mathrm{rescaling}\,\, \mathrm{of} \,\,\, \bigl(f^\eta_{t'} (D_1, h^{D_1}), f^\eta_{t'}(D_2, h^{D_2})\bigr),$$ where the rescaling is done via \eqref{Qmaprescaling} with the parameter $a$ chosen so that $B_1(0)$ has area one in the transformed quantum measure. Then the following hold: \begin{enumerate}
\item The inverse $\zlength_t$ of the operation $\zlength_{-t}$ is a.s.\ uniquely defined (via conformal welding).
\item $\zlength_{s+t} = \zlength_s \zlength_t$ almost surely for $s,t \in \R$.
\item The law of the pair $\bigl((D_1, h^{D_1}), (D_2, h^{D_2})\bigr)$ is invariant under $\zlength_t$ for all $t \in \R$.
\end{enumerate}
\end{theorem}

\begin {figure}[h]
\begin {center}
\includegraphics [height=2in]{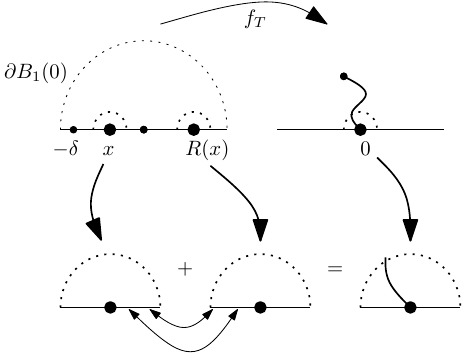}
\caption {\label{triplezoom} Choose $h$ as in Theorem \ref{reversecoupling} (normalized by $h_1(0) = 0$) except with the law of $h$ {\em weighted} by
$\nu_h\bigl([-\delta,0]\bigr)$ for some fixed $\delta \in (0,1)$.  Then conditioned on $h$, sample $x$ from $\nu_h$ restricted to $[-\delta,0]$ (normalized to be a probability measure).  Take $T$ so that $f_T$ is the map zipping up $[x,0]$ with $[0,R(x)]$.  Consider the three random surfaces obtained by choosing a semi-disc of quantum area $\tilde \eps$ centered at each of $x$ and $R(x)$ (on the left side) and $0$ (on the right side), and multiplying areas by $1/\tilde \eps$ (zooming in) so that all three balls have quantum area $1$.  In the $\tilde \eps \to 0$ limit, the left two quantum surfaces become independent $\gamma$-quantum wedges, and the right is the conformal welding of these two.}
\end {center}
\end {figure}

It also follows from Theorem \ref{conformalweldingunique} and the subsequent discussion that the two independent $\gamma$-quantum wedges uniquely determine $h$ and $\eta$ almost surely.
We refer to the group of transformations $\zlength_t$ as the {\bf length quantum zipper}.  When $t>0$, applying $\zlength_t$ is called ``zipping up'' the pair of quantum surfaces by $t$ quantum length units and applying $\zlength_{-t}$ is called ``zipping down'' or ``unzipping'' by $t$ quantum length units.  When we defined the operations $\zcap_t$, $h$ was defined only up to additive constant, and the zipping maps $f_t$ were independent of that constant.  By contrast, $\zlength_t$ represents zipping by an actual quantity of quantum length and hence cannot be defined without the additive constant being fixed.

We will give a detailed proof in Section \ref{s.zipper}, which is in some sense the heart of the paper.  But for now, let us give a brief overview of the proof and the relationship to our other results.  We will start with the scenario described in Figure \ref{zipperfig}, with $h$ normalized to have mean zero on $\partial B_1(0)$, except that the measure $dh$ on $h$ is replaced by the probability measure whose Radon-Nikodym derivative w.r.t.\ $dh$ is $\nu_h(-\delta,0)$ for some fixed $\delta$ (see Figure \ref{triplezoom}).

Then we will sample $x$ from $\nu_h$ restricted to $(-\delta,0)$ (normalized to be a probability measure) and ``zip up'' until $x$ hits the origin (to obtain a ``quantum-length-typical'' configuration).  We then zoom in near the origin (multiplying the area by $\tilde \eps^{-1}$ --- and hence the boundary length by $\tilde \eps^{-1/2}$ --- say).  We then use a variant of Proposition \ref{p.gammawedge} (namely, Proposition \ref{p.gammawedge2}) to show that (in the $\tilde \eps \to 0$ limit) the lower two rescaled surfaces on the lower left of Figure \ref{triplezoom} become independent $\gamma$-quantum wedges.

The fact that the curve on the right in Figure \ref{triplezoom} is (in the $\tilde \eps \to 0$ limit) an SLE$_\kappa$ independent of the canonical description $h$ on the right will be shown in Section \ref{s.zipper} by directly calculating the law of the process that ``zips up'' $[x,0]$ with $[0,R(x)]$.
It could also be seen by showing that we can construct an equivalent pair of glued surfaces by beginning with Figure \ref{zipperfig} (with $h$ normalized to have mean zero on $\partial B_1(0)$) and then zipping {\em down} by a random amount (chosen uniformly from an interval) of quantum length, then zooming in by multiplying lengths by $1/\tilde \eps$, and then taking the $\tilde \eps \to 0$ limiting law.   (In this case, the domain Markov property of the original SLE, and its independence from the original GFF, would imply that the conditional law of the still-zipped portion of the curve is an SLE$_\kappa$, independent of $h$.)

Similar arguments to those in \cite{2008arXiv0808.1560D} will show that the procedure in Figure \ref{triplezoom} produces a configuration related to the one in Figure \ref{reverseftfig} except that it is in some sense {\em weighted} by the amount of quantum mass near zero.  It will turn out that this weighting effectively adds $-\gamma \log|\cdot|$ to the $\h_0$ of Theorem \ref{reversecoupling} and Corollary \ref{stationaryzipping}.  This is why Theorem \ref{t.lengthstationary} involves a $(\gamma-2/\gamma)$-quantum wedge, instead of a $(-2/\gamma)$-quantum wedge, as one might initially guess based on Theorem \ref{reversecoupling}.  Once we have all of this structure in place, the really crucial step will be showing that parameterizing time by the amount of ``left boundary quantum length'' zipped up yields the same stationary picture as parameterizing time the amount of ``right boundary quantum length'' zipped up.  Given this, we will then use the ergodic theorem to show that over the long term, the amount of left bounday quantum length zipped up approximately agrees with the amount of right boundary length zipped up.  Using scale invariance symmetries we will then deduce that this agreement almost surely holds exactly on all scales.

\subsection{Organization}

Section \ref{s.interpretation} provides heuristic justification and motivation for the main results about AC geometry and Liouville quantum gravity.  (An interpretation of AC geometry in terms of ``imaginary curvature'' appears in the appendix.) Section \ref{GFFoverview} gives a brief overview of the Gaussian free field.   Section \ref{couplingsection} proves Theorems \ref{forwardcoupling} and \ref{reversecoupling}, along with a generalization to other underlying geometries. Section \ref{s.zipper} proves Theorems \ref{t.lengthstationary} and \ref{conformalwelding} (in that order), along with some additional results about zipping processes and time changes.  (Recall that we have already proved that Theorem \ref{conformalweldingunique} is a consequence of  Theorem \ref{conformalwelding}.)  Section \ref{s.questions}, finally, presents a list of open problems and conjectures.

\section{Geometric interpretation} \label{s.interpretation}
This section summarizes some of the conjectures and intuition behind our main results, including some discrete-model-based reasons that one would expect the coupling and welding theorems to be true.  This section may be skipped by the reader who prefers to proceed directly to the proofs.

\subsection{Forward coupling: flow lines of $e^{ih/\chi}$} \label{windingintro}

Fix a planar domain $D$, viewed as a subset of $\mathbb C$, a
function $h: D \rightarrow \mathbb R$, and a constant $\chi > 0$. An
{\bf AC ray} of $h$ is a flow line of the complex vector field
$e^{ih/\chi}$ beginning at a point $x \in \overline D$ --- i.e., a
path $\eta:[0,\infty) \to \C$ that is a solution to the ODE:
\begin{equation} \label{e.flowODE} \eta'(t) := \frac{\partial}{\partial t} \eta(t) =
e^{ih(\eta(t))/\chi} \text{  when $t > 0$ }, \,\,\,\,\,\,\,\,
\eta(0) = x,
\end{equation}
until time $T = \inf \{t > 0: \eta(t) \not \in D \}$. When $h$ is Lipschitz, the standard
Picard-Lindel\"of theorem implies that if $x \in D$, then
\eref{e.flowODE} has a unique solution up until time $T$ (and $T$ is
itself uniquely determined). The reader can visually follow the flow
lines in Figure \ref{flowfigure}.

\begin {figure}[htbp]
\begin {center}
\includegraphics [bb=0in 0in 4in 4in, width = 2.5in]{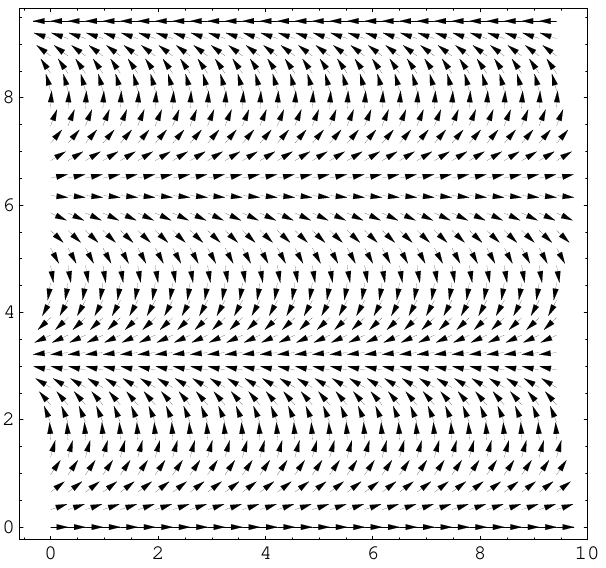}
\includegraphics [bb=0in 0in 4in 4in, width = 2.5in]{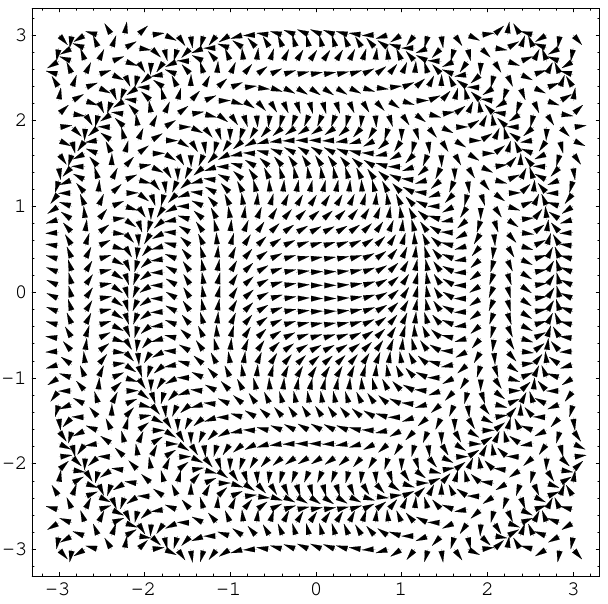}
\caption {\label{flowfigure} The complex vector flow $e^{ih}$: $h(x,y) = y$,
$h(x,y) = x^2+y^2$.}
\end {center}
\end {figure}

If $h$ is continuous, then the time derivative $\eta'(t)$ moves
continuously around the unit circle, and $h(\eta(t)) -
h(\eta(0))$ describes the net amount of {\it winding} of $\eta'$
around the circle between times $0$ and $t$.

A major problem (addressed in depth in an imaginary geometry series \cite{ms2012imag1,ms2012imag2,ms2012imag3,ms2013imag4}) is to make sense of these flow lines when $h$ is a multiple of the Gaussian free field.  We will give here just a short overview of the way these objects are constructed.  Suppose that $\eta$ is a smooth simple path in $\H$ beginning at the
origin, with (forward) Loewner map $f_t = f^\eta_t$.  We may assume that $\eta$ starts out in the vertical direction,
so that the winding number is $\pi/2$ for small times.
Then when $\eta$ and $h$ are both smooth, the statement that
$\eta$ is a flow line of $e^{i h/\chi}$ is equivalent to
the statement that for each $x$ on $\eta\bigl((0,t)\bigr)$ we have \begin{equation} \label{argleft} \chi \arg f_t'(z) \to -h(x)\end{equation} as $z$ approaches $x$ from
the left side of $\eta$ and \begin{equation} \label{argright} \chi \arg f_t'(z) \to - h(x) + \chi \pi\end{equation} as $z$ approaches $x$ from the right side of $\eta$ (as Figure \ref{windingfig} illustrates).  Recall that $\arg f'_t(z)$ --- a priori determined only up to a multiple of $2 \pi$ --- is chosen to be continuous on $\H \setminus \eta([0,t])$ and $0$ on $\R$.  If $\chi=0$, then \eqref{argleft} and \eqref{argright}
hold if and only if $h$ is identically zero along the path $\eta$, i.e.,
$\eta$ is a zero-height {\em contour line} of $h$.

\begin {figure}[htbp]
\begin {center}
\includegraphics [width=4.5in]{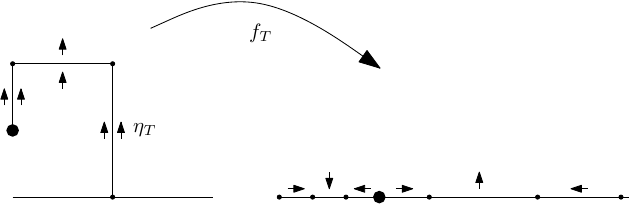}
\caption {\label{windingfig} Winding number along $\eta_T$ determines $\arg f_T'$, which is the amount a small arrow near $\eta_T$ is rotated by $f_T$.}
\end {center}
\end {figure}

In \cite{MR2486487, SchrammSheffieldGFF2}, it is shown that when one takes certain approximations $h^\eps$ of an instance $h$ of the GFF that are piecewise
linear on an $\eps$-edge-length triangular mesh, then conditioned on a zero chordal contour line of $h^\eps$ there is in some $\eps\to 0$ limiting sense a constant ``height gap'' between the expected heights immediately to one side of the contour line and those heights on the other.
%(Alternatively, one could approximate $h$ by letting $h^\eps$ be a convolution of $h$ with a smooth bump function %supported on $B_\eps(0)$.  However, the results in \cite{SchrammSheffieldGFF2} were not proved in that setting.)
We might similarly expect that if one looked at the expectation of $h^\eps$, given a chordal flow line $\eta^\eps$ of $e^{ih^\eps/\chi}$, there would be a constant order limiting height gap between the two sides, see Figure \ref{forwardftarrowfig}.

This suggests the form of $\h_t$ given in Theorem \ref{forwardcoupling}, which comes from taking \eqref{argleft} and \eqref{argright} and modifying the height gap between the two sides by adding a multiple of $\arg f_t$.  (As in \cite{SchrammSheffieldGFF2}, the size of the height gap --- and hence the coefficient of $\arg f_t$ in the definition of $\h_t$ --- is determined by the requirement that $\h_t(z)$ be a martingale in $t$, see Section \ref{couplingsection}.) Interestingly, the fact that winding may be ill-defined at a particular point on a fractal curve turns out to be immaterial.  It is the harmonic extension of the boundary winding values (the $\arg f_t'$) that is needed to define $\h_t$, and this is defined even for non-smooth curves.

\begin {figure}[htbp]
\begin {center}
\includegraphics [width=4in]{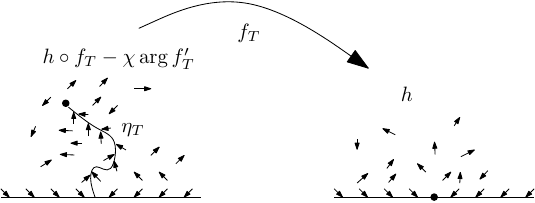}
\caption {\label{forwardftarrowfig} Forward coupling with arrows in $e^{ih/\chi}$ direction (sketch), illustrating the constant angle gap between the two sides of the curve $\eta$, constant angles along the positive and negative real axes, and random angles (not actually point-wise defined if $h$ is the GFF) in $\H \setminus \eta$.}
\end {center}
\end {figure}

The time-reversal of a flow line of $e^{ih^\eps/\chi}$ is a flow line of $e^{i(h^\eps/\chi + \pi)}$, which at first glance appears to imply that there should not be a height gap between the two sides (since if the left side were consistently
smaller for the forward path, then the right side would be consistently smaller for the reverse path).
To counter this intuition, observe that, in the left diagram in Figure \ref{flowfigure}, the left-going infinite
horizontal flow lines (at vertical heights of $k \pi$, $k$ odd) are ``stable'' in that the flow line
beginning at a generic point slightly off one of these lines will quickly converge to the line. The
right-going horizontal flow lines (at heights $k \pi$, $k$ even) are unstable.  In a stable flow
line, $h$ appears to generally be larger to the right side of the flow line and smaller to the left side.
It is reasonable to expect that a flow line of $e^{ih^\eps/\chi}$ started from a generic point would be approximately stable in that direction --- and in particular would look qualitatively different from the time reversal of a flow line of $e^{i(h^\eps/\chi+ \pi)}$ started from a generic point.

\subsection{Reverse coupling: planar maps and scaling limits} \label{ss.discretelimitdiscussion}
In this section, we conjecture a connection between path-decorated planar maps and SLE-decorated Liouville quantum gravity (in particular, the quantum-length-invariant decorated quantum wedge of
Theorem \ref{t.lengthstationary}).  We will explain the details in just one example based on the uniform spanning tree.  (Variants based on Ising and $O(n)$ and FK models on random planar maps --- or on random planar maps without additional decoration besides the chordal paths --- are also possible. Many rigorous results for percolation and the Ising model have been obtained for deterministic graphs in \cite{smirnov, MR2227824, smirnov:towards,smirnov-2007, CN,  2009arXiv0910.2045C} (and in many other papers we will not survey here), and one could hope to extend these results to random graphs.  One could also consider discrete random surfaces decorated by loops and in the continuum replace SLE decorations with CLE decorations \cite{MR2494457, 2010merged}.)  As mentioned earlier, we will see that the more surprising elements of Theorem \ref{t.lengthstationary} are actually quite natural from the discrete random surface point of view.

Let $G$ be a planar map with exactly $n$ edges (except that each edge on the outer face is counted as {\em half} an edge) and let $T$ be a subgraph consisting of a single boundary cycle, a chordal path from one boundary vertex $a$ to another boundary vertex $b$ that otherwise does not hit the boundary cycle, and a spanning forest rooted at this ``figure 8'' structure.  (See Figure \ref{planarmapfig}.)  Here $T$ is like the wired spanning tree (in which the entire boundary is considered to be one vertex), except that there is also one chord connecting a pair of boundary vertices.  What happens if we consider the uniform measure on all pairs $(G,T)$ of this type?  This model is fairly well understood combinatorially (tree-rooted maps on the sphere are in bijective correspondence with certain walks in $\Z^2$ --- see, e.g.,  \cite{MR0205882} as well as \cite{MR2285813} and the references therein --- and our model is a simple variant of this) and in particular, it follows from these bijections that the length of the
boundary of the outer face of this map will be of order $\sqrt n$ with high probability when $n$ is large.  Now, can we understand the scaling limit of the random pair $(G,T)$ as $n \to \infty$?

\begin {figure}[htbp]
\begin {center}
\includegraphics [width=3.5in]{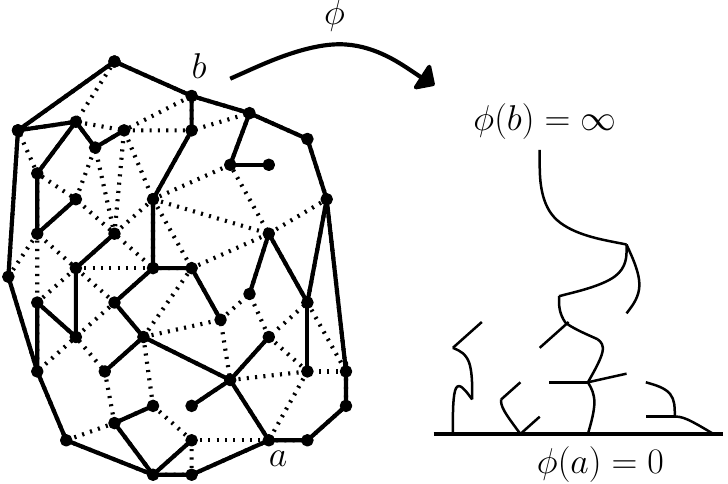}
\caption {\label{planarmapfig} Planar map with a distinguished outer-boundary-plus-one-chord-rooted spanning tree (solid black edges), with chord joining marked boundary points $a$ and $b$, plus image of tree under conformally uniformizing map $\phi$ to $\H$ (sketch).}
\end {center}
\end {figure}

There are various ways to pose this problem.  For example, one could consider $G$ as a metric space and aim for convergence in law w.r.t.\ the Gromov-Hausdorff metric on metric spaces.  The reader is probably aware that there is a sizable literature on the realization of a random metric space called the Brownian map as a Gromov-Hausdorff scaling limit of random planar maps of various types.  However, since this paper is concerned with the conformal structure of random geometr, we will try to phrase the the problem in a way that keeps track of that structure.

First, we would like to understand how to conformally map the planar map to the half plane, as in Figure \ref{planarmapfig}.  We may consider $G$ as embedded in a two-dimensional manifold with boundary in various ways, one of which we sketch here: first add an interior vertex to each face of $G$ and an edge joining it to each vertex of that face (as in Figure \ref{edgetosquarefig}).  Each interior edge of $G$ is now part of a quadrilateral (containing one vertex for each interior face of $G$ and one for each vertex of $G$) and we will endow that quadrilateral with the metric of a unit square $[0,1]\times [0,1]$.  Similarly, the triangle containing an exterior edge of $G$ is endowed with the metric of half a unit square (split on its diagonal, with the exterior edge as the hypotenuse).  When two squares or half squares share an edge, the points along that edge are identified with one another in a length preserving way.  We may view the collection of (whole and half) unit squares, glued together along boundaries, as a manifold (with isolated conical singularities at vertices whose number
of incident squares is not four) with a uniquely defined conformal structure (note that it is trivial to define a Brownian motion on the manifold, since it a.s.\ never hits the singularities).  We may choose a conformal map $\phi$ from this manifold to $\H$, sending $a$ to $0$ and $b$ to $\infty$, as sketched in Figure \ref{planarmapfig}.

\begin {figure}[htbp]
\begin {center}
\includegraphics [width=3.5in]{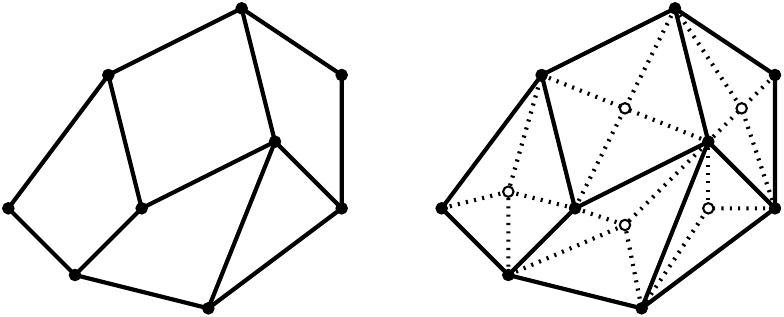}
\caption {\label{edgetosquarefig} An arbitrary planar map can be used to construct a collection of stitched-together unit squares and half unit squares.  The result is viewed as a two dimensional manifold with boundary.}
\end {center}
\end {figure}

This $\phi$ is determined only up to scaling, but we can fix the scaling in many ways.  We will do so by considering a number $k < n$ and requiring that the area of $\phi^{-1} (B_1(0))$ be equal to $k$.  Then $\phi$ determines a random measure on $\H$ (the image of the area measure on the manifold) in which the measure of $B_1(0)$ is deterministically equal to $k$; let $\mu_{n,k}$ denote this random measure divided by $k$, so that $\mu_{n,k}(B_1(0)) = 1$.
We expect that if one lets $n$ and $k$ tend to $\infty$ in such a way that $n/k$ tends to $\infty$, then the random measures $\mu_{n,k}$ will converge in law with respect to the metric of weak convergence on bounded subsets of $\H$ to the $\mu = \mu_h$ corresponding to the canonical description $h$ of the $(\gamma - 2/\gamma)$-quantum wedge of Theorem \ref{t.lengthstationary}.  (By compactness, the laws of the $\mu_{n,k}$ restricted to the closure of $B_1(0)$ have at least a subsequential limit.)  We similarly conjecture that $\nu_{n,k}$ --- defined to be $1/\sqrt{k}$ times the image of the manifold's boundary measure --- will converge in law to the corresponding $\nu_h$.  (We remark that one could alternatively formulate the conjecture by taking an infinite volume limit first --- i.e., letting $n$ go to infinity while keeping $k$ constant to define a limiting measure $\mu_{\infty, k} := \lim_{n \to \infty} \mu_{n,k}$.  This kind of infinite volume limit of random planar maps was constructed in \cite{MR2013797}.  One can subsequently take $k \to \infty$ and conjecture that the limit is $\mu_h$.  A similar conjecture in \cite{2008arXiv0808.1560D} was formulated in terms of infinite volume limits.)

We are currently unable to prove these conjectures, but related questions about Brownian motion on random surfaces have been explored in \cite{GR}, where it was shown that certain infinite random triangulations and quadrangulations ({\em without} boundaries) are parabolic (as opposed to hyperbolic) Riemann surfaces \cite{GR}.  (This is equivalent to showing that a Brownian motion visits each face infinitely often almost surely; see analogous discrete results in \cite{MR2013797}.)

Now let us make some more observations.  If we take $k$, $n$, and $n/k$ to be large and condition on $G$, $a$, and $b$, then what is the conditional law of $\phi(T)$, as depicted in Figure \ref{planarmapfig}?  The conditional law of $T$ itself is uniform among all valid $8$-rooted spanning forest configurations.  The physics literature frequently invokes a kind of ``conformal invariance Ansatz'' which suggests that this random path (and many other random sets in critical two dimensional statistical physics) should be a conformally invariant object.

In this case, we claim that the law of the chordal path should be approximately that of a chordal SLE$_2$ {\em even after we have conditioned on $G$, $a$, and $b$, which determine the measure $\mu_{n,k}$}.  The reason for our claim is that a related SLE$_2$ convergence result is obtained in \cite{MR2044671} in the case that $G$ is a deterministic lattice graph, and this was generalized substantially in \cite{2008arXiv0809.2643Y} where it was shown that if a graph can be embedded in the plane in such a way that simple random walk approximates Brownian motion, then the uniform spanning tree paths approximate a form of SLE$_2$.  We do not know whether the hypotheses of \cite{2008arXiv0809.2643Y} hold in our setting.  Brownian motion is conformally invariant, but it is not clear whether simple random walk on our random $G$ approximates Brownian motion on the corresponding quadrangulated manifold with high probability.  However, it seems very natural to conjecture that the hypotheses hold.  In any case, we stress the following: if our scaling limit conjecture holds, then the asymptotic independence of the chordal path from $\mu_{n,k}$ would be consistent with the independence of $\eta$ and $h$ in Theorem \ref{t.lengthstationary}.

Next let $D_1$ and $D_2$ be the wired-spanning-tree decorated manifolds to the left and right of the chordal path.  Note that once we condition on the length of the chordal path in $(G,T)$ and the number of edges on each side of it, the laws of $D_1$ and $D_2$ are independent of one another.  We might guess that the local behavior of $D_1$ and $D_2$ near $a$ would be approximately independent of these global numbers.  We expect a similar property to hold in the scaling limit, which would be consistent with the independence of the left and right quantum surfaces described in Theorem \ref{t.lengthstationary}.  (The idea of gluing together independent discrete surfaces in this manner has been explored in many works by Duplantier and others, beginning perhaps in \cite{1988PhRvL..61.1433D}.  The idea of gluing a whole series of discrete surfaces was used in \cite{MR1666816} to heuristically derive certain ``cascade relations'' via the KPZ formula.)

Finally, if we condition on the point $b$ and on $D_1$ and $D_2$, then the length of the path along which $D_1$ and $D_2$ are glued to each other is uniform among all possibilities (which range between $1$ and the minimum $M$ of the boundary lengths of the two $D_i$'s minus $1$).  In other words, once $D_1$ and $D_2$ and $b$ are all fixed, we can randomly decide how far to ``zip up'' or ``unzip'' these two surfaces (moving the vertex $a$ accordingly).  If $r$ is the random number of steps we zip, then $r$ and $r+m$ have approximately the same law (as long as $m/M$ is small).  We expect a similar property to hold in the scaling limit, which would be consistent with the quantum-length-zipper invariance described in Theorem \ref{t.lengthstationary}.

\section{Gaussian free field overview} \label{GFFoverview}

We refer the reader to \cite{Sh} for a survey of the Gaussian free field (GFF) and several additional references.  For
completeness, we include a short overview, closely following \cite{Sh, SchrammSheffieldGFF2}.  For the reader who is already familiar with the zero and free boundary GFF, it may be sufficient (to set notation) to read only the numbered equations in this section and the statement of Proposition \ref{p.GFFdef}.

\subsection{GFF definitions}
\subsubsection{Dirichlet inner product}
Fix a simply connected planar domain $D\subset \C$ (with $D \not = \C$).  Let $H_s(D)$ be the space of smooth, compactly supported functions on $D$,
and let $H(D)$ (sometimes denoted by $\H_0^1(D)$
or $W^{1,2}(D)$) be its Hilbert space closure under the Dirichlet inner product
$$(f_1,f_2)_\nabla := (2\pi)^{-1} \int_D \nabla f_1(z) \cdot \nabla f_2(z) dz.$$
Let $\psi$ be a conformal map from another domain $\widetilde D$ to $D$.  Then an elementary change of
variables calculation shows that $$\int_{\widetilde D} \nabla (f_1 \circ \psi ) \cdot
\nabla (f_2 \circ \psi )\,dx = \int_{D} (\nabla f_1 \cdot \nabla f_2)\,dx.$$
In other words, the Dirichlet inner product is invariant under conformal transformations.

We write $(f_1,f_2) = \int_D f_1(x)f_2(x)dx$ for the $L^2$ inner product on $D$.
We write $\|f\| := (f,f)^{1/2}$ and $\|f\|_\nabla := (f,f)_\nabla^{1/2}$.
If $f_1,f_2 \in H_s(D)$, then integration by parts gives \begin{equation} \label{e.intbyparts} (f_1,f_2)_\nabla = \frac{1}{2\pi}(f_1, -\Delta f_2).
\end{equation}

\subsubsection{Distributions and the Laplacian}

It is conventional to use $H_s(D)$ as a space of test functions.  This space is a topological vector space in which the topology is
defined so that $\phi_k \to 0$ in $H_s(D)$ if and only if
there is a compact
set on which all of the $\phi_k$ are supported and
the $m$th derivative of $\phi_k$ converges uniformly to zero for each integer $m \geq 1$.

A {\bf distribution} on $D$ is a continuous linear
functional on $H_s(D)$.
Since $H_s(D) \subset L^2(D)$, we
may view every $h \in L^2(D)$ as a distribution
$\density\mapsto (h,\density)$.
A {\bf modulo-additive-constant} distribution on $D$ is a continuous linear functional on the subspace of $H_s(D)$ consisting
of $\density$ for which $\int_D \density(z)dz = 0$.  We will frequently abuse notation and use $h$ --- or more precisely the map denoted by $\density \to (h,\density)$ --- to represent a general distribution (which is a functional of $\density$), even though $h$ may not correspond to an element of $L^2(D)$.  (Later, we will further abuse notation and use $\density$ to represent a non-smooth function or
a measure; in the latter case $(h, \density)$, when defined, will represent the integral of $h$ against that measure.)

We define partial derivatives and integrals of distributions in the
usual way (via integration by parts), i.e., for $\density \in H_s(D)$,
$$\bigl(\frac{\partial}{\partial x} h, \density\bigr):= -\bigl(h, \frac{\partial}{\partial x}
\density\bigr).$$  In particular, if $h$ is a distribution then
$\Delta h$ is a distribution defined by $(\Delta h, \density):= (h, \Delta \density)$.
When $h$ is a distribution and $\density \in H_s(D)$,
we also write $$(h, \density)_\nabla := \frac{1}{2\pi}(-\Delta h, \density) = \frac{1}{2\pi}(h, -\Delta \density).$$

When $x \in D$ is fixed, we let $\widetilde G_x(y)$ be the harmonic
extension to $y \in D$ of the function of $y$ on $\partial D$ given by
$-\log|y-x|$. Then {\bf Green's function in the domain $D$} is
defined by
$$G(x,y) = -\log|y-x| - \widetilde G_x(y).$$
When $x \in D$ is fixed, Green's function may be viewed as a
distributional solution of $\Delta G(x,\cdot)= -2 \pi
\delta_x(\cdot)$ with zero boundary conditions \cite{Sh}.
It is non-negative for all $x, y \in D$ and $G(x,y) = G(y,x)$.

For any $\density \in H_s(D)$, we write $$-\Delta^{-1} \density := \frac{1}{2\pi}\int_D G(\cdot ,y)\,\density (y)\,dy.$$  This is a $C^\infty$
function in $D$ whose Laplacian is $-\density$.  Indeed, a similar definition can be made if $\density$ is any signed measure (with finite positive and finite negative mass) rather than a smooth function.  Recalling \eqref{e.intbyparts}, if $f_1 = -2\pi \Delta^{-1} \density_1$ then $(h, f_1)_\nabla = (h, \density_1)$, and
similarly if $f_2 = -2\pi \Delta^{-1} \density_2$.  Now
$(f_1, f_2)_\nabla = (\density_1,-2\pi \Delta^{-1} \density_2)$ describes a covariance that can (by the definition of $-\Delta^{-1} \density_2$ above) be
rewritten as
\begin{equation}\label{e.greencovariance}\Cov\bigl((h, \density_1), (h, \density_2) \bigr) = \int_{D \times D} \density_1(x)\, G(x,y)\,
\density_2(y) \, dx \, dy.\end{equation}

If $\density \in H_s(D)$, may define the map $(h, \cdot)$ by $(h, \density) :=
(h, -2\pi \Delta^{-1} \density)_\nabla$, and this definition describes a distribution \cite{Sh}.  (It is not hard
to see that $-2\pi \Delta^{-1} \density \in H(D)$, since its Dirichlet
energy is given explicitly by \eqref{e.greencovariance}.)

\subsubsection{Zero boundary GFF}

An instance of the GFF with zero boundary conditions on $D$ is a
random sum of the form $h = \sum_{j = 1}^{\infty} \alpha_j f_j$
where the $\alpha_j$ are i.i.d.\ one-dimensional standard (unit
variance, zero mean) real Gaussians and the $f_j$ are an orthonormal
basis for $H(D)$.  This sum almost surely does not converge
within $H(D)$ (since $\sum_{j=1}^\infty |\alpha_j|^2$ is a.s.\ infinite).  However, it does converge almost surely within the space
of distributions --- that is, the limit $(\sum_{j=1}^{\infty} \alpha_j f_j, \density)$
almost surely exists for all $\density \in H_s(D)$, and the limiting value as a function
of $\density$ is almost surely a continuous functional on $H_s(D)$ \cite{Sh}.
We may view $h$ as a sample from the measure space
$(\Omega, \mathcal F)$ where $\Omega = \Omega_D$ is the set of distributions
on $D$ and $\mathcal F$ is the smallest $\sigma$-algebra that makes
$(h, \density)$ measurable for each $\density \in H_s(D)$, and we sometimes denote by $dh$
the probability measure which is the law of $h$.  If $f_j$ are chosen in $H_s(D)$, then the values $\alpha_j$ are clearly $\mathcal F$-measurable.
In fact, for any $f \in H(D)$ with $f = \sum_j \beta_j f_j$ the sum $(h,f)_\nabla := \sum_j \alpha_j \beta_j$ is a.s.\ well defined and is
a Gaussian random variable with mean zero and variance $(f,f)_\nabla$.

\subsection{Green's functions on $\C$ and $\H$: free boundary GFF} \label{s.green}
The GFF with free boundary conditions is defined the same way as the GFF with zero boundary conditions except that
we replace $H_s(D)$ with the space of all smooth functions with gradients in $L^2(D)$ (i.e., we remove the requirement that the functions be compactly supported).  However, to make the correspondingly defined $H(D)$ a Hilbert space, we have to consider functions
only modulo additive constants (since all constant functions have norm zero).  On the whole plane $\C$, we may define the Dirichlet inner product on the Hilbert space closure $H(\C)$ of the space of such functions defined modulo additive constants.

Generally, given a compactly supported $\density$ (or more generally, a signed measure), we can write
\begin{equation}\label{e.deltainverse}-\Delta^{-1} \density(\cdot) := \frac{1}{2\pi} \int_\C G(\cdot, y) \density(y)dy,\end{equation}
with $G(x,y) = -\log|x-y|$.

As before, for compactly supported $f$ and $g$, we have $(f,g)_\nabla = \frac{1}{2\pi}(f, -\Delta g)$ by integration by parts, and moreover $(f,-\Delta^{-1} \density)_\nabla = \frac{1}{2\pi}(\density, f)$.  The same holds for bounded and not necessarily compactly supported smooth functions $f$ and $g$ if the gradient of $-\Delta^{-1} \density$ tends to zero at infinity, which in turn happens if and only if $\int_\C \density(z) dz = 0$.

If $\int_\C \density(z) dz \not = 0$ then the Dirichlet energy of $-\Delta^{-1} \density$ will be infinite and moreover $(h, \density)$ will not be independent of the additive constant chosen for $h$.  (If we view $\C$ as a Riemann sphere, then $\int_\C \density(z) dz \not = 0$ can also be interpreted as the statement that the Laplacian of $-\Delta^{-1} \density$ has a point mass at $\infty$.)  When $h$ is the free boundary GFF on $\C$, we will thus define the random variables $(h,\density)$ only if the integral of $\density$ over $\C$ is zero.  If $\density_1$ and $\density_2$ each have total integral zero, we may write

\begin{equation} \label{e.Ccovariance} \Cov ((h, \density_1), (h, \density_2)) = \int_{\C \times \C} \density_1(x) G(x,y) \density_2(y)dxdy.\end{equation}

Using $z \to \bar z$ to denote complex conjugation, we define, for smooth functions $h \in H(\C)$, the pair of projections
\begin{eqnarray*} h^\odd(z) & := &\frac{1}{\sqrt{2}} (h(z) - h(\bar z)), \\
h^\even(z)& := &\frac{1}{\sqrt{2}} (h(z) + h(\bar z)).
\end{eqnarray*}
If $h$ is an instance of the free boundary GFF on $\C$, we may still define $h^\odd$ and $h^\even$ as projections of $h$ onto complementary orthogonal subspaces.  Their restrictions to $\H$ are instances of the zero boundary GFF and free boundary GFF, respectively on $\H$.  For $\rho$ supported on $\H$ we write (for $z \in \mathbb C$) $\density^*(z) := \density(\overline z)$.  Then we have by definition
\begin{eqnarray*} (h^\odd,\density) = \frac{1}{\sqrt 2} (h, \density - \density^*) \\
 (h^\even,\density) = \frac{1}{\sqrt 2} (h, \density + \density^*). \end{eqnarray*}
Note that $(h^\even, \density)$ is only defined if the total integral of $\density$ is zero, while $(h^\odd, \density)$ is defined without that restriction (since in any case the total integral of $\density - \density^*$ will be zero).

For $\density_1$ and $\density_2$ supported on $\H$ we now compute the following (first integral taken over $\C \times \C$, second over $\H \times \H$): \begin{align} \Cov\bigl( (h^\odd, \density_1), (h^\odd, \density_2) \bigr) &=  \frac{1}{2} \int (\density_1(x) - \density^*_1(x)) \log|x-y| (\density_2(y) - \density^*_2(y))dxdy \notag \\
 & = \int \density_1(x)G^{\H_0}(x,y)\density_2(y)dxdy, \label{e.H0covariance} \end{align}
where $G^{\H_0}(x,y) := \log|x - \bar y| - \log|x-y|.$  Similarly (first integral over $\C \times \C$, second over $\H \times \H$), \begin{align} \Cov\bigl( (h^\even, \density_1), (h^\even, \density_2) \bigr) &= \frac{1}{2} \int (\density_1(x) + \density^*_1(x)) \log|x-y| (\density_2(y) + \density^*_2(y))dxdy \notag \\
& =  \int \density_1(x)G^{\H_F}(x,y)\density_2(y)dxdy, \label{e.Hfreecovariance} \end{align}
where $G^{\H_F}(x,y) := -\log|x - \bar y| - \log|x-y|$.

\subsection{GFF as a continuous functional} \label{ss.GFFascontinuousfunctional}

Note that we could have used \eqref{e.H0covariance} and \eqref{e.Hfreecovariance} to give an alternate and more direct definition of the zero and free boundary Gaussian free fields on $\H$.  Here \eqref{e.H0covariance} and \eqref{e.Hfreecovariance} define inner products on the space of functions $\density$ on $\H$.  They are well defined when $\density_1$ and $\density_2$ are smooth and compactly supported functions on $\H$ (each with total integral zero in the case of \eqref{e.Hfreecovariance}).  By taking the Hilbert space closure of functions of this type, we get a larger space of $\density$, which correspond to Laplacians of elements of $H(\H)$, and which cannot all be interpreted as functions on $\H$.  For example, the $\density$ for which $(h,\density)$ is $h_\eps(z)$, the mean value of $h$ on $\partial B_\eps(z)$, is not a function, though it can be interpreted as a measure --- a uniform measure on $\partial B_\eps(z)$ --- and the inner products \eqref{e.H0covariance} and \eqref{e.Hfreecovariance} still make sense when $\density_1(z)dz$ and $\density_2(z)dz$ are replaced with more general measures, as do the definitions of $-\Delta^{-1} \density_1$ and $-\Delta^{-1} \density_2$.

The $(h,\density)$ are centered jointly Gaussian random variables, defined for each $\density$ in this Hilbert space, with covariances given by the inner products \eqref{e.H0covariance} and \eqref{e.Hfreecovariance} (which can be defined on the entire Hilbert space).  For each particular $\density$ in this Hilbert space, $(h, \density)$ is a.s.\ well defined and finite; however, $\density \to (h, \density)$ is almost surely not a continuous linear functional defined on the entire Hilbert space, since a.s.\ $h \not \in H(\H)$.

In addition to the description of $h$ as a distribution above, there are various ways to construct a space of $\density$ values --- a subset of the complete Hilbert space --- endowed with a topology that makes $\density \to (h, \density)$ almost surely continuous.  For example, the map $h \to h_\eps(z)$ is an a.s.\ H\"older continuous function of $\eps$ and $z$ \cite{2008arXiv0808.1560D}.  Also, the zero boundary GFF can be defined as a random element of $(-\Delta)^{-\eps} L^2(D)$ for any $\eps > 0$, and is hence a continuous linear function on $(-\Delta)^{\eps} L^2(D)$, if $D$ is bounded. (See \cite{Sh} for definitions and further discussion of fractional powers of the Laplacian in this context.)  Also, as mentioned earlier, both the free and zero boundary GFFs can be understood as random distributions \cite{Sh}.

The issues that come up when defining $\density \to (h, \density)$ as a continuous function on some topological space of $\density$ values are the same ones that come up when rigorously constructing a Brownian motion $B_t$: one can give the joint law of $B_t$ for any finite set of $t$ values explicitly by specifying covariances, and this determines the law for any fixed countable set of $t$ values, but one needs to overcome some (mild) technicalities in order to say ``$B_t$ is almost surely a continuous function.'' Indeed, if one uses the smallest $\sigma$-algebra in which $B_t$ is measurable for each fixed $t$, then the event that $B_t$ is continuous is not even in the $\sigma$-algebra.

On the other hand, if we are given a construction that produces a random continuous function with the right finite dimensional marginals, then it must be a Brownian motion.  A standard fact (proved using characteristic functions and Fourier transforms) states that a random variable on a finite dimensional space is a centered Gaussian with a given covariance structure if and only if all of its one dimensional projections are centered Gaussians with the appropriate variance. Thus, to establish that $B_t$ is a Brownian motion, it is enough to show that each finite linear combination of $B_t$ values is a (one-dimensional) centered Gaussian with the right variance.  The following proposition formalizes the analogous notion in the GFF context.  It is a standard and straightforward result about Gaussian processes (see \cite{Sh} for a proof in the zero boundary case; the free boundary case is identical):

\begin{proposition} \label{p.GFFdef}
The zero boundary GFF on $\H$ is the only random distribution
$h$ on $\H$ with the property that for each
$\density \in H_s(\H)$ the random variable $(h,\density)$
is a mean-zero Gaussian with variance given by \eqref{e.H0covariance} (with $\density_1=\density_2=\density$).
Similarly, the free boundary GFF is the only random modulo-additive-constant
distribution on $\H$ with the property that for each
$\density \in H_s(\H)$ with $\int_{\H} \density(z) dz = 0$ the random variable $(h,\density)$
is a mean-zero Gaussian with variance given by \eqref{e.Hfreecovariance}.
\end{proposition}

In our proofs of Theorem \ref{forwardcoupling} and Theorem \ref{reversecoupling} in Section \ref{couplingsection}, we will first construct a random distribution in the manner prescribed by the theorem statement and then check the laws of the one dimensional projections (which determine the laws of the finite and countably infinite dimensional projections) to conclude by Proposition \ref{p.GFFdef} that it must be the GFF.

We remark that knowing $h$ as a distribution determines the values of $\alpha_j$ in a basis expansion $h = \sum_j \alpha_j f_j$, as long as the $- \Delta f_j$ are sufficiently smooth.  This in turn determines the value of $h_\eps(z)$ almost surely for a countable dense set of $\eps$ and $z$ values, which determines the values for all $\eps$ and $z$ by the almost sure continuity of $h_\eps(z)$ \cite{2008arXiv0808.1560D}.  This is convenient because it means that $h$, as a distribution, a.s.\ determines $(z,\eps) \to h_\eps(z)$ as a function, which in turn determines $\mu_h$ and $\nu_h$.  (We could alternatively have defined $h_\eps(z)$ --- and hence $\mu_h$ and $\nu_h$ --- using weighted averages of $h$ defined by integrating against smooth bump functions on $B_\eps(z)$ instead of averages on $\partial B_\eps(z)$.  Though we won't do this here, one can easily construct measures this way that are almost surely equivalent to $\mu_h$ and $\nu_h$.)

\section{Coupling the GFF with forward and reverse SLE} \label{couplingsection}

\subsection{Proofs of coupling theorems} \label{subsec::couplingproofs}

This section will simultaneously prove Theorem \ref{forwardcoupling} and Theorem \ref{reversecoupling}.  It is instructive
to prove them together, and we will put the relevant calculations in tables, with those for the forward SLE coupling
of Theorem \ref{forwardcoupling} on the left side and those for the reverse SLE coupling of Theorem \ref{reversecoupling} on the right.

Now, using the language of stochastic differential equations and applying \Ito/'s formula in the case
$W_t = \sqrt{\kappa} B_t$, we compute the time derivatives of the four processes $f_t(z)$, $\log f_t(z)$, $f_t'(z)$, and $\log
f_t'(z)$ in both forward and reverse SLE settings.  Here $f_t'(z)$ denotes the spatial derivative $\frac{\partial}{\partial z} f_t$.  (Similar calculations appear in \cite{SchrammSheffieldGFF2} in the case $\kappa = 4$.)

\renewcommand{\arraystretch}{2}

\vspace{.2 in}
\begin{center}
\begin{tabular} {|l|l|}
\hline
{\bf FORWARD FLOW SLE} & {\bf REVERSE FLOW SLE} \\
\hline
$df_t(z) = \frac{2}{f_t(z)}dt - \sqrt{\kappa} dB_t$ & $df_t(z) = \frac{-2}{f_t(z)}dt - \sqrt{\kappa} dB_t$  \\
\hline
$d \log f_t(z)=\frac{(4-\kappa)}{2f_t(z)^2}dt - \frac{\sqrt{\kappa}}{f_t(z)}dB_t$ & $d \log f_t(z) =\frac{-(4+\kappa)}{2f_t(z)^2}dt - \frac{\sqrt{\kappa}}{f_t(z)}dB_t$ \\
\hline
$d f_t'(z) = \frac{-2f_t'(z)}{f_t(z)^2}dt$ & $d f_t'(z) = \frac{2f_t'(z)}{f_t(z)^2}dt$ \\
\hline
$d \log f_t'(z) = \frac{-2}{f_t(z)^2}dt$ & $d \log f_t'(z) = \frac{2}{f_t(z)^2}dt$ \\
\hline
\end{tabular}
\end{center}
\vspace{.2 in}

We next define the martingales $\h_t$ in both settings and compute their stochastic derivatives. The purpose of the stochastic calculus below is to show that  the quantities $(\h_t, \density)$ are continuous local martingales (the fact that they are martingales will become apparent later) and to explicitly computing their quadratic variations, so that they can be understood as Brownian motions subject to an explicit time change. Ultimately, we will use the properties of these Brownian motions to establish couplings between SLE and the Gaussian free field.

Note that while the two columns have differed only in signs until now, the definitions of $\h_t$ below will diverge in that one involves the imaginary and one the real part of $\h^*_t$.  We will write $\gamma := \sqrt{\min(\kappa, 16/\kappa)} \in [0,2]$.

\vspace{.2 in}
\begin{center}
\begin{tabular} {|l|l|}
\hline
{\bf FORWARD FLOW SLE} & {\bf REVERSE FLOW SLE} \\
\hline
$\chi:= \frac{2}{\sqrt{\kappa}} - \frac{\sqrt{\kappa}}{2}$ & $Q: = \frac{2}{\sqrt{\kappa}} + \frac{\sqrt{\kappa}}{2} = \frac{2}{\gamma} + \frac{\gamma}{2}$ \\
\hline
$\h^*_t(z) := \frac{-2}{\sqrt{\kappa}}  \log f_t(z) - \chi \log f_t'(z)$ & $\h^*_t(z) := \frac{2}{\sqrt{\kappa}}  \log f_t(z) + Q \log f_t'(z)$ \\
\hline
$d \h^*_t(z) =  \frac{2}{f_t(z)} dB_t $ & $d \h^*_t(z) =  \frac{-2}{f_t(z)} dB_t $ \\
\hline
$\h_t(z) := \Im \h^*_t(z) $ & $\h_t(z) := \Re \h^*_t(z) $ \\
\hline
$d \h_t(z) = \Im \frac{2}{f_t(z)} dB_t $ & $d \h_t(z) =  \Re \frac{-2}{f_t(z)} dB_t $ \\
\hline
\end{tabular}
\end{center}
\vspace{.2 in}

Before continuing with the calculation, we make several remarks.

\begin{remark}The form of $d\h_t(z)$ in the forward case is significant. At time
$t=0$, the function $-2\Im(f_t(z)^{-1})$ is simply $-2\Im (z^{-1})$.
This is a positive harmonic function whose level sets are circles in
$\H$ that are tangent to $\mathbb R$ at the origin.  It is a multiple of the so-called Poisson kernel, and it
is a derivative of the Green's function $G(y,z) = G^{\H_0}(y,z) = \log \left|
\frac{y-\bar z}{y-z} \right|$ in the following sense:
$$[\frac{\partial}{\partial s} G(is,z)]_{s=0} = \frac{\partial}{\partial s} \left|
\frac{z+is}{z-is} \right|_{s=0} = \Re \frac{2iz}{|z^2|} =  2\Im(z^{-1}).$$  Intuitively, the
value of $-2 \Im(f_t(z)^{-1})$ represents the harmonic measure of the tip of $\eta_t := \eta([0,t])$ as seen
from the point $z$.  Roughly speaking, as one makes observations of the GFF at points near the tip of $\eta_t$, the conditional expectation of $h$
goes up or down by multiples of this function. \end{remark}

\begin{remark} Also, in the forward case, $\h_0$ is the harmonic function on $\H$ with boundary
conditions $-2\pi/\sqrt{\kappa}$ on the negative real axis and $0$ on the
positive real axis.  We could have (for sake of symmetry) added a constant to $\h_0$ (and general $\h_t$) so that $\h_0$ is equal to $-\lambda$ on the negative real axis
and $\lambda$ on the positive real axis, where $\lambda := \pi/\sqrt{\kappa}$.
Observe that when $\kappa = 4$, we have $\chi =
0$ and hence each $\h_t$ would be the harmonic function on $\H \backslash
\eta_t$ with boundary conditions $-\lambda$ on the left side of the
tip of $\eta_t$ and $\lambda$ on the right side.  In this case, the $\lambda = \pi/2$
is the same (up to a $\sqrt{2\pi}$ factor stemming from a different choice of normalization for the
GFF) as the value $\lambda = \sqrt{\pi/8}$ obtained in \cite{SchrammSheffieldGFF2}.
\end{remark}

\begin{remark} In the reverse case, the expression for $d\h_t$ has $\Re \frac{-2}{f_t(z)}$ in place of $\Im \frac{2}{f_t(z)}$.  Intuitively, at time zero,
when one observes what $f_t$ looks like for small $t$, one learns something about the {\em difference} between $h$ just to the left of $0$ and
$h$ just to the right of $0$.  (It is this difference that determines the ratio of the $\nu_h$ densities to the left and to the right of zero, which is what determines how the zipping-up should behave in the short term.)
The conditional expectation of $h$ thus changes by a small multiple of $\Re \frac{2}{f_t(z)}$, which is negative
on one side of the imaginary axis and positive on the other side.  Unlike $\Im \frac{2}{f_t(z)}$, the function $\Re \frac{2}{f_t(z)}$ is non-zero on $\R$. \end{remark}

We use $\langle X_t, Y_t \rangle$ to denote cross variation between processes $X_t$ and $Y_t$ up to time $t$, so that $\langle X_t, X_t \rangle$ represents the quadratic variation of the process $X_t$ up to time $t$.  (The cross variation $\langle X_t, Y_t \rangle$ is also often written as $\langle X,Y \rangle_t$.)  In both forward and reverse flow settings, $\h_t(z)$ is a continuous local martingale for each fixed $z$ and is thus a Brownian motion under the quadratic variation parameterization, which we can give explicitly:

\vspace{.2 in}
\begin{center}
\begin{tabular} {|l|l|}
\hline
{\bf FORWARD FLOW SLE} & {\bf REVERSE FLOW SLE} \\
\hline
$C_t(z) := \log \Im f_t(z) - \Re \log f_t'(z)$ & $C_t(z) := - \log \Im f_t(z) - \Re \log f_t'(z)$ \\
\hline
$d\langle \h_t(z), \h_t(z) \rangle = -d C_t(z)$ & $d\langle \h_t(z), \h_t(z) \rangle = -d C_t(z)$ \\
\hline
\end{tabular}
\end{center}
\vspace{.2 in}

If $z$ is a point in a simply connected domain $D$, and $\phi$ conformally maps the unit disc to $D$, with $\phi(0) = z$, then we refer to the quantity $|\phi'(0)|$ as the {\em conformal radius} of $D$ viewed from $z$.  If, in the above definition of conformal radius, we replaced the unit disc with $\H$ and $0$ with $i$, this would only change the definition by an additive constant.  Thus, in the forward flow case, $C_t(z)$ is (up to an additive constant) the log of the conformal radius of $\H \setminus \eta([0,t])$ viewed from $z$.   In both cases $\h_t(z)$ is a Brownian motion when parameterized by the time parameter $-C_t(z)$ (which is increasing as a function of $t$).  The fact that $d\langle \h_t(z), \h_t(z) \rangle = -d C_t(z)$ may be computed directly via \Ito/'s formula but it is also easy to see by taking $y \to z$ in the formulas for $\langle \h_t(y), \h_t(z)\rangle$ and $-dG_t(y,z)$ that we will give below.

We will now show that weighted averages of $\h_t$ over multiple points in $\H$ are also continuous local martingales (and hence Brownian motions when properly parameterized).
The calculation will make use of the function $G(y,z)$, which we take to be the zero-boundary Green's function $G^{\H_0}(y,z)$ on $\H$ in the forward case and the free-boundary Green's function $G^{\H_F}(y,z)$ in the reverse case.

Now write $G_t(y,z) = G(f_t(y), f_t(z))$ in the reverse case.
In the forward case, write $G_t(y,z) = G(f_t(y), f_t(z))$ when $y$ and $z$ are both
in the infinite component of $\H \backslash \eta_t$ --- otherwise,
let $G_t(y,z)$ be the limiting value of $G_s(y,z)$ as $s$
approaches the first time at which one of $y$ or $z$ ceases to be in
this infinite component. The reader may check that for fixed $y$ and
$z$, this limit exists almost surely when $4 < \kappa < 8$: it is
equal to zero when $y$ and $z$ are in different connected components
of $\H \backslash \eta_t$, and when $y$ and $z$ lie in the same
component, it is simply the Green's function of $y$ and $z$ on this
bounded domain.  Now we let $\density$ be a smooth compactly supported function on $\H$ (which we will assume has mean zero in the reverse case) and do some more calculations.

\vspace{.2 in}
\begin{center}
\begin{tabular} {|l|l|}
\hline
{\bf FORWARD FLOW SLE} & {\bf REVERSE FLOW SLE} \\
\hline
$G(y,z) := \log|y-\bar z| - \log |y- z|$ & $G(y,z) := - \log |y-z| - \log |y-\bar z|$ \\
\hline
$G_t(y,z) := G(f_t(y), f_t(z))$ & $G_t(y,z) := G(f_t(y), f_t(z))$ \\
\hline
$dG_t(y,z) = -\Im \frac{2}{f_t(y)} \Im\frac{2}{f_t(z)} dt$ & $dG_t(y,z) = -\Re \frac{2}{f_t(y)} \Re\frac{2}{f_t(z)} dt$ \\
\hline
$d\langle \h_t(y), \h_t(z)\rangle = -dG_t(y,z)$ & $d\langle \h_t(y), \h_t(z)\rangle = -dG_t(y,z)$ \\
\hline
$E_t(\density) := \int_{\H} \density(y) G_t(y,z) \density(z) dydz$ & $E_t(\density) := \int_{\H} \density(y) G_t(y,z) \density(z) dydz$\\
\hline
$d\langle (\h_t,\density), (\h_t,\density)\rangle = -dE_t(\density)$ & $d\langle (\h_t,\density), (\h_t,\density)\rangle = -dE_t(\density)$ \\
\hline
\end{tabular}
\end{center}
\vspace{.2 in}

Each of the equations above comes from a straightforward \Ito/ calculation.  To explain their derivation, we begin by expanding the $dG_t$ computation in the forward case (the reverse case is similar):

{\allowdisplaybreaks
\begin{eqnarray*} d G_t(x,y) &=& - d\, \Re \log [f_t(x) - f_t(y)] + d\, \Re \log [f_t(x) -
\overline{f_t(y)}] \\
&=& - 2\,\Re \frac {f_t(x)^{-1} - f_t(y)^{-1}}{f_t(x) - f_t(y)}\,dt
+ {}
\\ & &\qquad 2
\,\Re \frac { f_t(x)^{-1} - \overline{f_t(y)^{-1}}}{f_t(x) - \overline{f_t(y)}}\,dt \\
&=& 2\, \Re\bigl( f_t(x)^{-1}f_t(y)^{-1}\bigr)\,dt  - 2\, \Re\bigl( f_t(x)^{-1} \bigl(\overline{f_t(y)}\bigr)^{-1}\bigr)\,dt \\ &=&  2\, \Re \bigl( i\, f_t(x)^{-1}\, \Im[2 f_t(y)^{-1}]\bigr) \,dt \\
&= & \,- \Im \frac{2}{f_t(x)}\, \Im \frac{2}{f_t(y)}\,dt\,.  \end{eqnarray*}}The fact that $d\langle \h_t(y), \h_t(z)\rangle = -dG_t(y,z)$ is then immediate from our calculation of $d\h_t$.

The fact that $d\langle (\h_t,\density), (\h_t,\density)\rangle = -dE_t(\density)$ is essentially a Fubini calculation but it requires some justification. First, we claim that the $(\h_t, \density)$ are continuous martingales.  We begin by considering $\h_t(z)$ for a fixed $z$ in the support of $\density$.  We have shown above that the quantity $\h_t(z)$ is a Brownian motion under a certain parameterization.  In the reverse case, the Loewner evolution gives that $|\frac{\partial}{\partial t}C_t(z)|$ is uniformly bounded above for $z$ in the support of $\density$ and for all times $t$.    (Note that $\Im f_t(z)$ is strictly increasing in $t$.)  This immediately implies that $\h_t(z)$ is a martingale (not merely a local martingale) because for each $z$ and $t$, $\h_t(z)$ represents the value of a Brownian motion stopped at a random time that is strictly less than a constant times $t$.  The fact that the expectation of $\h_t(z)$ --- given the filtration up to time $s<t$ --- is $\h_s(z)$ is then immediate from the optional stopping theorem.

In the forward case, one obtains something similar by noting that the law of the conformal radius $r$ of $z$ in $\H \setminus \eta([0,t])$ has a power law decay as $r \to 0$ --- i.e., the probability that $-C_t(z) > c$ decays exponentially in $c$, and is in fact bounded by an exponentially decaying function that is independent of $z$, for $z$ in the support of $\density$.  (A precise description of the law of the conformal radius at time infinity appears as the main construction in \cite{MR2491617}.)  This implies that $\h_t(z)$ is a Brownian motion stopped at a time whose law decays exponentially (uniformly over $z$ in the support of $\density$) which is again enough to apply the optional stopping theorem and conclude that $\h_t(z)$ is martingale.  In both cases, we obtain that for any $t$, the probability distribution function for $|\h_t(z)|$ decays exponentially fast, uniformly for $z$ in the support of $\density$.  In both cases, we also see that (for any fixed $t$), $\h_t(z)$ is an $L^1$ function of $z$ and the probability space, which allows us to use Fubini's theorem and conclude that the $(\h_t, \density)$ are martingales.

Let $L^p_{\mathrm{loc}}$ denote the set of $\psi$ for which the integral of $|\psi|^p$ over every compact subset of $\H$ is finite.  The exponential decay above implies that $\h_t$ is almost surely in $L^1_{\mathrm{loc}}$, since the expected integral of $|\h_t|$ over any compact set is finite.  (Note that we can define $\h_t$ arbitrarily on the measure zero set $\eta([0,t])$ without affecting the definition of $\h_t$ as an element of $L^1_{\mathrm{loc}}(\H)$.)  In fact, since $\mathbb E |\h_t(z)|^p$ is bounded uniformly for $z$ in a compact set, it follows that $\h_t$ is almost surely in $L^p_{\mathrm{loc}}(\H)$ for any $p \in (1, \infty)$.  The fact that $\h_t$ is almost surely in $L^1_{\mathrm{loc}}$ also implies that it can be understood as a random distribution on $\H$.

Moreover, \begin{equation} \label{e.supht} \sup_{s \in [0,t]} |\h_s(z)|\end{equation} also has, by Doob's inequality, a law that decays exponentially, uniformly in $z$.  Thus \eqref{e.supht} also belongs a.s.\ to $L^p_{\mathrm{loc}}(\H)$ for any $p \in (1, \infty)$.
From this and the a.s.\ continuity of SLE it follows that $(\h_t, \density)$ is a.s.\ continuous in $t$.  (This continuity is obvious in the reverse case; in the forward case, it is also obvious if one replaces $\density$ by $\density_\eps$, which we define to be zero on an $\eps$ neighborhood of $\eta$ and $\density$ elsewhere.  The fact that \eqref{e.supht} belongs to $L^p_{\mathrm{loc}}(\H)$ implies that the $(\h_t, \density_\eps)$ converge to $(\h_t, \density)$ uniformly, for almost all $\eta$, and a uniform limit of continuous functions is continuous.)

Now we can show $d\langle (\h_t,\density), (\h_t,\density)\rangle = -dE_t(\density)$, as noted in \cite{SchrammSheffieldGFF2}, either via a stochastic Fubini's theorem (see e.g.\ \cite[\S IV.4]{MR1037262}) or by using the following simpler approach proposed in private communication by Jason Miller.

\begin {figure}[htbp]
\begin {center}
\includegraphics [width=3.5in]{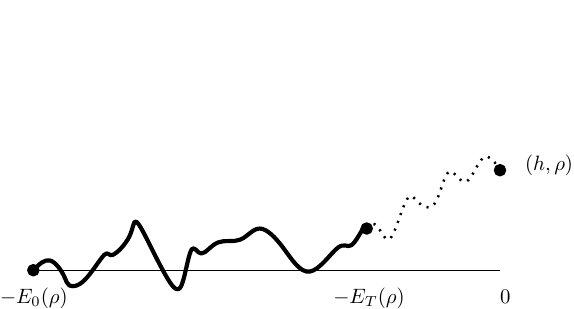}
\caption {\label{brownianenergyfig} The pair $\bigl(-E_t(\density), (\mathfrak h_t, \density) \bigr)$ traces the graph of a Brownian motion (solid curve) as $t$ ranges from $0$ to $T$.  Conditioned on this, the difference between $(h,\density)$ and $(\mathfrak h_T, \density)$ is a centered Gaussian of variance $E_T(\density)$.  Choosing  $(h,\density)$ to be $(\mathfrak h_T, \density)$ plus a Gaussian of this variance is equivalent to continuing the Brownian motion parameterized by $-E_t(\density)$ time (solid curve) all the way to time zero (dotted curve) and letting $(h,\density)$ be its value at time zero.}
\end {center}
\end {figure}

First note that $\langle (\h_t, \density_1), (\h_t, \density_2) \rangle$ is characterized by the fact that
$$(\h_t, \density_1)(\h_t, \density_2)- \langle (\h_t, \density_1), (\h_t, \density_2) \rangle$$
is a local martingale.  Thus it suffices for us to show that
\begin{equation}\label{e.intqvmartingale}
(\h_t, \density_1)(\h_t, \density_2) + \int \density_1(x)\, \density_2(y)\, G_t(x,y)\,dx\,dy
\end{equation}
is a martingale.  We know from the above calculations that $$\h_t(x)\h_t(y) + G_t(x,y)$$ is a martingale for fixed $x$ and $y$ in $\H$.  Since $G_t(x,y)$ is non-increasing and the $\h_t(z)$ have laws that decay exponentially, uniformly in $z$, we can use Fubini's theorem to conclude that \eqref{e.intqvmartingale} is a martingale.
Thus we have that $(\h_t,\density)$ is a Brownian motion when parameterized by time $-E_t(\density)$.  To complete the proofs of Theorems \ref{forwardcoupling} and \ref{reversecoupling}, recall that in the theorem statements $\widetilde h$ denotes an instance of the free boundary GFF on $\H$, and note that
since each $(\h_T + \widetilde h \circ f_T, \density)$ is a sum
of a standard Brownian motion stopped at time $E_0(\density) - E_{T}(\density)$ and a conditionally independent Gaussian of variance $E_{T}(\density)$, it has the same law as a Gaussian of variance $E_0(\density)$ and mean $(\h_0, \density)$. (See Figure \ref{brownianenergyfig}.)  For future reference, we note that in the reverse flow case one may integrate the expression for $d\h_t(z)$ above to find (using the stochastic Fubini's theorem) that
\begin{equation} \label{e.htdensity}
d(\h_t, \density) = \bigl(-2 \Re (f_t)^{-1} , \density\bigr) dB_t.
\end{equation}

\begin{remark} The statement of Theorem \ref{forwardcoupling} excluded the case $\kappa \geq 8$, since \SLEk/ is space-filling in that case and $\h_t$ cannot be defined as a function almost everywhere.  Nonetheless, we may still define $(\h_t, \density)$ to be the solution to the stochastic differential equation $d(\h_t, \density) = \bigl(- 2 \Im (f_t)^{-1}, \density \bigr) dB_t$.  In this case, the calculations above
again yield that
$d\langle (\h_t,\density), (\h_t,\density)\rangle = -dE_t(\density)$, which as before implies that $(\h_T + \widetilde h \circ f_T, \density)$ and $(\h_0 + \widetilde h, \density)$ agree in law for each $\density$, just as in the $\kappa < 8$ case, which yields a $\kappa \geq 8$ analog of Theorem \ref{forwardcoupling}.  (Figure \ref{brownianenergyfig} still makes sense then $\kappa \geq 8$.) \end{remark}

It will be useful for later purposes to note that (at least in the reverse SLE case) the graph in Figure \ref{brownianenergyfig} actually uniquely determines (and is uniquely determined by) the process $W_t = \sqrt \kappa B_t$ almost surely, see Figure \ref{brownianenergy2fig}.  This is a special case of a much more general theorem about stochastic processes (see Chapter IX, Theorem 2.1 of \cite{MR2000h:60050} --- it suffices that $(\mathfrak h_t, \density)$ satisfies an SDE in $W_t$ with a diffusive coefficient that remains strictly bounded away from zero and infinity, at least as long as we stop at any time strictly before $t=\infty$).  This means that the evolution of $\eta$ can be described by the Brownian motion in Figure \ref{brownianenergyfig}, as well as by the Brownian motion $B_t$.  Context will determine which description is more convenient to work with.

\begin {figure}[htbp]
\begin {center}
\includegraphics [width=5.4in]{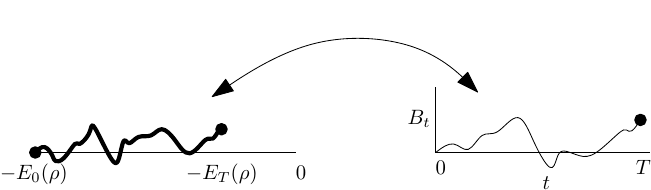}
\caption {\label{brownianenergy2fig} The graph traced by $\bigl(-E_t(\density), (\mathfrak h_t, \density) \bigr)$  as $t$ ranges from $0$ to $T$ (left) and the graph traced by $(t,B_t)$ (right), where $W_t = \sqrt \kappa B_t$.  The left graph uniquely determines the right graph, and vice versa, almost surely.  Each has the law of a standard Brownian motion (up to a stopping time).}
\end {center}
\end {figure}

\subsection{Alternative underlying geometries and SLE$_{\kappa, \rho}$} \label{ss.slekr}

Both Theorems \ref{forwardcoupling} and \ref{reversecoupling} can be generalized to other values of $\h_0$ using the so-called SLE$_{\kappa, \rho}$ processes.  (As discussed at the end of Section \ref{ss.randomgeometryGFF}, changing $\h_0$ can be interpreted as changing the underlying geometry on which Liouville quantum gravity is defined.)
We
generalize the latter here (see \cite{MR2525778} for the former).  We take $G = G^{\H_F}$ in the following.  Slightly abusing notation, we will consider situations where $\density(y)dy$ represents a general signed measure (instead of requiring that $\density$ be a smooth test function).  In this case $(F,\density) = \int F(y) \density(y)dy$ represents integration of $F$ w.r.t.\ this measure.

\begin{theorem} \label{reversecouplingkapparho}
Fix $\kappa > 0$ and a signed measure $\density(y)dy$ on $\overline \H$ with finite positive and finite negative mass supported on some closed $\mathcal C \subset \mathcal \H$.  Write
\begin{equation} \label{e.hatht} \hat \h_t(z) =  \h_t(z) + \frac{1}{2\sqrt \kappa}\int G_t(y,z) \density(y)dy,\end{equation}
where $\h_t(z)$ is as in Theorem \ref{reversecoupling}, and let $\widetilde h$ be an instance of the
free boundary GFF on $\H$, independent of $B_t$.   Let $\eta_T$ be the segment generated by the reverse Loewner flow
\begin{equation} \label{e.dWkapparho1} d f_t(z)  =  \frac{-2}{f_t(z)}dt - dW_t,\end{equation}
where \begin{equation} \label{e.dWkapparho} dW_t = \Bigl(\int \Re \frac{-1}{f_t(y)} \density(y)dy \Bigr) dt + \sqrt{\kappa}dB_t = \bigl( -\Re (f_t)^{-1}, \density \bigr)dt + \sqrt{\kappa}dB_t,\end{equation}
up to any stopping time $T \geq 0$ at or before the smallest $t$ for which $0 \in f_t(\mathcal C)$ (here $f_t$ is extended continuously from $\H$ to $\overline \H$).  Then the following two random distributions (modulo additive constants) on $\H$ agree in law: $h = \hat \h_0 + \widetilde h$ and $\hat \h_T + \widetilde h \circ f_T$.
\end{theorem}

We will make several observations before we prove Theorem \ref{reversecouplingkapparho}.  First, if $\density$ is supported on a set of $n$ points $y_1, \ldots, y_n$ in $\H$, with masses given by real numbers $\rho_1, \rho_2, \ldots, \rho_n$, then the process defined by \eqref{e.dWkapparho1} and \eqref{e.dWkapparho} is (the reverse form of) what is commonly called an SLE$_{\kappa, \rho}$ process in the literature: in this case, \eqref{e.dWkapparho} takes the form
\begin{equation} \label{e.finitesupportslekr} dW_t = \sum_{i=1}^n \Re \frac{-\density_i}{f_t(y_i)} dt + \sqrt{\kappa}dB_t,\end{equation}
which is the same expression one finds in the usual definition of the forward-flow SLE$_{\kappa, \rho}$ process, as in \cite{MR2188260}.  (Note that the notation here differs from  \cite{MR2188260}, since here we use $\density$ to denote the measure, not the vector of mass values $\density_i$.)  In the special case that $\density$ is supported at a single point $x \in \R$, with mass $\rho_1$ we find that
\begin{equation} \label{e.singleforce} d f_t(x) = \frac{-2}{f_t(x)}dt - dW_t = \frac{-2 + \rho_1}{f_t(x)}dt  - \sqrt{\kappa}dB_t,\end{equation}
so that $f_t(x)/\sqrt{\kappa}$ is a Bessel process of dimension $\delta$ satisfying $(\delta-1)/2 = (\rho_1-2)/\kappa$, i.e., \begin{equation} \label{e.besseldim} \delta = 1 + \frac{2(\rho_1-2)}{\kappa}. \end{equation}  For this and future discussion it will be useful to recall a few standard facts:
 \begin{enumerate} \item The Bessel process $X_t$ of dimension $\delta$ by definition satisfies $dX_t = dB_t + \frac{\delta-1}{2} X_t^{-1}dt$.  Hence $d \log X_t = \frac{1}{X_t}dB_t + \frac{\delta-1}{2X_t^2}dt - \frac{1}{2X_t^2} dt$.  The process $\log X_t$, when parameterized by its quadratic variation, is a Brownian motion with a constant drift of magnitude $\frac{\delta-2}{2}$.
 \item If $X_t$ is a Bessel process of dimension $\delta$ started at $X_0 = x$ and run until the first time $T$ that it reaches zero, then the time reversal $X_{T-t}$ has the law of a Bessel process of dimension $\delta'$, started at zero and run until the last time that it hits $x$, where $\delta'$ is the dimension one gets by changing the sign of the drift in $\log X_t$.  That is, $\frac{\delta-2}{2} = -\frac{\delta'-2}{2}$, so that \begin{equation} \label{e.delta'} \delta = 4-\delta'.
     \end{equation}
 \item In the usual forward flow definition of SLE$_{\kappa, \rho_1'}$ the function $f_t(x)$ is a Bessel process of dimension \begin{equation} \label{e.forwardbesseldimension} \delta' = 1 + \frac{2 (\rho_1'+2)}{\kappa}. \end{equation} The reason for the difference from \eqref{e.besseldim} can be seen by considering the case $\rho_1 = \rho_1' = 0$.  In the reverse process, the Loewner drift is pulling $f_t(x)$ toward the origin, while in the forward process the Loewner drift is pushing $f_t(x)$ away from the origin.  In both cases $\rho_1$ (or $\rho_1'$) indicates a quantity of additional force pushing $f_t(x)$ away from the origin.
 \item Combining \eqref{e.besseldim},  \eqref{e.delta'}, and \eqref{e.forwardbesseldimension} gives a relationship between $\rho_1$ and $\rho_1'$.  Namely, $1 + \frac{2 (\rho_1'+2)}{\kappa} = 4 - ( 1 + \frac{2(\rho_1-2)}{\kappa}),$ so that \begin{equation} \label{e.rhorho'} \rho_1' = \kappa - \rho_1.\end{equation}  This means that if we run a reverse SLE$_{\kappa, \rho_1'}$ until the time $T$ at which $f_t(x)$ hits zero, then $f_T$ maps $\H$ to $\H \setminus \eta_T$ where $\eta_T$ has the law of an initial segment of a forward SLE$_{\kappa, \rho_1}$.  In particular, if $\rho_1' = \kappa$, then $\eta_T$ has the law of an ordinary SLE stopped at a time $T$ (which corresponds to the last time that a Bessel process hits a certain value).  This will be important later.
 \end{enumerate}

Recall from Section \ref{ss.randomgeometryGFF} that changing $\h_0$ to $\hat \h_0$ can be interpreted as changing the underlying geometry on which Liouville quantum gravity is defined.  Moreover, $\density$ is proportional to $-\Delta (\hat \h_0 - \h_0)$, and $-\Delta \hat \h_0$ is proportional to the overall Gaussian curvature density (see the appendix).

We will give a formal proof of Theorem \ref{reversecouplingkapparho} below using \Ito/ calculus, but first let us offer an informal explanation of why the result is true.  The idea behind Theorem \ref{reversecouplingkapparho} is to interpret \eqref{e.hatht} as the expectation of $h$ in a certain weighted measure and \eqref{e.dWkapparho} as the description of the law of $W_t$ in that measure. This is easiest to understand when we first switch coordinates using the correspondence shown in Figure \ref{brownianenergy2fig}.  Suppose first that $\density$ is such that \eqref{e.Hfreecovariance} is finite with $\density_1 = \density_2 = \density$ and that the total integral of $\density$ is zero.  If $dh$ is the law of a (centered or not centered) GFF then the standard Gaussian complete-the-square argument shows that $e^{(h, \density)}dh = e^{(h, -2\pi\Delta^{-1} \density)_\nabla}dh$ (normalized to be
a probability measure) is the law of the standard GFF plus $-2\pi\Delta^{-1} \density$.  When we weight the law of the Brownian motion in Figure \ref{brownianenergyfig} by $e^{\alpha(h, \density)}$ for some constant $\alpha$ (note that $(h,\density)$ is the terminal value that the Brownian motion in that figure reaches at time zero) this is equivalent to adding a constant drift term to the Brownian motion (parameterized by $-E_t(\density)$) in Figure \ref{brownianenergyfig}.

We take $\alpha = \frac{1}{2 \sqrt \kappa}$ and weight by \begin{equation} \label{e.densityweight} e^{(h, \frac{1}{2 \sqrt \kappa}\density)}, \end{equation} which, as explained above, modifies the law in a way that amounts to adding the drift term of $\frac{1}{2 \sqrt \kappa} \bigl(E_0(\density) - E_t(\density)\bigr)$ to the Brownian motion in Figure \ref{brownianenergyfig}.  Recalling the correspondence shown in Figure \ref{brownianenergy2fig}, the fact that the left figure is a Brownian motion with this constant drift (up to a stopping time) completely determines the law of $W_t$ up to that stopping time.  Indeed, recalling \eqref{e.htdensity}, we find that the law of $W_t$ is necessarily the one described by \eqref{e.dWkapparho}.

We have now related the weighted measure to \eqref{e.dWkapparho}, but what does this have to do with \eqref{e.hatht}?  Observe that $$(\hat \h_t, \density) = (\h_t, \density) + \frac{1}{2 \sqrt \kappa} E_t(\density)$$ represents the conditional expectation (in the weighted measure) of $(h, \density)$ given $B_\cdot$ up to time $t$.  In fact, by the standard complete-the-square argument, the function $\hat \h_t$ in Theorem \ref{reversecouplingkapparho} represents the conditional expectation (in the weighted measure) of $h$, given $B_\cdot$ up to time $t$, and is thus a martingale in $t$.  The above construction (and a bit of thought) actually constitutes a proof of Theorem \ref{reversecouplingkapparho} when \eqref{e.Hfreecovariance} is finite and the total integral of $\density$ is zero.

The argument above can be adapted to more general $\density$.  If the total integral of $\density$ is not zero, we may modify $\density$ by adding some mass very far from the origin, so that the total integral becomes zero but the drift in \eqref{e.dWkapparho} does not change very much.  If \eqref{e.Hfreecovariance} is infinite, we may be able to modify it to make it finite: for example, if $\density$ is a point mass, then we may replace it with a uniform measure on a tiny ball centered at that point mass, and the harmonicity of $G_t(\cdot, z)$ in \eqref{e.hatht} and of $\Re(f_t)^{-1}$ in \eqref{e.dWkapparho} show that (outside of this small ball) neither \eqref{e.hatht} nor \eqref{e.dWkapparho} is affected by this replacement.
We will present a more direct \Ito/ calculation below, which also applies when \eqref{e.Hfreecovariance} is infinite.

\proofof{Theorem \ref{reversecouplingkapparho}}
We will follow the calculations of Theorem \ref{reversecoupling} and check where differences appear.
First, we find the following:
\vspace{.2 in}
\begin{center}
\begin{tabular} {|l|}
\hline
{\bf REVERSE FLOW SLE} \\
\hline
$df_t(z) = \frac{-2}{f_t(z)}dt +  \Bigl( \int \Re \frac{1}{f_t(y)} \density(y)dy \Bigr) dt - \sqrt{\kappa} dB_t$  \\
\hline
$d \log f_t(z) =\frac{-(4+\kappa)}{2f_t(z)^2}dt + f_t(z)^{-1} \Bigl(\int \Re \frac{1}{f_t(y)} \density(y)dy \Bigr) dt - \frac{\sqrt{\kappa}}{f_t(z)}dB_t$ \\
\hline
$d f_t'(z) = \frac{2f_t'(z)}{f_t(z)^2}dt$ \\
\hline
$d \log f_t'(z) = \frac{2}{f_t(z)^2}dt$ \\
\hline
\end{tabular}
\end{center}
\vspace{.2 in}
Also, as before, we compute
$$dG_t(y,z) =-\Re \frac{2}{f_t(y)} \Re\frac{2}{f_t(z)} dt ,$$
and recall that
$$\hat \h_t(z) := \frac{2}{\sqrt{\kappa}}  \log|f_t(z)| + Q \log |f_t'(z)| + \frac{1}{2\sqrt \kappa}\int G_t(y,z) \density(y)dy.$$
We then find that when computing $d\hat \h_t(z)$ the extra term in $d \frac{2}{\sqrt{\kappa}} \Re \log f_t(z)$ cancels the term $d\frac{1}{2\sqrt \kappa}\int G_t(y,z) \density(y)dy$
so that $d \hat \h_t(z) =  \Re \frac{-2}{f_t(z)} dB_t $, just as in the proof of Theorem \ref{reversecoupling}.  The remaining calculations are the same as in the proof of Theorem \ref{reversecoupling}.
\qed
\vspace{.1in}

We next remark, in the context of Theorem \ref{reversecouplingkapparho}, that if $(\H, \hat \h_T + \widetilde h \circ f_T)$ is a quantum surface, then
\begin{equation} \label{e.hatht2} f_T(\H, \hat \h_T + \widetilde h \circ f_T) = f_T(\H, \h_T + \frac{1}{2\sqrt \kappa}\int G_T(y,\cdot) \density(y)dy + \widetilde h \circ f_T),\end{equation}
and this can be written
\begin{equation} \label{e.ftkapparhosurface}  (\H \setminus K_T, \h_0 + \frac{1}{2\sqrt \kappa}\int G(f_T(y),\cdot) \density(y)dy + \widetilde h). \end{equation}
When $\kappa < 4$, this suggests the following interpretation of Theorem \ref{reversecouplingkapparho}.  We start with $\hat \h_0 + \widetilde h$, which is actually only defined up to additive constant, so it determines a quantum surface up to a multiplicative constant.  We zip up this (modulo multiplicative constant) quantum surface until a stopping time $T$, and condition on the zipper map $f_T$.  Then the conditional law of the new zipped-up quantum surface (which is also defined only up to a multiplicative constant) is the same as the original law except that $\density$ is replaced by the $f_T$ image of $\density$.  We have not fully established that this interpretation is correct, because we have not yet shown that $\hat \h_0 + \widetilde h$ uniquely determines $f_T$.  As in Theorem \ref{reversecoupling}, we have only shown that sampling $h$ from $\hat \h_0 + \widetilde h$ is equivalent to first zipping up according to a given law, then sampling the field from a putative conditional law in the zipped up picture, and then unzipping.

\section{Quantum zippers and conformal welding} \label{s.zipper}

\subsection{Overview of zipper proofs}
The goal of this section is to prove the statements in Section \ref{s.intro} that we have not yet proved: namely,  Theorem \ref{conformalwelding} (which, by the a.s.\ removability of the SLE paths, implies Theorem \ref{conformalweldingunique} and Corollary \ref{stationaryzipping}, as explained in Section \ref{s.intro}) and Theorem \ref{t.lengthstationary}.  Both proofs will be completed in Section \ref{ss.zipperconlusion}.

Given what we know now, Theorem \ref{conformalwelding} may not seem surprising.  In light of Theorem \ref{reversecoupling}, we know that, given the independent pair $(h,\eta)$ described there, the quantum boundary lengths along the left and right sides of $\eta$ are both a.s.\ well defined.  Indeed, recall that to measure the $\nu_h$ length along the left side of $\eta([0,t])$, we may ``unzip'' via $f^\eta_t$ --- and the transformation rule \eqref{Qmap} --- so that $\eta([0,t])$ maps to an interval of $\R$, and then measure the quantum length of that interval.  We only need to show that the quantum boundary measure of $\eta$ measured from the left agrees with the quantum boundary measure of $\eta$ measured from the right.

We already have some information about how these measures depend on $h$.  For example, it is immediate from the definition of $\nu_h$ that if we change $h$ --- by adding a smooth function to $h$ that is equal to a constant $C$ in a region $A \subset \H$ --- then this has the effect of multiplying the length of $\eta([0,t]) \cap A$ (as measured from both left and right sides) by $e^{C\gamma/2}$.  More generally, if $\nu_h$ is the measure from one of the two sides (viewed as a measure on $\overline \H$, which happens to be supported along $\eta \cup \R$) and $\phi: \overline \H \to \R$ is smooth, then \begin{equation} \label{e.nuhplusphi} \nu_{h + \phi} = e^{\frac{\gamma}{2} \phi} \nu_h,\end{equation} (i.e., the measure whose Radon-Nikodym derivative w.r.t.\ $\nu_h$ is $e^{\frac{\gamma}{2} \phi}$).  This implies that the left and right measures depend on $h$ in a similar way, but it does not show that they agree.

One possible way to prove agreement might be to suppose otherwise for contradiction.  Then if we cover $\eta$ with a lot of small balls of comparable quantum size, we will find that in some of the balls the quantum length of $\eta$ is greater measured from the left side and in some greater from the right side.    Still, we would expect some long range near-independence and a law of large numbers to show that all of these differences average out when the balls are small enough; taking limits as the balls get smaller should show that the lengths on the two sides are in fact equal.
(A related argument based was used in \cite{MR2486487} to prove the so-called height gap lemma for discrete Gaussian free fields.)

An alternative approach to proving Theorem \ref{conformalwelding} would be to try to show that both the left and right boundary measures give ``quantum length'' measures of the curve that correspond to the quantum analog (as in \cite{DS3}) of the natural time parameterization constructed via the Doob-Meyer decomposition in \cite{2009arXiv0906.3804L, 2010arXiv1006.4936L}.  The uniqueness arguments introduced in \cite{2009arXiv0906.3804L} could then be used to show that these two measures must agree.

While both of the above approaches appear viable, we will actually prove Theorem \ref{conformalwelding} with a third approach, which we feel is instructive.  Namely, we will first construct the invariant measure described by Theorem \ref{t.lengthstationary}, in the manner outlined in Figure \ref{triplezoom}, and then use symmetries and the ergodic theorem to deduce that the quantum measures on the two sides of the path are almost surely equal.  Before doing this, we establish some results of independent interest. Section \ref{ss.capstataddconst} describes a $T=\infty$ analog of Theorem \ref{reversecoupling} and Section \ref{ss.reparameterization} describes an interesting space-time symmetry: by changing the underlying geometry of the space used to define Liouville quantum gravity, one may construct a quantum zipper that is invariant with respect to a modified notion of capacity time.

\subsection{Zipping up ``all the way'' and capacity stationarity} \label{ss.capstataddconst}

Let $\Gamma^0$ be the law of the pair $(h, \eta)$ in which $h$ is $\frac{2}{\sqrt{\kappa}}  \log|z| + \widetilde h$ (with $\widetilde h$ a free boundary GFF on $\H$, defined modulo additive constant) and $\eta$ is an independent SLE$_\kappa$.  This is the measure that appears in Theorem \ref{reversecoupling} and Corollary \ref{stationaryzipping}.  Note that the space of $\bigl( (D_1,h^{D_1}), (D_2, h^{D_2}) \bigr)$ configurations is in one to one correspondence with the space of pairs $(h, \eta)$ on a full $\Gamma^0$ measure set, so we can view the $\zcap_t$ of Corollary \ref{stationaryzipping} as acting on the pair $(h, \eta)$ and we will denote the thus transformed pair by $\zcap_t(h, \eta)$.  By Theorem \ref{reversecoupling}, the $\zcap_t$ of Corollary \ref{stationaryzipping} with $t < 0$ (the ``unzipping'' direction) are measure preserving transformations of $\Gamma^0$.

Using Theorem \ref{reversecoupling} alone, we do not yet have a definition of $\zcap_t$ for $t>0$ (the ``zipping up'' direction).  However, we can construct a stationary process $(h^t, \eta^t)$, defined for all $t$, such that whenever $t<0$ and $s \in \R$ we have \begin{equation} \label{hst} (h^{s+t}, \eta^{s+t}) = \zcap_t(h^s,\eta^s).\end{equation}  To construct this process, note that once we choose $(h^T,\eta^T)$ for some large fixed constant $T$ from $\Gamma^0$, the evolution rule \eqref{hst} determines $(h^t,\eta^t)$ for all $t<T$.  (The fact that the unzipping maps preserve $\Gamma^0$ implies that for each fixed $t < T$, the pair $(h^t, \eta^t)$ also has the law of $\Gamma^0$.) Taking a limit as $T \to +\infty$ (using Kolmogorov's consistency theorem), we obtain a process defined for all $t \in \R$.  Denote by $\Gamma$ the law of this process.  We frequently write $h = h^0$ and $\eta = \eta^0$.

For the remainder of Section \ref{s.zipper}, we will use $f_t$ to denote ``zipping down'' maps when $t < 0$ and ``zipping up'' maps when $t>0$.  Once Corollary \ref{stationaryzipping} is established, this will amount to defining $f_t$ for all $t \in \R$ (using the notation of Corollary \ref{stationaryzipping}) as
\begin{equation} \label{e.ftalltimedef} f_t = \begin{cases} f_t^h & t \geq 0 \\  f_{-t}^\eta & t \leq 0 \end{cases}.\end{equation}
(We stress that this contrasts with earlier notation, where we defined $f_t$ only for $t > 0$ and --- to highlight similarities --- used $f_t$ to describe both forward Loewner evolution in the AC-geometry context and reverse Loewner evolution in the Liouville quantum gravity context.)  However, using Theorem \ref{reversecoupling} alone, it is not a priori clear that $h$ determines the map $f^h_t$ in \eqref{e.ftalltimedef}.  Even so, it {\em is} clear that these maps are determined by the process $(h^t, \eta^t)$ with the law $\Gamma$ constructed above, so for now we will define $f_t$ to be the maps determined by process $(h^t, \eta^t)$.

Now compare \eqref{e.floewner} and \eqref{e.rloewner} and note that the Brownian motion corresponding to the forward Loewner evolution of an SLE segment $\eta_T$ (the restriction of a Brownian motion to an interval of time) is (up to additive constant) the time reversal of the Brownian motion corresponding to the reverse Loewner evolution $f_t$ for the same curve.  Indeed, if we write $\widetilde B_t = B_{T-t}$ and $\widetilde f_t = f_{T-t}$, then \eqref{e.rloewner} holds for a general pair $B_t$ and $f_t$ (indexed by $t \in [0,T]$) precisely when \eqref{e.floewner} holds for $\widetilde B_t$ and $\widetilde f_t$.  The $f_t$ map is driven by a Brownian motion in the $t<0$ ``unzipping'' direction (since the curve that it unzips is an SLE curve) and from the definition of $f_t$ for general times (which involved starting at stationarity for some large $T$ and unzipping from there to get $f_t$ for $t<T$, taking the limit as $T \to \infty$) we see that we may define a standard Brownian motion $B_t$ for all times $t$ so that $B_0 = 0$ and $f_t$ satisfies \eqref{e.rloewner} for all time.  (This is also to be expected from  \eqref{e.ftalltimedef}, since once we show that $f^h_t$ is determined by $h$, we will expect that the forward flow $f^\eta_t$ and reverse flow $f^h_t$ are driven by Brownian motions, for $t \geq 0$, that are independent of one another.)

We have shown above how to construct the $f_t$ corresponding to $\Gamma$ from a single Brownian motion (indexed by all of $\R$).  We will now give an alternate description of the law $\Gamma$ by coupling this $f_t$ with $h$ using Theorem \ref{reversecoupling} applied with $T = +\infty$.  Informally, this amounts to ``zipping up all the way'' and then sampling a GFF on the fully zipped up surface.  (This description comprises the remainder of Section \ref{ss.capstataddconst}. It will not be used in subsequent sections and may be skipped on a first read.  See \cite{wedgespaper} for an analysis of what happens when two sides of a general quantum wedge are zipped together.)  To explain what this means, recall that we showed in the proof of Theorem \ref{reversecoupling} that for each fixed $\density$ with total integral zero, the process $(\h_t, \density)$ is a Brownian motion parameterized by $-E_t(\density)$; in particular, this implies that $\lim_{t \to +\infty} (\h_t, \density)$ exists almost surely and has variance at most $E_0(\density)$.

We will be interested in $\h_t$ modulo additive constant.  Let us focus on the function $\h_t(z) - \h_t(y)$ for some fixed point $y \in \H$.  By harmonicity, $\h_t(z)$ is equivalent to its mean value on a ball centered at $z$, and hence $\h_t(z) - \h_t(y)$ can be written as $(\h_t, \density)$ for a smooth, compactly supported $\density$ with mean zero.  Thus $\h_t(z) - \h_t(y)$ is a Brownian motion when parameterized by $-E_t(\density)$, and the total amount of quadratic variation time elapsed as $t \to +\infty$ is at most $E_0(\density)$ almost surely.  When $z$ is restricted to a bounded domain $D$ with closure in $\H$, we can define such a $\density$ for each $z$ so that the quantity $E_0(\density)$ is bounded uniformly in $z$.   By Doob's inequality, this implies that $\sup_{t} |\h_t(z) - \h_t(y)|$ has a finite second moment, uniformly bounded for $z$ in $D$, and hence $$\mathbb E \int_D \sup_t|\h_t(z) - \h_t(y)|^2 dz < \infty.$$  In fact, since $\h_t(z) - \h_t(y)$ converges almost surely for each $z$, the above shows that $\h_t(\cdot) - \h_t(y)$ restricted to $D$ is almost surely Cauchy in the space $L^2(D)$; since the $\h_t$ are harmonic on $D$, this implies that $\h_t$ almost surely converges uniformly (on each compact subset of $\H$, modulo additive constant) to a harmonic limit $\h_\infty$ as $t \to +\infty$.

On the other hand, the maps $f_t$ clearly do not converge to a finite limit; the Loewner evolution shows that when $z \in \H$ is fixed, $f_t(z)$ cannot converge to any value except for $\infty$ as $t \to +\infty$.  However, let us define $\hat f_t$ to be $a_t f_t + b_t$ where complex constants $a_t$ and $b_t$ are chosen for each $t$ so that $\hat f_t(i-1) = (i-1)$ and $\hat f_t(i+1) = i+1$.  If $f_T$ is as in Figure \ref{reverseftfig}, then $\hat f_T$ would be as in Figure \ref{reverseftnormalizedfig}.
\begin{proposition}
The limit \begin{equation}
\label{e.hatftlimit} \hat f_\infty(z) := \lim_{t \to \infty} \hat f_t(z)\end{equation} exists almost surely as an analytic map from $\H$ to $\C$. It conformally maps $\H$ to the complement, in $\C$, of a certain random path from some finite starting point in $\C$ to $\infty$. (One may then translate the surface so that this starting point is the origin.)
\end{proposition}

\begin {figure}[h]
\begin {center}
\includegraphics [width=3in]{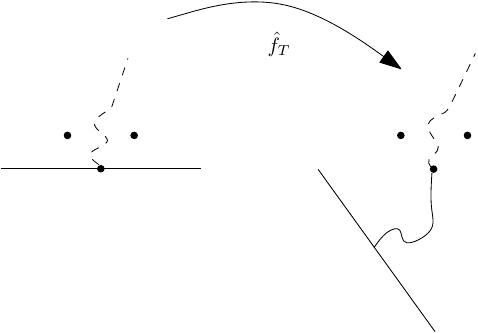}
\caption {\label{reverseftnormalizedfig} The map $\hat f_T$ is first normalized so that $\hat f_T(i-1) = (i-1)$ and $\hat f_T(i+1) = i+1$.  (Here $i \pm 1$ are shown as dots; a curve $\eta$ and its image under $\hat f_T$ is also shown.)  The limit $\hat f_\infty(z) := \lim_{t \to \infty} \hat f_t(z)$ a.s.\ exists as a conformal map from $\H$ to the complement in $\C$ of a semi-infinite path from some starting point to $\infty$.}
\end {center}
\end {figure}

\begin{proof}
To establish the existence of this limit, we argue that each of the two terms on the RHS of \begin{equation} \label{e.anotherhtdef} \h_t(z) = \h_0(f_t(z)) + Q \log |f_t'(z)|\end{equation} converges almost surely (modulo additive constant) to a limit.  By \eqref{e.hatftlimit} the LHS has a limit, so it will suffice to show that $\h_0(f_t(z)) - \h_0(f_t(y))$ a.s.\ converges uniformly to zero (for $y,z \in D$).  For contradiction, suppose otherwise.  Then, for some fixed $C$, let $T_k$ be the first time after time $2^k$ for which there exist $y$ and $z$ in $D$ such that \begin{equation} \label{e.loggap} \Bigl| \log |f_T(z)| - \log |f_T(y)| \Bigr| \geq C.\end{equation}  By assumption, we may choose $C$ so that with positive probability each of the $T_k$ is finite.  In each case, we may assume $|f_T(z)| \geq |f_T(y)|$ so that \eqref{e.loggap} implies $|f_T(z)|/|f_T(y)| > 1+C$.

Now, conditioned on $f_t$ up to such a time $T=T_k$,  there is at least a constant $C_1$ probability that $\h_t(z) - \h_t(y)$ will change by at least some constant $C_2$ during the next $\min \{|f_T(y)|^2, |f_T(z)|^2\}$ units of capacity time after time $T$.  By scaling, it is enough to observe that this is true when $|f_T(y)| < |f_T(z)|$ and $|f_T(y)| = 1$, which follows from the claim in the caption to Figure \ref{ftyftz}.

\begin {figure}[h]
\begin {center}
\includegraphics [width=1.5in]{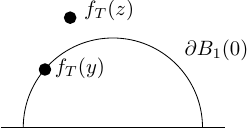}
\caption {\label{ftyftz} Assume that $|f_T(y)| = 1$ and $|f_T(z)|>1+C$ and consider the reverse Loewner flow driven by $W_t = \sqrt \kappa B_t$).
Claim: for some $C_1, C_2>0$ (depending only on $C$) there is at least a $C_1$ probability that $\h_t(z) - \h_t(y)$ changes by at least $C_2$ after time $T$.}
\end {center}
\end {figure}

To prove this claim, recall \eqref{e.anotherhtdef} (and the definition of $\h_0$ in Theorem \ref{reversecoupling}):
\begin{equation} \label{e.htdiff} \h_t(z) - \h_t(y) =   \frac{2}{\sqrt{\kappa}}\Bigl(  \log|f_t(z)| - \log|f_t(y)| \Bigr) + Q \Bigl( \log |f_t'(z)| - \log |f_t'(y)|\Bigr).\end{equation}
To deal with the first difference on the RHS of \eqref{e.htdiff}, note that $$\log|f_T(z)-a| - \log|f_T(y)-a|$$ cannot, as a function of $a$, be constant on the interval $[-1/2,0]$ or constant on the interval $[0,1/2]$.  Indeed (if $|f_T(z)|>1+C$) the amount by which this function varies over each interval is at least some $c>0$.  Thus if $W_t$ goes up or down by $1/2$ during $t_\Delta$ units of time after time $T$ (and $t_\Delta$ is small enough), then the Loewner flow \eqref{e.rloewner} shows that $f_t(z)-f_T(z)$ is approximately $W_T-W_t$ during this time, so that the first difference on the LHS of \eqref{e.htdiff} changes by at least $\frac{c}{\sqrt{\kappa}}$.  If we choose $t_\Delta$ small enough, we can also ensure that the second difference on the RHS of \eqref{e.htdiff} changes by less than $\frac{c}{2\sqrt{\kappa}}$ (recall $d \log f'_t = (2/f_t^2)dt$).  Thus, \eqref{e.htdiff} changes by at least $C_2 := \frac{c}{2\sqrt{\kappa}}$.  The probability of this depends only on $t_\Delta$, which depends only on $C$.

It follows from the above that on the event that the $T_k$ are all finite, the quantity \eqref{e.htdiff} a.s.\ changes by $C_2$ after $T_k$ for infinitely many choices of $k$, contradicting the fact (which we have already shown) that $\h_t(z) - \h_t(y)$ a.s.\ converges uniformly to a limit for $y,z \in D$.  We conclude that there is almost surely a last time for which $\log |f_T(z)| - \log |f_T(y)| \geq C$ for some $y$ and $z$ in $D$, and since this holds for any $C$, the RHS of \eqref{e.anotherhtdef} a.s.\ converges uniformly to zero for $y,z \in D$.  Since $\log |f_t'(z)|$ (the second term on the RHS of \eqref{e.anotherhtdef}) then converges almost surely uniformly (modulo additive constant) to a limit, it follows (adding $i$ times the harmonic conjugate) that $\log f_t'(z)$ converges almost surely uniformly on $D$ (modulo additive constant), which implies the uniform convergence of $f_t'$ (modulo multiplicative constant) and hence the the convergence of $\hat f_t$ to a limit.  Since the limiting map is analytic, it follows that the path (as in Figure \ref{reverseftnormalizedfig}) converges to a limiting continuous path as well.
\end{proof}

Theorem \ref{reversecoupling} can now be stated for $T = +\infty$.  The statement is the same except that we replace $$h \circ f_T + Q \log |f'_T| = \h_0 \circ f_T + \widetilde h \circ f_T + Q \log |f'_T|$$ with
$$\widetilde h \circ \hat f_\infty +  Q \log |\hat f'_\infty|,$$
considered modulo additive constant.  (Recall that the $\h_0(f_t(z))$ term on the RHS of \eqref{e.anotherhtdef} a.s.\ tends to zero modulo additive constant as $t \to +\infty$.)

Note that in this case, after ``zipping up all the way to $+\infty$'' the image $\hat f_\infty(\H)$ is a (dense) subset of $\C$, rather than $\H$, and thus $\widetilde h$ is defined as a free boundary GFF on $\C$ rather than just on $\H$.  Given this change, there is nothing in the proof of Theorem \ref{reversecoupling} in Section \ref{couplingsection}, where the law of $(h, \density)$ was checked one test function $\density$ at a time, that fails to hold if we take $T = +\infty$.

To summarize:
\begin{proposition} We may produce a sample from $\Gamma$ explicitly as follows: first sample a standard Brownian motion $B_t$ for $t \in \R$, which determines $f_t$ (and hence $\eta^t$) for all time (positive and negative) by solving \eqref{e.rloewner}.  Then define (modulo additive constant) $\h_\infty := \lim_{t \to +\infty} \h_t$ and $\hat f_\infty := \lim_{t \to +\infty} \hat f_t$ and write $$h = h^0 = \widetilde h \circ \hat f_{\infty} + \h_\infty$$ where $\widetilde h$ is a free boundary GFF on $\C$.  Define $h^t$ for all other finite $t$ as the quantum surface transformation of $h^0$ under $f_t$.  That is, $h^t = h^0 \circ f_t^{-1} + Q \log |(f_t^{-1})'|$.
\end{proposition}

The entire process $(h^t, \eta^t)$ is thus determined by $\widetilde h$ and the process $B_t$.  Note that for each $t$, the curve $\eta^t$ is a random path.  The forward Loewner evolution of this path --- parameterized by its capacity time $s$ --- is given by $f_{t-s}$, with $s$ ranging from $0$ to $+\infty$.  Its Loewner driving function is (as a function of $s$) $\sqrt \kappa B_{t-s}$.  Note that $\zcap_t$ has the effect of translating the process $B_{\cdot}$ by $t$ units to the left (and adding a constant to maintain $B_0=0$).

\subsection{Re-parameterizing time: general observations} \label{ss.reparameterization}

Now we make a general observation about the GFF.  Suppose we write $dh$ for the law of a (not necessarily centered) free boundary GFF $h$, and that $f$ belongs to the Hilbert space on which $h$ is defined, so that $\Var (h,f)_\nabla = (f,f)_\nabla$.  As mentioned in Section \ref{ss.slekr}, the standard Gaussian complete-the-square calculation shows that $e^{(h,f)_\nabla} dh$ (multiplied by a constant to make it a probability measure) is the law of the original GFF plus $f$.  In other words, weighting the law of $h$ by $e^{(h,f)_\nabla}$ is equivalent to deterministically adding $f$ to $h$.  Thus weighting the law of $h$ by $e^{(h, \density)}$ (we assume that $\density$ has total integral zero) is equivalent to adding $-2\pi \Delta^{-1} \density$ to $h$.  Here $\Delta^{-1}$ is defined in the Neumann sense, i.e., by taking by \eqref{e.deltainverse} with $G(x,y) = G^{\H_F}(x,y)$.  Note also the following:

\begin{proposition} \label{p.densitystationary} Let $(h^t,\eta^t)$ be the process with law $\Gamma$ (as in Section \ref{ss.capstataddconst}), let $\density$ be a signed measure with total integral zero (for which $(h^0,\density)$ is a.s.\ finite), and write $$s_\density (t) := \int_0^t e^{(h^s, \density)}ds.$$  Then the weighted measure $e^{(h^0,\density)} \Gamma$ (normalized to be a probability measure) is stationary w.r.t.\ $s_\density$ time.
\end{proposition}
\begin{proof}
First observe that the constant $ \mathbb E_\Gamma e^{(h^0,\density)}$ is finite, and by Fubini's theorem $\mathbb E_\Gamma \int_0^T  e^{(h^t, \density)}dt $ is $T$ times that constant.  The ergodic theorem implies that as $T \to \infty$ along integers the random variables $s_\density(T)/T$ converge to a (possibly random) value $c$ almost surely and in $L^1$. To see that $c$ is a.s.\ constant, we would need to check that the pair $(h^\cdot, \eta^\cdot)$ is ergodic. We do not really need ergodicity for the proof of this proposition but we mention as an aside that it is not hard to prove (the Brownian motion generating $\eta_t$ is ergodic; and if $t_1$ and $t_2$ are distinct times, one can use to GFF properties to show that the restrictions of $h$ to tiny neighborhoods of $\eta(t_1)$ and $\eta(t_2)$ are nearly independent; letting the neighborhoods get small, and using scale invariance, one can establish a long-range mixing property that implies ergodicity). The reason that ergodicity is not necessary for the proof of the proposition is that if $(h^\cdot, \eta^\cdot)$ had multiple ergodic components but each component were invariant under the operation of translation by a fixed amount of $s_\density$ time, then the entire process would be invariant under this operation as well; so it is enough to focus on a single ergodic component.

Now, suppose that $\pi_T$ denotes uniform measure on $[0,T]$ and we sample $s$ from $\pi_T$.  Then consider the measure $\Gamma \times \pi_T$.  Once we have a sample from this measure, we may use it to generate a pair $(\bar h^t, \bar \eta^t) = (h^{t+s},\eta^{t+s})$. That is, $(\bar h^\cdot, \bar \eta^\cdot)$ is the original process $(h^\cdot, \eta^\cdot)$ translated by the random quantity $s$, which was chosen uniformly from $[0,T]$.
The stationarity of $\Gamma$ implies that $(\bar h^\cdot, \bar \eta^\cdot)$ also has the law $\Gamma$.  Now suppose that we instead sample from the weighted measure $\tilde \Gamma = e^{(h^s, \density)}\Gamma \times \pi_T$ (normalized to be a probability measure). Since $e^{(h^s, \density)} = e^{(\bar h^0, \density)}$, the induced measure on $(\bar h^\cdot, \bar \eta^\cdot)$ has the law $e^{(h^0,\density)} \Gamma$.

However, sampling from $\tilde \Gamma$ can also be done by first choosing the process $(h^\cdot, \eta^\cdot)$ from its marginal law, which is the measure $\Gamma$ weighted by $\int_0^T  e^{(h^t, \density)}dt$ (which converges in law to $c \Gamma$ --- normalized to be a probability measure --- in the total variation sense as $T \to \infty$, by the $L^1$ ergodic theorem) and then sampling $s$ from $e^{(h^t, \density)}\pi_T$ (normalized to be a probability measure), which amounts to choosing a uniform $s_\density$ time from the long interval.  Since, as noted above, the process centered at this random time has the law $e^{(h^0,\density)} \Gamma$, we find, by letting $T$ tend to infinity, that $e^{(h^0,\density)} \Gamma$ must be stationary with respect to $s_\density$ time.\end{proof}

Recall from Section \ref{ss.randomgeometryGFF} that adding $-2\pi \Delta^{-1} \density$ to $h$ is equivalent to changing the underlying geometry w.r.t.\ which Liouville quantum gravity is defined.  Thus, Proposition \ref{p.densitystationary} says that a Liouville quantum gravity measure on an alternative underlying geometry ($h$ still taken modulo additive constant) is stationary under zipping/unzipping in a modified version of capacity time.

This symmetry between space and time is intriguing for its own sake.  It is related to but should not be confused with Theorem \ref{reversecouplingkapparho}.  Both theorems involve modifying the function $\h_0$.  However, Theorem \ref{reversecouplingkapparho} does not describe a new stationary process; instead, it gives a modified law for the zipping up map $f_t$ and does not involve changing the time parameter.  On the other hand, one can use Theorem \ref{reversecouplingkapparho} to describe how the process of Proposition \ref{p.densitystationary} evolves in capacity time.  The following is an immediate consequence of Theorem \ref{reversecouplingkapparho}.

\begin{proposition} \label{p.kapparhozipping}
Given a signed measure $\density$ on $\H$ with $(-\Delta^{-1} \density, \density) < \infty$, let $\Gamma_\density$ denote the measure whose Radon-Nikodym derivative with respect to $\Gamma$ is proportional to $e^{(h^0, \density)}$ (normalized to be a probability measure).  Then the $\Gamma_\density$ law of the zipping up map $f_t$, for $t \geq 0$, is the law of the modified SLE process in Theorem \ref{reversecouplingkapparho}.  Moreover, given $f_t$ for some time $t \geq 0$, the $\Gamma_\density$ conditional law of $h^0$ is that of $\widetilde h \circ f_t +{\hat \h}_t$, as defined in the statement of Theorem \ref{reversecouplingkapparho}.
\end{proposition}

\subsection{Conclusion of zipper proofs}
\label{ss.zipperconlusion}

\begin{proposition} \label{p.gammawedge2}
The conclusion of Proposition \ref{p.gammawedge} (see Figure \ref{boundaryx}) remains true even if, when we choose $h$ and $x$, we condition on a particular pair of values $L_1 = \nu_h [a,x]$ and $L_2 = \nu_h[x,b]$.
\end{proposition}
\begin{proof}
We use the setup and notation of Proposition \ref{p.gammawedge}.  Recall the explicit form of the weighted measure $\nu_h[a,b]dh$ given in \cite{2008arXiv0808.1560D}, as described earlier.  Once we condition on $x$, the conditional law of $h$ is simply that of a zero boundary GFF (except with free boundary conditions along $\R$) {\em plus} a random harmonic function {\em plus} $-\gamma \log|x - \cdot|$. Roughly speaking, the proposition follows from the fact that the restriction of $h$ to an extremely small neighborhood of $x$ (which tells us what the quantum surface looks like when we zoom in near $x$) is almost independent of the pair $(L_1, L_2)$.

To express this point more carefully, consider smooth functions $\phi_1$ and $\phi_2$ on $\partial D$ that are supported on disjoint neighborhoods $U_1 \subset D$ and $U_2 \subset D$ with $\overline{U_1} \cap \R \subset (a,x)$ and $\overline{U_2} \cap \R \subset (x,b)$.  We assume (for $i \in \{1,2\}$) that each $\phi_i$ is equal to one on some interval of $\overline{U_i} \cap \R$ and $\phi_i(z) \in [0,1]$ for all $z \in D$, and also that $x \not \in \overline{U_1} \cup \overline{U_2}$.

Now, by the definition of the GFF we can write $h = \alpha_1 \phi_1 + \alpha_2 \phi_2 + h_0$ where $\alpha_1$ and $\alpha_2$ are centered Gaussian random variables and $h_0$ is projection of $h$ onto the orthogonal complement of the span of $\phi_1$ and $\phi_2$, so that $(h_0, \phi_1)_\nabla = (h_0, \phi_2)_\nabla = 0$ almost surely.  In this construction, $\alpha_1$, $\alpha_2$ and $h_0$ are independent of each other.  Recalling \eqref{e.nuhplusphi}, we have, for $i \in \{1,2\}$, that the restriction of $\nu_h$ to $\partial U_i$ is given by $$\nu_h = e^{\frac{\gamma}{2} \alpha_i \phi_i}\nu_{h_0}.$$

In particular, this implies that once we condition on $h_0$, each $\nu_h(\partial U_i \cap \R)$ is a.s.\ given by an increasing smooth function of $\alpha_i$. Here the smoothness can be verified by differentiating with respect to $\alpha$ and noting that no matter how many times one differentiates one obtains a compactly supported test function integrated against $\nu_{h_0}$. This in particular implies that, once we condition on $h_0$, the quantity $\nu_h(\partial U_i \cap \R)$ has a law which is absolutely continuous with respect to Lebesgue measure on $(a,\infty)$ (for some $a$ that depends on $h_0$) and has a smooth density function.  Let $\psi_1$ and $\psi_2$ be the density functions for these laws.  If we fix $h_0$ outside of some small $B_{\bar \eps}(x)$ disjoint from $U_1 \cup U_2$ but rerandomize the restriction $h_0'$ of $h_0$ to $B_{\bar \eps}(x)$ --- conditioned on $(L_1, L_2)$ --- then this is the same as rerandomizing $h_0'$ without conditioning on $(L_1,L_2)$, except that we weight the law by (a quantity proportional to) the product below (viewed as a function of $h'_0$) $$\psi_1\Bigl(L_1 - \nu_h \bigl([a,x] \setminus \partial U_1 \bigr) \Bigr) \psi_2\Bigl(L_2 - \nu_h \bigl([x,b] \setminus \partial U_2 \bigr) \Bigr).$$  As $\bar \eps$ tends to zero, the amount by which resampling $h_0'$ changes either $\nu_h[a,x]$ or $\nu_h[x,b]$ is a quantity that tends to zero in probability. We conclude that weighting the law of $h_0'$ by the big expression above (which is a smooth bounded function of $\nu_h[a,x]$ and $\nu_h[b,x]$) affects this law by an amount that (in total variation sense) tends to zero as $\eps \to 0$.
\end{proof}

In order to prove Theorem \ref{t.lengthstationary}, we will follow the argument sketched in Figure \ref{triplezoom} in Section \ref{ss.lengthstationary}.  We begin with a lemma:

\begin{lemma} \label{l.weightedconditionalzipping} Suppose we first sample $(h^t, \eta^t)$ from $\nu_h[-\delta,0]\Gamma$ (normalized to be a probability measure; here we write $h=h^0$) and then sample $x$ from $\nu_h[-\delta, 0]$.  Consider the joint law of the $x$ and the $(h^t, \eta^t)$ process chosen in this way.  Then given $x$, the conditional law of the zipping up process $f_t$ is that of the modified SLE process in Theorem \ref{reversecouplingkapparho}, with the $\density$ measure given by $\gamma^2$ times a Dirac distribution at $x$ minus $\gamma^2$ times a unit of uniformly distributed mass on the unit circle.
Moreover (as in Proposition \ref{p.kapparhozipping}), given $f_t$ for some time $t \geq 0$, the conditional law of $h^0$ is that of $\widetilde h \circ f_t +{\hat \h}_t$, as defined in the statement of Theorem \ref{reversecouplingkapparho}.
\end{lemma}

\begin{proof}
Recall \eqref{e.nudef} and define
\begin{equation} \label{e.nueps} \nu^\eps_h :=
\eps^{\gamma^2/4}e^{\gamma h_\eps(x)/2}dx.\end{equation}
Suppose that for some positive $\delta' < \delta$ and $\eps < \delta'$ we weight by $\nu^\eps_h[-\delta,-\delta']$ instead of $\nu_h[-\delta, 0]$ --- and then, given $h$, we choose $x$ from $\nu_h^\eps$ restricted to $[-\delta,-\delta']$ and normalized to be a probability measure.  Then we know that this is the same as sampling from \begin{equation} \label{e.egammaheps} e^{\gamma [h_\eps(x)-h_1(0)]/2} dxdh\end{equation} (times a normalizing constant) where $x$ is Lebesgue on $[-\delta, \delta']$ and $dh$ is the $\Gamma$ law of $h$.  (We have $h_\eps(x) - h_1(0)$ instead of $h_\eps(x)$ in the exponent because the former is defined independently of the additive constant for $h$, and is what we obtain if we choose the additive constant so that $h_1(0)=0$.)  Thus, given $x$, the conditional law of $h$ is $e^{\gamma ( h_\eps(x)-h_1(0))/2}dh$, and Proposition  \ref{p.kapparhozipping} exactly determines the law of $f_t$ (where the relevant $\density$ measure is a uniform measure on $\partial B_\eps(x)$ minus a uniform measure on $\partial B_1(0)$).

The above gives an explicit way to construct a sample from $\nu^\eps_h[-\delta, -\delta']$ and the corresponding $f_t$.  To establish the lemma, we need to argue that an analogous result holds when $\nu^\eps_h$ is replaced by $\nu_h$.  One way to deduce the $\nu_h$ result from the $\nu^\eps_h$ result is as follows.  Let $U$ be the $\eps$ neighborhood of $[-\delta, -\delta']$ in $\H$, as in Figure \ref{Usketch}.  Decompose the free boundary GFF Hilbert space into the space $\Supp_U$ (the closure of the space of smooth functions supported on $U$ that vanish on $\partial U \setminus \R$, with free boundary conditions on $\partial U \cap \R$) and the space $\Harm_U$ (the orthogonal complement of $\Supp_U$, which consists of functions that are harmonic on $U$ with Neumann boundary conditions along $\R$ \cite{SchrammSheffieldGFF2}).  Let $h^{\mathrm{HARM}}$ be the projection of $h$ onto $\Harm_U$ and $h^{\mathrm{SUPP}}$ the projection of $h$ onto $\Supp_U$, so that \begin{equation} \label{e.hdecomp} h =h^{\mathrm{HARM}}+ h^{\mathrm{SUPP}}.\end{equation}  Roughly speaking, $h^{\mathrm{HARM}}$ represents the conditional expectation of $h$ given all of the values of $h$ outside of $U$.  Note that $h^{\mathrm{SUPP}}$, which is independent of $h^{\mathrm{HARM}}$, has the law of a GFF on $U$ with zero boundary conditions on $\partial U \setminus \R$ and free boundary conditions on $\partial U \cap \R$.

\begin {figure}[htbp]
\begin {center}
\includegraphics [width=3in]{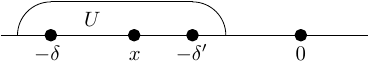}
\caption {\label{Usketch} $U \subset \H$ is an $\eps$ neighborhood of $[-\delta, -\delta']$.}
\end {center}
\end {figure}

If we take any $\eps' < \eps$ then the measure $e^{\gamma h_{\eps'}(x)/2}dh$ (normalized to be a probability measure) induces a law for $h^{\mathrm{HARM}}$ that does not depend on $\eps'$.  We claim that the $\Gamma$ law of $f_t$ (up until the first time $T_U$ that a point in $f_t(U)$ reaches the origin) is independent of the projection of $h=h^0$ onto $\Supp_U$.  This follows from Theorem \ref{reversecoupling} and the way we constructed $\Gamma$: we may sample $h$ and $f_{T_U}$ by first sampling $f_{T_U}$, then sampling $h_{T_U}$ from the law of $\widetilde h + \h_0$, and then using a coordinate transformation via $f_{T_U}^{-1}$ to recover $h^0$.  The decomposition \eqref{e.hdecomp} can be made using $f_{T_U}(U)$ instead of $U$ before we apply this coordinate transformation.  It then follows from the conformal invariance of the zero boundary GFF (with free boundary conditions along $\R$) that the $h=h^0$ we obtain after the coordinate transformation has a projection to $\Supp_U$ that is independent of $f_{T_U}$.

Thus, the conditional law of $f_{T_U}$ given $h$ depends only on the projection of $h$ onto $\Harm_U$.  Since weighting $\Gamma$ by $\nu^{\eps'}_h[-\delta, -\delta']$ and by $\nu_h[-\delta, -\delta']$ affects the law of the projection of $h$ to $\Harm_U$ in an identical way, the law of this projection, and hence the law of $f_t$ (up until $T_U$), is the same in each weighted measure.  Since this holds for any interval $[-\delta,-\delta']$ we conclude that the joint law of $x$ and $f_t$ (up until $T_U$) does not change if we replace $\nu_h^{\eps}$ with $\nu_h$.  To be explicit now about what we get from Proposition  \ref{p.kapparhozipping}, recall from \eqref{e.dWkapparho} and the subsequent discussion that we have $d f_t(z)  =  \frac{-2}{f_t(z)}dt - dW_t,$ where $$ dW_t = \bigl( -\Re (f_t)^{-1}, \density \bigr)dt + \sqrt{\kappa}dB_t,$$ in the case that we weight by  $e^{(h, \frac{1}{2\sqrt{\kappa}} \density)}$, the weight from \eqref{e.densityweight}.  In our case, we take $\density$ to be such that $$(h, \frac{1}{2\sqrt{\kappa}} \density) = \frac{\gamma}{2} (h_{\eps}(x) - h_1(0)),$$ which implies $$(h, \density) = \gamma^2 (h_{\eps}(x) - h_1(0)),$$
and the lemma follows by taking $\eps$ to zero.
\end{proof}

\proofof{Theorem \ref{t.lengthstationary}}
We begin by analyzing the construction described in Lemma \ref{l.weightedconditionalzipping} in more detail.  By Lemma \ref{l.weightedconditionalzipping}, we find that if $h$ and $x$ are taken from the measure \eqref{e.egammaheps}, then given $x$, the law of the zipping up procedure in Figure \ref{triplezoom} (up until $f_t(x+\eps)$ reaches zero) is determined by \begin{equation} \label{e.dwtcirclepoint} dW_t =  \Bigl(- \frac{\gamma^2}{f_t(x)} + \Re \int_0^\pi \frac{ \gamma^2}{ f_t(e^{i\theta})} \frac{d\theta}{\pi} \Bigr) dt + \sqrt{\kappa} dB_t.\end{equation}
As explained in the proof of Lemma \ref{l.weightedconditionalzipping}, this law is independent of $\eps$ and also holds when $\nu^{\eps}_h$ is replaced by $\nu_h$ in the construction of $x$ and $h$ above.

Now if the second term in the expression for the $dt$ piece of \eqref{e.dwtcirclepoint} were not there, then \eqref{e.dwtcirclepoint} would correspond to the one-force-point SLE$_{\kappa, \rho_1}$ with $\rho_1 = \gamma^2 = \kappa$, and the corresponding $f_t(x)$ would evolve precisely as a Bessel process of (recall \eqref{e.besseldim}) dimension
\begin{equation}\label{e.actualbesseldim} 1 + \frac{2(\gamma^2-2)}{\kappa} = 3-\frac{4}{\gamma^2} < 2,
\end{equation}
which would imply that $f_t(x)$ eventually reaches zero almost surely \cite{MR2000h:60050}.  Moreover, by \eqref{e.rhorho'}, the zipped up curve obtained, at the time $x$ hit zero, would look like ordinary SLE up to some time with a size of order $\delta^2$.  For the next point in the proof it is (slightly) easier if we replace the mean value on $\partial B_1(0) \cap \H$ (in Figure \ref{triplezoom} and in \eqref{e.dwtcirclepoint}) with the mean value on an arc of $\partial B_1(0) \cap \H$ that is bounded away from $\R$.  (The statement and proof of Lemma \ref{l.weightedconditionalzipping} do not change if we replace $\partial B_1(0) \cap \H$ with such an arc.)  In this case, since $\Re f_t(z)$ is increasing in $t$ for each $z \in \H$, the magnitude of the second term in the $dt$ piece of \eqref{e.dwtcirclepoint} is a.s.\ bounded by a constant for all $t$.  Thus when we zoom in (as in Figure \ref{triplezoom}), the effective drift from this term in the zoomed-in process tends uniformly to zero.

We now claim that the limit sketched on the right in Figure \ref{triplezoom} is a $(\gamma-2/\gamma)$-quantum wedge with an independent SLE$_\kappa$ curve. Lemma \ref{l.weightedconditionalzipping} implies that the limiting curve is an SLE$_\kappa$ independent of the field. To see that the field corresponds to a quantum wedge, note that before rescaling it looks like a GFF with a $(\gamma-2/\gamma) (-\log|\cdot|)$ singularity at the origin. (Recall that the $-2\gamma$ corresponds to the log singularity present in the capacity invariant model, and the $\gamma$ comes from the extra force point that has just collided with the origin.) The fact that the limit is then a quantum wedge follows from the argument in Proposition \ref{p.gammawedge} (which shows generally that if $h$ looks like a free boundary GFF plus a log singularity at some point, then one obtains a quantum wedge when one zooms in near that point). We also claim that in this limiting object (the $(\gamma-2/\gamma)$-quantum wedge decorated by an independent SLE$_\kappa$ curve) the two quantum surfaces divided by the curve each have the law of a $\gamma$-quantum wedge. This follows for the left side from Proposition \ref{p.gammawedge} and for the right side by symmetry. (Note that a $(\gamma-2\gamma)$-quantum wedge decorated by an independent SLE$_\kappa$ is an object with left-right symmetry.) 

At this point, we need to show that the quantum length measures along the two sides of $\eta$ almost surely agree. Since $x$ was sampled uniformly from quantum measure (before we zipped up and zoomed in) the SLE$_\kappa$ decorated $(\gamma-2/\gamma)$ quantum wedge must be invariant under the operation of unzipping by a fixed quantity of quantum boundary length as measured along the left of the two $\gamma$ quantum wedges. (This is similar to the argument used to show Proposition \ref{p.gammawedgeinvariant}.) Thus the pair $\bigl( (D_1, h_1), (D_2, h_2) \bigr)$ is invariant under zipping and unzipping by the boundary length measure.

Let $F(s)$ denote the quantum length (as measured from the right side) of the curve segment $\eta([0,t])$, where $t$ is chosen so that the quantum length (as measured from the left side) of $\eta([0,t])$ is $s$.  The ergodic theorem implies that $\lim_{s \to \infty} F(s)/s$ exists almost surely; denote this (possibly random) limit by $c$.  Since the law of the $(\gamma-2/\gamma)$-quantum wedge is scale invariant, the fact that $F(s) \approx cs$ holds for large scales (by the ergodic theorem) implies that the same holds on all scales, and indeed we must have identically $F(s) = cs$.  Since $c$ is determined by the restriction of $\eta$ and $h$ to arbitrarily small balls centered at zero. In other words, if $\mathcal F_\delta$ denotes the $\sigma$ algebra generated by the restiction of $h$ to $B_\delta(0)$ and the curve $\eta$ stopped when it first exits $B_\delta(0)$, then $c$ is measurable w.r.t.\ the intersection $\cap_{\delta > 0} \mathcal F_\delta$. However, it is easy to see that events in $\cap_{\delta > 0} \mathcal F_\delta$ have probability zero or one, so that $c$ must be a.s.\ constant. (This is proved explicitly in Lemma 8.2 of \cite{wedgespaper} in a setting that includes the field but not the path $\eta$. The statement including the path $\eta$ then follows from the well known fact that if $\mathcal G_\delta$ is generated by the restriction of Brownian motion to $[0, \delta]$ then events in $\cap_{\delta > 0}\mathcal G_\delta$ have probability zero or one.)  By the left-right symmetry of the law of SLE$_\kappa$ and the quantum wedge, we have $c = 1$ almost surely.

Now that we know that the length measures on the two sides agree, we next claim that each of the two sides in the limit sketched on the right in Figure \ref{triplezoom} are {\em independent} $\gamma$-quantum wedges.  We will deduce this by applying Proposition \ref{p.gammawedge2} to Figure \ref{triplezoom}.  More precisely, we consider the setting of the upper image in Figure \ref{triplezoom} and we condition on the restriction of the GFF to the complement of $B^1:= B_{\overline \eps}(x)$ and $B^2:= B_{\overline \eps}(R(x))$, and we also condition on the quantum lengths $\nu_h (B^1 \cap \R)$ and $\nu_h (B^2 \cap \R)$. We may choose $\delta$ small enough so that with high probability $R(x) \in B_1(0)$, and we will focus our attention on the event that this is the case. If $\overline \eps$ are small then with high probability these balls do not contain the origin and are contained in $B_1(0)$.

If we re-sample the restriction of $h$ to $B^1 \cap \H$ --- conditioned on $\nu_h  (B^1 \cap [x,\infty) )$ and $\nu_h(B_1 \cap (-\infty, x])$ and the restriction of $h$ to the complement of $B^1$ --- then Proposition \ref{p.gammawedge2} implies that even with this conditioning it is a.s.\ the case that the zoomed-in figures (as in lower left in Figure \ref{triplezoom}) converges in law to a $\gamma$-quantum wedge as $\overline \eps \to 0$.  The conditional law of $h$ (just given its values outside of $B^1$) is that of a random harmonic function plus a GFF on $B^1 \cap \H$ with free boundary conditions on $\R$ and zero boundary conditions on $\partial B^1 \cap \H$ \cite{Sh}. 

The same applies when we zoom in near $R(x)$. Now if we condition on the GFF values on $\partial B_1(0)$ together with the values on the line $i \mathbb R$, together with the lengths of $[-1,x]$ and $[x,0]$ and $[0,R(x)]$ and $[R(x), 1]$, then the conditional law of the restrictions of $h$ to the two halfs of $B_1(0) \cap \H$ are independent by the standard GFF Markov property.  A similar analysis to the above shows that even with this conditioning, one still obtains the laws of quantum wedges, near each of $x$ and $R(x)$, upon zooming in near those points. We conclude that in the $\overline \epsilon \to 0$ limit the $\gamma$-quantum wedges are independent of each other. This completes the proof of the items in the first paragraph of the theorem statement. Along the way we have also established the invariance of the law of the pair $\bigl( (h_1, D_1), (h_2 D_2) \bigr)$ under zipping up or down by a unit of quantum length  zipper stationarity, and the $\zlength_t$ properties are now immediate from this.
%It now follows by applying Proposition \ref{p.gammawedge2} to Figure \ref{triplezoom} that the two $\gamma$-quantum wedges are indeed independent of one another.
%The invariance of the law of the pair $\bigl( (h_1, D_1), (h_2 D_2) \bigr)$ under zipping up or down by a unit of quantum length now follows from the Proposition \ref{p.gammawedgeinvariant}.
  %  We will use some symmetry arguments to obtain the fact that both sides become $\gamma$-quantum wedges in the limit and that their joint law is symmetric under reflection across the vertical axis through zero.
 \qed
\vspace{.1in}

Deriving Theorem \ref{conformalwelding} as a consequence of Theorem \ref{t.lengthstationary} is now fairly straightforward.  Essentially one uses absolute continuity of the corresponding fields (at least when restricted to certain compact sets, away from the origin) to say that that if the left and right boundary lengths a.s.\ agree in the setting of Theorem \ref{t.lengthstationary} then they almost surely agree in the setting of Theorem \ref{conformalwelding} as well.

\proofof{Theorem \ref{conformalwelding}}
The above proof shows that if one draws an independent SLE$_\kappa$ on a $(\gamma-2/\gamma)$-quantum wedge, the quantum lengths measured along the left and right sides of the curve agree almost surely.  The setting of Theorem \ref{conformalwelding} is different because the law of $h$ is different. On the other hand, we would like to argue that it is not {\em so} different --- i.e., that the laws of the two random fields (restricted to a compact set, bounded away from the boundary) are absolutely continuous w.r.t.\ each other, so that a statement that is a.s.\ true for one is a.s.\ for the other.

That is, we would like to say that the $h$ of Theorem \ref{conformalwelding} ``looks like'' the canonical description $h$ of a $(\gamma-2/\gamma)$-quantum wedge, at least in the sense of absolute continuity of the restriction to compact subsets of $\H$.  This is most immediate for compact sets that lie outside of $B_1(0)$.  If we consider a canonical description $h$ of a quantum wedge, then we know that by definition $\mu_h(B_1(0)) = 1$. However, once we condition on the restriction of $h$ to $B_1(0)$, the conditional law of $h$ restricted to $\H \setminus \overline{B_1(0)}$ is that of a GFF (with zero boundary conditions on $\partial B_1(0)$, free elsewhere) plus a certain random smooth function on $\H \setminus \overline{B_1(0)}$.   (Note that if $h$ is a canonical description, it will remain a canonical description even after we resample its values outside of $B_1(0)$ in this way.)  In particular, the restriction of $h$ to some compact $D \subset \H \setminus \overline{B_1(0)}$) is absolutely continuous with respect to the restriction to $D$ of the $h$ of Theorem \ref{conformalwelding}, defined up to additive constant. (This is immediate from the absolutely continuity results in Section 3.1 of \cite{SchrammSheffieldGFF2}.)  We thus find that in the setting of Theorem \ref{conformalwelding} the left and right quantum lengths along $\eta$ are almost surely equal, at least outside of $B_1(0)$.  The object in Theorem \ref{conformalwelding} is scale invariant; hence, if the left and right quantum boundary lengths a.s.\ agree outside of a unit ball of radius $1$, then they a.s.\ agree outside of a unit ball of any radius, which means that they a.s.\ agree everywhere along the SLE$_\kappa$ curve. \qed
\vspace{.1in}

\begin{remark} The above argument also holds if we add any smooth function to $h$ --- which affects the restriction of $h$ to each compact subset of $\H$ in an absolutely continuous way (see \cite{Sh,SchrammSheffieldGFF2}) --- and then independently draw any path whose law is absolutely continuous w.r.t.\ that of SLE. That is, even in this setting one can use absolute continuity results about the GFF to prove that the quantum lengths along the two sides of the path agree almost surely.  In particular, Theorems \ref{conformalwelding} and \ref{conformalweldingunique} still apply in the setting of Theorem \ref{reversecouplingkapparho}.\end{remark}

\section{Twenty questions} \label{s.questions}
 {\it Update:  Since the first version of this paper was posted to the arXiv in 2010, there has been progress on several of the questions listed below. The current version leaves the questions as they were in 2010 but includes brief updates on work completed since.}
\vspace{.1in}

Many of the most fundamental questions about quantum gravity are open.  Before formulating some of these questions, we present some definitions that will appear in the questions.

We have already seen that quantum wedges are natural random quantum surfaces of infinite area and infinite boundary length.  We now describe some natural random quantum surfaces of unit area or unit boundary length, in terms of limits.  We will say more about the sense in which the limits exist below, and more detailed constructions appear in \cite{wedgespaper}.

\begin{enumerate}
\item Fix a smooth bounded domain $D$.  Let $h$ be a GFF with free boundary conditions on a linear segment $L$ of $\partial D$ and zero boundary conditions on $\partial D \setminus L$.  Fix $C>0$ and condition on $\mu_h(D) = C$.  Let $\hat h =  h - (\log C)/\gamma$, so that $\mu_{\hat h}(D) = 1$.  We claim that as $C \to \infty$, the law of $(D, \hat h)$ (viewed as a quantum surface) tends to a limit that does not depend on $D$ or $L$.  Call this the {\bf unit area quantum disc}.
\item The {\bf unit boundary length quantum disc} is defined the same way except that we condition on $\nu_h(L) = \sqrt C$ and the normalization gives $\nu_{\hat h}(L) = 1$.
\item The {\bf unit area quantum sphere} is defined the same way as the unit area quantum disc except that we take $D$ to have zero boundary conditions on all of $D$.
\end{enumerate}

In the case of the discs (the first and second objects described above), one way to formulate the convergence is to fix a point $x \in D \setminus L$ and then to choose $x_1$ and $x_2$ on $L$ at random from $\nu_{\hat h}$.  We change coordinates via \eqref{Qmap} to $\H$ so that $x_1$, $x_2$, and $x$ map respectively to $0$, $1$, and $\infty$ (see the left side of Figure \ref{unitsize}).  We then obtain a random measure on $\H$, and these unit random measures should converge in law with respect to the topology of weak convergence of measures on $\H$.

\begin {figure}[htbp]
\begin {center}
\includegraphics [width=5in]{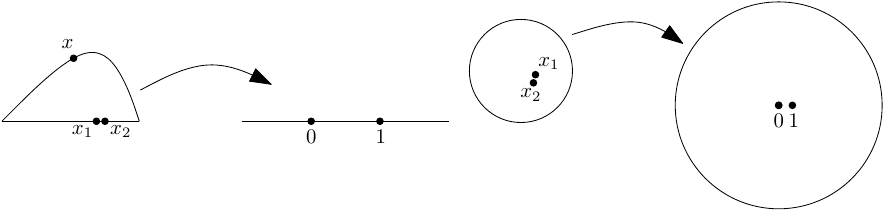}
\caption {\label{unitsize} Constructing unit area/length quantum disc (left) or unit area quantum sphere (right).  When $C$ is large, we expect $x_1$ and $x_2$ to be close with high probability, and for most quantum area to be near those points.}
\end {center}
\end {figure}

In the case of the sphere, one can choose two points $x_1$ and $x_2$ from $\mu_{\hat h}$ and conformally map $D$ to an origin-centered disc in such a way that one point goes to to $0$ and one to $1$.  (This determines the radius of the disc; see the right side of Figure \ref{unitsize}.)  We then obtain a random measure on $\C$ (which is incidentally supported on a disc of finite radius; the radius of this supporting disc tends to $\infty$ as $C \to \infty$), and these random measures should converge in law with respect to the topology of weak convergence of measures on $\C$.

We now present some conjectures and questions.  Some are very concrete and specific, and some are more speculative and open-ended.  Some of the physics questions have of course been extensively discussed in the physics literature already, but this literature is far too broad for us to survey here (though we can at least mention Polyakov's reflections, with references, on the history of the subject in \cite{2008arXiv0812.0183}, as well as the extensive reference list in \cite{2008arXiv0808.1560D}).

\vspace{.1in} \begin{center}{\bf Conjectures relevant to discrete model scaling limits   } \end{center}
\begin{enumerate}
\item \label{i.scalinglimitlist} Discrete random planar map models have Liouville quantum gravity as a scaling limit in the metric of weak convergence of area measures.  To be more specific, in addition to the conjectures in Section \ref{s.interpretation} and in \cite{2008arXiv0808.1560D}, we conjecture the following: \begin{enumerate}
    \item The uniform quadrangulation of a sphere with $n$ quadrilaterals scales to the unit area quantum sphere with $\gamma = \sqrt{8/3}$.  One way to formulate this is that if we choose three points uniformly on the quadrangulation and conformally map the quadrangulated surface to $\C$ with these three points going to $0$, $1$, and $\infty$, then the image of the area measure in $\C$ (normalized to the total area is one) converges in law (w.r.t.\ the topology of the local weak convergence) to the unit area quantum sphere measure described above.
    \item   The random quadrangulation of a disc with boundary length $n$, where the probability of a quadrangulation is proportional to a (critical) constant to the number of quadrilaterals, scales to the unit boundary length quantum disc with $\gamma = \sqrt{8/3}$.
    \item   The random quadrangulation of a disc with $n$ quadrilaterals, where the probability of a quadrangulation is proportional to a (critical) constant to the boundary length, scales to the unit area quantum disc with $\gamma = \sqrt{8/3}$.
    \item Similar statements hold for random quadrangulations weighted by the partition functions of Ising configurations, $O(n)$ configurations, uniform spanning trees, etc., with $\gamma = \sqrt{\kappa}$ for the appropriate $\kappa$.  In each case, the cluster boundary loops (of the statistical model on the random surface) scale to a CLE$_\kappa$ independent of the measure.  (Question: what happens if, for some $d$, we make the probability of a quadrangulation proportional to the partition function of the $d$-dimensional discrete Gaussian free field on the corresponding graph?)
    \end{enumerate}
{\it Update: The problem as stated remains open. However, convergence of FK-decorated random planar maps to CLE-decorated LQG has been established in the so-called peanosphere topology and in the so-called loop structure topology, in both infinite and finite volume settings \cite{2011arXiv1108.2241S, dms2014mating, quantum_spheres, gwynnemaosun, finitevolumeestimates, strongertopology, finitevolumelimit}, see also related results in \cite{berestyckilaslierray, multiburger, 2015arXiv150201013C, sunwilsonsandpile}. The construction of finite volume LQG spheres and disks (briefly considered above) is discussed in more detail in \cite{dms2014mating, lqg_sphere, quantum_spheres, twoperspectives}. Partial results in the direction of understanding the conformal structure of the discrete models appear in \cite{curien2014glimpse}. Roman Boikii and Stanislav Smirnov have also described progress in private communication.}
\item \label{i.quantummetric} A quantum surface is a coordinate-change-invariant random metric space in the sense of the footnote in Section \ref{ss.randomgeometryGFF}.  In the special case $\gamma = \sqrt{8/3}$ this is equivalent to the {\em Brownian map} (see e.g.\ \cite{MR2336042, MR2438999, MR2399286}), which is a term given to the random metric space scaling limit of uniform quadrangulations on the sphere.
{\it Update: The metric space structure of LQG has now been constructed in the pure quantum gravity case $\gamma = \sqrt{8/3}$ in recent work (some in progress) showing that $\sqrt{8/3}$-LQG is equivalent to the Brownian map \cite{ms2013qle, qlebm,qle_continuity, qle_determined}. No such construction is currently available for other values of $\gamma \in (0,2)$.}
\item Assuming that \ref{i.quantummetric} holds, the scaling limits in \ref{i.scalinglimitlist} also hold in the topology of Gromov-Hausdorff convergence of metric spaces.
{\it Update: In light of the equivalence of $\sqrt{8/3}$-LQG and the Brownian map  \cite{ms2013qle, qlebm,qle_continuity, qle_determined}, and the Gromov-Hausdorff convergence of discrete random planar maps to the Brownian map \cite{legalluniqueanduniversal, legalluniversallimit,miermontlimit}, this conjecture has now been established in the pure quantum gravity case $\gamma = \sqrt{8/3}$. It remains open for all other $\gamma \in (0,2)$.}

\vspace{.1in} \begin{center}{\bf Related but more modest requests:} \end{center}

\item Show that simple random walk on a random quadrangulated surface is a good approximation for Brownian motion on that surface.
{\it Update: This problem remains open. However, works by Gill and Rohde, by Gurel-Gurevich and Nachmias, and by Benjamini and Curien have derived some fundamental properties for random walks and Brownian motion on infinite volume versions of these surfaces \cite{GR, gurelgurevichnachmias, benjamini2013simple}.}
\item Prove any SLE convergence result at all for random planar maps. {\it Update: As mentioned above, there are now a number of convergence results involving the peanosphere and loop structure topologies.}
\item Show that the following toy model has a scaling limit which is a metric space.  Begin with a unit square $S$.  After an exponentially random amount of time, divide $S$ into new four squares of equal size.  Each time a new square is created, give it a new exponential clock (independent of the others) and so that it too will divide into four new squares after an exponentially random amount of time.  (Each time a clock at an existing box rings, the total number of boxes increases by $3$.  Thus if $N(t)$ is the number of boxes at time $t$, then $e^{-3t} N(t)$ is a martingale.)  After time $t$, consider two squares to be adjacent if they have part of an edge in common.  Show that for some $\beta>0$, the graph metric on the set of squares times $e^{-\beta t}$ converges almost surely to a random metric space parameterized by $S$. {\it Update: This toy problem remains completely unsolved. However, we remark that another interesting model for producing random squarings appears in \cite{berryleavitt}.}

\vspace{.1in} \begin{center}{\bf Structural questions about Liouville quantum gravity and AC geometry:} \end{center}

\item If a quantum surface is well defined as a metric space, what does the boundary of a unit ball look like?  Is it a form of SLE (some collection of SLE loops) or something completely different?  How about a shortest (geodesic) path between two points $z_1$ and $z_2$?  How about the boundary of the set of points closer to $z_1$ than to $z_2$?  What is the continuum analog of the breadth first search tree (the tree that appears in the Schaeffer bijection \cite{MR1465581})?  Can one even formulate a conjecture?  {\it Update: As mentioned above, much has been now been established in the case $\gamma = \sqrt{8/3}$, where the ball boundaries are constructed using the quantum Loewner evolution \cite{ms2013qle, qlebm,qle_continuity, qle_determined}, and a canonical embedding of the Brownian map in the sphere has been constructed. The problem remains completely open for other $\gamma \in (0,2)$.}
\item If $\kappa \in (4,8)$ then the complement of SLE$_\kappa$, run to time $T$, is not a single simply connected domain; it consists of infinitely many components.  The set of components that lie to the left side of $\eta$ comes with a tree-like hierarchical structure (where a component $A_1$ is ``above'' a component $A_2$ if $\eta$ traces $\partial A_1$ after it has started tracing, but before it has finished tracing, $\partial A_2$).  The set of components on the right side comes with a similar structure.   This suggest that when we ``unzip'' along $\eta$ (as in Theorem \ref{reversecoupling}) we obtain not merely a quantum surface parameterized by $\H$ but a quantum surface parameterized by $\H$ together with a tree-like structure of quantum-surface ``beads'' hanging off of its boundary.  Does an analog of Theorem \ref{conformalwelding} hold in this setting if one properly defines the boundary length of the tree-like structure?  How about an analog of Theorem \ref{conformalweldingunique}?  Proving the latter would likely require first showing that SLE$_{\kappa}$ is removable when $\kappa \in (4,8)$.  Is this the case?  Can a $\kappa \in (4,8)$ version of the quantum zipper be used to prove the time reversal symmetry of SLE for $\kappa \in (4,8)$? {\it Update: The time-reversal symmetry results were established in the imaginary geometry papers \cite{ms2012imag1,ms2012imag2,ms2012imag3,ms2013imag4}, and the questions about the $\kappa \in (4,8)$ analogs of Theorems \ref{conformalwelding} and \ref{conformalweldingunique} were answered affirmatively in \cite{dms2014mating}. Interestingly, the 		zipping up'' results in \cite{dms2014mating} were established without a proof that SLE$_{\kappa}$ is removable when $\kappa \in (4,8)$. It remains an open question whether these SLE$_{\kappa}$ processes are removable. A new result on the inversion symmetry of the quantum zipper welding (for the $\kappa < 4$ case) appears in \cite{rohde2013backward}.
}

\item What is the Hausdorff dimension of a quantum surface (assuming it is defined as metric space) for general $\gamma$?  Can one at least handle the special case $\gamma = \sqrt{8/3}$ (where the answer should be $4$, by analogy to discrete random triangulations; see e.g.\ \cite{MR2013797})? {\it Update: This question is settled in the case $\gamma = \sqrt{8/3}$, in light of the above-mentioned equivalence between the $\sqrt{8/3}$-LQG sphere and the Brownian map. In that case, the Hausdorff dimension is indeed $4$. For general $\gamma$, there is a Hausdorff dimension conjecture due to Watabiki \cite{watabiki1993analytic}, which has found some support in recent simulations \cite{2014NuPhB.889..676A}. This conjecture is further explained and discussed in \cite{ms2013qle}.}

\item Let $h^\eps$ be projections of $h$ onto the space of functions that are piecewise linear on an $\eps$-edge-length triangular lattice (or alternatively, mollifications of $h$ obtained by convolving with a bump function supported on a disc of radius $\eps$).  Do the flow lines of $e^{ih^\eps/\chi}$ converge in law to SLE curves (AC geometry flow lines) as $\eps \to 0$?  If so, this would extend the main result of \cite{MR2486487} to $\kappa \not = 4$.  Do the other results of \cite{MR2486487, SchrammSheffieldGFF2} for $\kappa = 4$ extend to $\kappa \not = 4$?
% {\em Update: There are still no proven methods for obtaining the flow lines associated to the continuum GFF as limits of flow lines associated to GFF mollifications or lattice projections.}

\item If $h^\eps$ are smooth mollifications or piecewise linear projections (as above) of the {\em free} boundary GFF of Theorem \ref{reversecoupling}, do the curves obtained by zipping up part of the boundary via conformal welding (using the identification map $R$ as in Section \ref{ss.confweldingintro}, except that we use the  approximate measures $e^{\gamma h^\eps(x)/2}dx$ instead of $\nu_h$ to define $R$) converge in law to SLE?

\item What happens when $\gamma = 2$ (so $\kappa = 4$, $Q = 2$)?  The definition of a quantum surface as an equivalence class --- as given in Section \ref{ss.randomgeometryGFF} --- makes sense for $\gamma = 2$, and as mentioned in the introduction.  Is the ``zipping up'' map in the coupling of Theorem \ref{reversecoupling} determined by $h$, as in the $\gamma < 2$ case?  Is it the limit of the curves one obtains by using the same GFF $\widetilde h$ but with different values of $\gamma$, and letting $\gamma$ approach $2$? Is SLE$_4$ almost surely removable, like SLE$_\kappa$ for $\kappa < 4$? {\it Update: this question remains open. However, recent results enable one to construct a quantum length measure $\nu_h$ (as well as a quantum area measure $\mu_h$) in the $\gamma=2$ case\cite{duplantier2012critical, duplantier2012renormalization}. It is natural to conjecture that zipping up gives a welding that respects this measure, and also that SLE$_4$ is removable. We note that a closely related result has been established by Tecu in \cite{tecu}, which builds on the work of \cite{2009arXiv0909.1003A}.}

\item Can the continuum limit of simple random walk on a random quadrangulation (or of Brownian motion on the corresponding Riemannian surface) be somehow understood directly, in terms of the Schaeffer or Mullin bijections, as a random process on an identified pair of continuum trees? {\em Update: The results in \cite{dms2014mating} and  \cite{ms2013qle, qlebm,qle_continuity, qle_determined} allow one to define a conformal structure (and hence a Brownian motion, with a time change defined using the ``Liouville Brownian motion'' theory developed in \cite{garbanrohdesvargas, berestyckibrownian}) on these mated pairs of continuum trees. The reader may decide whether these (SLE/LQG/QLE-based) constructions ought to be considered ``direct.''}

\item What is the most natural formulation of a higher genus quantum surface, where the conformal modulus is allowed to be random?  What is the right conjectural scaling limit of a random quadrangulation of a torus (or higher genus surface) with $n$ faces, as $n \to \infty$?  How about a surface with $n$ holes (weighting by the combined length of the hole boundaries in a critical way)?  One way to construct a quantum surface with holes is to put an independent conformal loop ensemble on top of a quantum surface and then cut out the regions surrounded by some of the loops.  (Here one would have to give a rule for specifying which loops to cut out --- for example, one might try to cut out the $k$ loops with the longest quantum boundary lengths.)  Is it possible to weld together surfaces constructed this way and produce higher genus surfaces that answer the questions above? {\it Update: See \cite{drv2015torus} for some recent work on LQG tori by David, Rhodes and Vargas, which cites many earlier works from the physics literature on higher genus LQG constructions. It remains an open problem to implement the ideas mentioned above (i.e., to construct canonical higher-genus LQG random surfaces via weldings) and to show that these approaches are equivalent to those discussed in \cite{drv2015torus}.}

\item One way to combine Liouville quantum gravity and AC geometry is to let $h$ be a complex valued free field.
 The real part of $h$ encodes Liouville quantum gravity and the imaginary part an AC geometry, which could be coupled with a conformal loop ensemble on the quantum surface \cite{MR2494457}.
 In this context, what new meaning, if any, comes out of the extra symmetries of the complex Gaussian free field (e.g., its invariance w.r.t.\ replacing $h$ with a modulus-one complex constant times $h$)?

\item Can one fully construct all rays in the AC geometry of an instance of the GFF (starting from all points in $D$) and determine which rays intersect each other and themselves, etc.?  Does the list of properties shown to hold for smooth $h$ in the appendix (Proposition \ref{smoothflowlinegeometry}) hold in the GFF setting as well? {\it Update: This question has been affirmatively answered in \cite{ms2012imag1,ms2012imag2,ms2012imag3,ms2013imag4}.}

\vspace{.1in} \begin{center}{\bf Returning to the original physics motivation$\ldots$} \end{center}

\item Can one construct a variant of Liouville quantum gravity with some of the additional complexities of a physical string theory (and figure out what the following things mean in this context mathematically)?
\begin{enumerate}
\item Minkowski space time.
\item Complex weights (in place of a probability measure).
\item Supersymmetry.
\item Gauge theories.
\item Embeddings in Calabi-Yau manifolds.
\item Quantum surfaces of non-deterministic genus.
\end{enumerate}

\item When introducing Liouville quantum gravity, Polyakov wrote in 1981 that it was necessary to develop a theory of random surfaces
``because today gauge invariance plays the central role in
physics.  Elementary excitations in gauge theories are formed by the
flux lines (closed in the absence of charges) and the time
development of these lines forms the world surfaces.  All transition
amplitude[s] are given by the sums over all possible surfaces with
fixed boundary.''  \cite{MR623209}  Can one formulate any version of this gauge/string duality as a theorem or conjecture
relating gauge theory flux lines to Liouville quantum gravity?
{\it Update: We mention one recent rigorous formulation of Yang-Mills gauge-string duality on a lattice due to Chatterjee (not obviously related to Liouville quantum gravity) \cite{2015arXiv150207719C}.}

\item Can any Liouville quantum gravity insights be used to study higher dimensional random metrics?  What are the most natural models?
{\it Update: We remark that certain higher dimensional analogs are easy to describe. If $h$ is an instance of the log-correlated Gaussian field (LGF) as defined e.g.\ in the surveys \cite{2014arXiv1407.5605D, 2014arXiv1407.5598L}, then $e^{\gamma h(z)}dz$ is easy to define as a random measure when $\gamma$ is in the right range, see e.g.\ the survey \cite{rvsurvey}. Furthermore, the restriction of a higher dimensional LGF to a (two-dimensional) plane yields a GFF on the plane, so that the higher dimensional LGF can be interpreted as a coupling of planar GFFs, one for each planar subspace; it remains unclear whether the imaginary geometry or level set structures corresponding to the individual slices can be unified in a coherent way \cite{2014arXiv1407.5598L}.}

\item Are there any natural three or four dimensional versions of AC geometry?  For example, is it possible to interpret paths in the three dimensional uniform spanning tree scaling limit \cite{MR2350070} as geodesics of a random affine connection? {\em Update: This remains open. We mention one attempt to describe a three-dimensional analog of SLE$_6$ (at least on the discrete level) via so-called tricolor percolation \cite{tricolor}, which also references relevant field theoretic constructions from the physics literature.
}
\end{enumerate}

\appendix

\section{AC geometry: real vs.\ imaginary Gaussian curvature}

This section gives a formal definition/interpretation/explanation of the AC geometry (and its relationship to Liouville quantum gravity) when $h:D \to \R$ is smooth.  Denote by $\flowset(h,D)$ the collection of curves that are flow lines of $e^{i(h/\chi + c)}$ (beginning at a point in $\overline D$),
for some $c \in [0,2\pi)$.  For each $x \in D$, there is a
one-parameter family of ``rays'' in $\flowset(h,D)$ (indexed by $c$)
that begin at $x$ and ultimately hit the boundary of $D$.  We refer
to $c$ as the {\em angle} of the ray.  Two rays of $D$ are {\em
parallel} if they have the same angle.  When $h$ is constant, these
are the rays of Euclidean geometry. In general, the collection of rays is quite
interesting and satisfies some of the axioms describing rays in Euclidean geometry.
For example, the reader may verify the following:

\begin{proposition} \label{smoothflowlinegeometry}
If $h$ is a Lipschitz function on a simply connected domain
$D\subset \C$, then the following hold:
\begin{enumerate}
\item Each ray beginning at $x \in D$ is a differentiable simple path (i.e., it doesn't intersect itself).
\item Two parallel rays in $\flowset(h,D)$ are disjoint unless one is a subset of the other.
\item Two non-parallel rays in $\flowset (h,D)$ intersect at most once in $D$.
\item If $x \in D$ lies on a ray beginning at $y \in D$, then $y$ lies on a ray of opposite
direction beginning at $x$.
\item The sum of the angles of a ``triangle''---whose edges are segments of rays in $\flowset(h,D)$---is
always $\pi$.
\end{enumerate} \end{proposition}

When $h$ is the GFF, the ``flow line'' described by the coupling in Theorem \ref{forwardcoupling} begins at a special location on the boundary of $D$, but in principle one would like to make sense of the entire set $\flowset(h,D)$ when $h$ is an instance of the Gaussian free
field and $\chi$ is a positive constant.  Even when $h$ is smooth, the AC geometry is not a Euclidean or non-Euclidean geometry in the usual sense; in particular, it does not come with a notion of length or distance.  We will interpret the AC rays of a smooth $h$ as the geodesics of a random torsion-free (defined below) affine connection that is naturally dual to the Levi-Civita connection of a Liouville quantum surface.

Recall that an {\em affine connection} determines, for any smooth path segment $P$ in $D$, an orthogonal map on the two-dimensional tangent space, which we may represent as multiplication by a complex number $\alpha(P)$.  Intuitively, it describes the way a small object transforms under parallel transport (``sliding without rotating'') along the path $P$.  The Levi-Civita connection for a Riemannian metric is (by definition) the unique metric-preserving and torsion-free affine connection.  For a two dimensional metric conformally parameterized by a subset of $D$ with area measure $e^{\gamma h(z)}dz$, metric preservation implies that
\begin{equation} \label{logalpha} \log |\alpha(P)| = \frac{\gamma}{2}h(z_1) - \frac{\gamma}{2} h(z_2),\end{equation} where $z_1$ and $z_2$ are the first and last endpoints of $P$.  In particular, this quantity is independent of the trajectory $P$ takes between $z_1$ and $z_2$.

We recall the following standard fact \cite{2008arXiv0808.1560D} (see also keywords ``isothermal coordinates'' and ``Gaussian curvature'' in any text on Riemannian surfaces): write $\lambda = \gamma h$
and note that given a measurable subset $A$
of $D$, the integral $\int_A e^{\lambda(z)} dz$ (where $dz$ denotes
Lebesgue measure on $D$) is the area of the portion of $\mathcal M$
parameterized by $A$.  The function $K = - e^{-\lambda} \Delta
\lambda$ (where $\Delta \lambda := \lambda_{xx} + \lambda_{yy}$ is
the Laplacian operator) is called the {\bf Gaussian curvature} of
$\mathcal M$.  If $A$ is a measurable subset of the $(x,y)$
parameter space, then the integral of the Gaussian curvature with
respect to the portion of $\mathcal M$ parameterized by $A$ can be
written $\int_A e^{\lambda(z)} K(z) dz = \int_A -\Delta \lambda(z)
dz$ where $dz$ denotes Lebesgue measure on $D$. In other words,
$-\Delta \lambda$ gives the density of Gaussian curvature in the
isothermal coordinate space.  In particular, $\mathcal M$ is flat if
and only if $h$ is harmonic.

We may define $\mathcal A(z,v)$ for each $z \in D$ and $v \in \C$ so that if $P$ is parameterized by $t \in [0,1]$ we have
$$\log \alpha(P) = \int_0^1 \mathcal A\bigl(P(t), \frac{\partial}{\partial t} P(t)\bigr) dt.$$
We may identify $\C$ with $\R^2$ so that $\nabla h$ is a vector-valued function of $\R^2$; given $v \in \C$, let $[\nabla h\cdot v]$ denote the dot product of this vector with the vector $v$ (viewed as an element of $\R^2$), which is a real number.

In the case of the Levi-Civita connection, we have $\Re \mathcal A(z,v) = \frac{-\gamma}{2} [\nabla h \cdot v]$.
Given this, $\Im \mathcal A(z,v)$ is determined by the requirement that the connection is torsion-free.
Using this notation, the statement that the connection is torsion-free means
that $\mathcal A(z,iv) = -i \mathcal A(z,v)$ for all $z$ and $v$, which implies
$$\mathcal A(z,v) = \frac{-\gamma}{2} [\nabla h \cdot v] + i\frac{-\gamma}{2} [\nabla h \cdot iv].$$
If $P$ is the boundary of a smooth region $R \subset D$, then Green's theorem implies that \begin{equation} \label{eqn::gcurvature} \log \alpha(P) = i \int_R -\frac{-\gamma}{2} \Delta h(z) dz.\end{equation}  The RHS of \eqref{eqn::gcurvature} is (up to a constant imaginary multiplicative factor) the integral of the Gaussian curvature over the region of the surface parameterized by $R$.  In fact, one may even {\em define} Gaussian curvature to be the function for which this the case.

The AC rays are the geodesics of a connection satisfying an analog of \eqref{logalpha}:
\begin{equation} \label{argalpha} \arg \alpha(P) = h(z_2)/\chi - h(z_1)/\chi,\end{equation}
and if we require that this connection also be torsion-free we obtain
$$\mathcal A(z,v) = i [\nabla h \cdot v]/\chi - [\nabla h \cdot iv]/\chi,$$
and hence, if $P$ is the boundary of a smooth region $R$, then $$\log \alpha(P) = \int_R (- \Delta h(z)/\chi) dz \in \R.$$
Intuitively, this means that a small object sliding around the path $P$ will undergo no net rotation, but will change in size (the opposite of what happens in a Levi-Civita connection of a Riemannian surface).  The definition of Gaussian curvature density suggested above, in terms of \eqref{eqn::gcurvature}, suggests an imaginary Gaussian curvature proportional to the real curvature obtained in Liouville quantum gravity.

We remark that the Laplacian of the GFF has a natural interpretation as the charge density of a two dimensional Coulomb gas (see the discussion and references in the introduction to \cite{MR2486487}).  Thus, Liouville quantum gravity can be interpreted as imposing a Gaussian curvature density equal to a real constant times this charge density.  AC geometry is the same, except that the constant is required to be imaginary.

\bibliographystyle{halpha}
\addcontentsline{toc}{section}{Bibliography}
\bibliography{welding}

\bigskip

\filbreak
\begingroup
\small
\parindent=0pt

\bigskip
\vtop{
\hsize=5.3in
Department of Mathematics\\
Massachusetts Institute of Technology\\
Cambridge, MA, USA } \endgroup \filbreak \end{document}